\RequirePackage{fix-cm}
\documentclass{svjour3_edit}                     
\usepackage[textsize=tiny]{todonotes}

\smartqed  
\usepackage{graphicx}
\usepackage{amsmath}
\usepackage{amsfonts}
\usepackage{amssymb}
\usepackage{bbm}
\usepackage{mathrsfs}
\usepackage{float}
\usepackage{array}
\usepackage{arydshln}
\usepackage[top=3cm, bottom=3cm, left=3.5cm, right=4.5cm]{geometry}
\newcommand{\one}{\mathbbm{1}}
\newcommand{\R}{\mathbb{R}}

\newcommand{\N}{\mathbb{N}}
\newcommand{\cc}{\mathscr{C}}
\newcommand{\Ls}{\mathcal{L}}

\newcommand{\Mcal}{\mathcal{M}}

\newcommand{\id}{\mathop{\textup{Id}}}

\newcommand{\veps}{\varepsilon}
\newcommand{\Qast}{Q^\ast\hspace{-2pt}}

\newcommand{\inner}[1]{\langle #1\rangle}

\newcommand{\overbar}[1]{\mkern 1.0mu \overline{\mkern-1.0mu#1\mkern-1.0mu}\mkern 1.0mu}

\newcommand{\Xint}[1]{\mathchoice
{\XXint\displaystyle\textstyle{#1}}%
{\XXint\textstyle\scriptstyle{#1}}%
{\XXint\scriptstyle\scriptscriptstyle{#1}}%
{\XXint\scriptscriptstyle\scriptscriptstyle{#1}}%
\!\int}
\newcommand{\XXint}[3]{{\setbox0=\hbox{$#1{#2#3}{\int}$}
\vcenter{\hbox{$#2#3$}}\kern-.5\wd0}}

\newtheorem{algorithm}{Algorithm}

\begin{document}

\title{Convergence and regularization for monotonicity-based shape reconstruction in electrical impedance tomography
}

\titlerunning{Convergence and regularization for monotonicity-based EIT shape reconstruction}        

\author{Henrik Garde         \and
        Stratos Staboulis 
}


\institute{H. Garde \and S. Staboulis \at
              Department of Applied Mathematics and Computer Science\\
              Technical University of Denmark\\ 
              2800 Kgs. Lyngby, Denmark  \\
              \email{ssta@dtu.dk}           
           \and
           H. Garde \\
           \email{hgar@dtu.dk}
}


\maketitle

\begin{abstract}
The inverse problem of electrical impedance tomography is severely ill-posed, meaning that, only limited information about the conductivity can in practice be recovered from boundary measurements of electric current and voltage. Recently it was shown that a simple monotonicity property of the related Neumann-to-Dirichlet map can be used to characterize shapes of inhomogeneities in a known background conductivity. In this paper we formulate a monotonicity-based shape reconstruction scheme that applies to approximative measurement models, and regularizes against noise and modelling error. We demonstrate that for admissible choices of regularization parameters the inhomogeneities are detected, and under reasonable assumptions, asymptotically exactly characterized. Moreover, we rigorously associate this result with the complete electrode model, and describe how a computationally cheap monotonicity-based reconstruction algorithm can be implemented. Numerical reconstructions from both simulated and real-life measurement data are presented.
\end{abstract}

\section{Introduction}

In {\em electrical impedance tomography} (EIT), the aim is to extract information about the internal properties of a physical object by external measurements of electric current and voltage. In practice, through a set of surface electrodes, currents of prescribed magnitudes are conducted into the object and the voltages needed for maintaining the currents are recorded. The obtained current-voltage data are used for imaging the internal electrical conductivity distribution of the object. Examples of EIT applications include, among others, monitoring patient lung function, control of industrial processes, non-destructive testing of materials, and locating mineral deposits \cite{Borcea2002a,Cheney1999,Hanke2003,Uhlmann2009}.

In EIT, electricity inside the domain is modelled by the {\em conductivity equation}
\begin{equation}\label{eq:intro-eq}
\nabla \cdot (\gamma \nabla u) = 0, \quad {\rm in} \ \, \Omega
\end{equation}
where $\Omega \subset \R^n$ describes the spatial dimensions of the object, $\gamma = \gamma(x)$ the conductivity (or admittivity) distribution, and $u = u(x)$ the potential of the electric field. The ideal data obtainable by current-voltage measurements are characterized by the {\em Neumann-to-Dirichlet operator} $\Lambda(\gamma)$ which relates a boundary current density to the corresponding boundary potential through a Neumann-problem for \eqref{eq:intro-eq}. 

The {\em inverse conductivity problem} is stated as {``given $\Lambda(\gamma)$, determine $\gamma$''}. The problem has been extensively studied during the past decades \cite{Calder'on1980,Nachman1996,Uhlmann2009} --- the known solvability conditions depend on the spatial dimension. In the plane, a general bounded real-valued $\gamma$ with a positive lower bound is uniquely determined by $\Lambda(\gamma)$ as long as $\Omega$ is simply connected \cite{Astala2006a}. In the three-dimensional space, more regularity of $\Omega$ and $\gamma$ is in general required \cite{Sylvester1987}. Despite the unique solvability, the mapping $\Lambda(\gamma) \mapsto \gamma$ has points of discontinuity in the $L^\infty$-norm \cite{Borcea2002a}. Therefore, the problem of reconstructing a general $\gamma$ from $\Lambda(\gamma)$ is ill-posed in the sense of Hadamard. 

A practical version of the inverse conductivity problem is {``given a noisy and discrete approximation of $\Lambda(\gamma)$, reconstruct information about $\gamma$''}. Arguably, the most flexible framework for computational EIT reconstruction is provided by iterative output least-squares type methods \cite{Cheney1990a,Heikkinen02,Lechleiter2006,Vauhkonen1998a} which work with realistic measurement models and allow incorporation of prior information into the model. However, in many applications, full-scale imaging may not be the top priority --- especially if it is computationally expensive due to high dimensionality of the computational domain. Instead, one may be interested in locating conductivity inhomogeneities in a known and/or uninteresting background medium. A variety of iterative \cite{Chung2005,Hyvonen2010b} and direct inhomogeneity detection methods have been introduced and elaborately studied. Two of the most prominent direct EIT methods relevant to this work are {\em the factorization method} \cite{Bruhl2001,Bruhl2000,Lechleiter2008a} and {\em the enclosure method} \cite{Ikehata1999a,Ikehata2000c}, which rely on repetitive but computationally very cheap testing of criteria that theoretically characterize shapes of conductivity inhomogeneities. 

In this paper we study the {\em monotonicity method}, the idea of which originates in \cite{Tamburrino2006,Tamburrino2002}. This direct method detects conductivity inhomogeneities by utilizing the fact that the forward map $\gamma \mapsto \Lambda(\gamma)$ is monotonically decreasing. To describe the method's main idea, suppose for simplicity that $\Omega$ has a smooth boundary and that the target conductivity is of form $\gamma = 1 + \chi_D$ where $D \subseteq \Omega$ is an unknown open set we would like to reconstruct. The monotonicity property implies that if a ball $B \subseteq D$ then $\Lambda(1 + \chi_B)-\Lambda(\gamma) \geq 0$ in the sense of semidefiniteness. Therefore, the collection $\Mcal$ of all open balls satisfying the latter criterion forms a cover for $D$. Recently, it was shown that, if $D$ has no holes, the converse holds: $\cup \Mcal$ coincides with $D$. Moreover, the result remains valid even if $\Lambda(1 + \chi_B)$ is replaced with the affine approximation $\Lambda(1) + \tfrac{1}{2}\Lambda'(1)\chi_B$ (note the factor $\tfrac{1}{2}$) regardless of the possibly large linearization error \cite{Harrach13}. The affine formulation is numerically tempting as the Fr\'echet derivative can be calculated ahead of time and applied by evaluating computationally cheap matrix-vector products. 

A key property of the monotonicity method is that it can be trivially formulated using the realistic {\em complete electrode model} (CEM) --- just by replacing $\Lambda(\gamma)$ by its CEM counterpart. The feasibility of the CEM-based monotonicity method was first considered in \cite{Harrach15}. By suitably relaxing the monotonicity test, it was shown that a non-trivial upper bound for $D$ can be obtained from inaccurate CEM data. In this paper, we consider the converse, that is, whether a sequence of suitably regularized reconstructions $\Mcal_\alpha$ in some sense converges to a limit as the discretization error and measurement noise level tend to zero. We study this question by extending the ideas of \cite{Harrach13} where {\em $B$-dependent} regularization of the monotonicity method is discussed.

Our main theoretical result (Theorem~\ref{thm:h-mono}) states that, as a suitable sequence of regularization parameters tends to zero, the noisy and discrete semidefiniteness test converges {\em uniformly} in $B$ to the continuum counterpart. As a corollary, we obtain the set-theoretic convergence of the regularized reconstruction $\Mcal_\alpha$ to the idealistic $\Mcal$. To be more precise, we formulate the above results in terms of a certain class of admissible approximate models for $\Lambda(\gamma)$. However, we also rigorously show how the results apply to the CEM by constructing approximate sequences for a Fr\'echet derivative of $\gamma\mapsto \Lambda(\gamma)$ of arbitrary order (Theorems~\ref{thm:cem-cm}--\ref{thm:cem-cm-lin} and Remark~\ref{rem:cem-cm-lin}). The approximation technique generalizes the ideas of \cite{Hyvonen09} which studies the CEM in the limit as the number of electrodes grows towards infinity. All in all, this paper is thematically comparable to \cite{Lechleiter2008a} where a rigorous asymptotic connection between the factorization method and the CEM is established.  

Before describing the structure of this paper, let us make some general comments on the extensions and limitations of the considered reconstruction method. Compared to other direct inhomogeneity detection techniques, the monotonicity method has the advantage that the inhomogeneity characterization results naturally extend to the {\em indefinite case} \cite{Harrach13} (both jumps and drops off the background conductivity). Although fast and easy to implement, a drawback of the standard implementation is that it requires the knowledge of the background conductivity. Fortunately, as in the context of the factorization method \cite{Harrach2009}, this problem can be partly avoided if A/C measurements with different frequencies are available. The appropriate analysis is carried out in \cite{Harrach2015} where the monotonicity method is generalized to complex-valued admittivities.

The contents of this article are organised as follows. In section~\ref{sec:eit} the idealized continuum model (CM) of EIT is rigorously defined, and its fundamental monotonicity properties are revised. In section~\ref{sec:reg-mono} the regularized monotonicity method is formulated and the main result (Theorem~\ref{thm:h-mono}) of the paper is proven. Section~\ref{sec:cem} introduces the CEM and demonstrates how it can be used to construct approximations to the CM. Two alternative algorithmic implementations of the monotonicity method are described in section~\ref{sec:alg-impl}. Moreover, a theoretical justification for the algoritmic use of the CEM is given. Finally, in section~\ref{sec:numerical}, the CEM-based reconstruction algorithms are tested numerically in two and three spatial dimensions for simulated measurements, and in two dimensions for real-life (cylindrically symmetric) water tank measurement data. 

\section{Electrical impedance tomography based on monotonicity} \label{sec:eit}

In this section we formulate the continuum model (CM) and revise its fundamental monotonicity property that motivates the monotonicity method. For simplicity, we only consider static EIT where all electric quantities are modelled by real-numbers. Generalizations of the method for A/C measurements, that is, complex-valued quantities, can be found in \cite{Harrach2015}. 

\subsection{Continuum model} 

Let $\Omega \subset \R^n$, $n = 2$ or $3$, be a bounded domain with Lipschitz-regular boundary $\partial\Omega$. Consider the elliptic boundary value problem
\begin{equation}
\label{eq:cm}
\begin{array}{ll}
\nabla\cdot(\gamma\nabla u) = 0, & \text{ in } \Omega, \\[4pt]
\nu\cdot\gamma\nabla u = f,  & \text{ on }\partial\Omega,
\end{array}
\end{equation}
where $\nu$ denotes the outward-pointing unit normal of $\partial\Omega$ and the real-valued coefficient function $\gamma = \gamma(x)$ belongs to 
\begin{equation*}
L_+^\infty(\Omega) = \{w\in L^\infty(\Omega)\, \colon {\rm ess}\inf w > 0\}.
\end{equation*}
By the Lax--Milgram theorem, for a given Neumann-boundary datum
\begin{equation}\label{eq:H-curr}
f \in L^2_\diamond({\partial\Omega}) = \{ w\in L^2({\partial\Omega}) \, \colon \langle w,\one\rangle = 0 \},
\end{equation}
problem \eqref{eq:cm} has a unique weak solution
\begin{equation}\label{eq:H-pot}
u\in H_\diamond^1(\Omega) = \{w \in H^1(\Omega)\, \colon  \langle w|_{\partial\Omega}, \one \rangle = 0\}.
\end{equation}
Above and from here on $\langle\cdot ,\cdot\rangle$ denotes the $L^2(\partial\Omega)$-inner product, $w|_{\partial\Omega}$ stands for the trace of $w$ on $\partial\Omega$, and $\one \equiv 1$ on $\partial\Omega$. 
In electrostatics, $u$ models the electric potential in $\Omega$ induced by the electrical conductivity distribution $\gamma$ and the input current density $f$. The extra conditions in the function space definitions \eqref{eq:H-curr} and \eqref{eq:H-pot} correspond to current conservation law and choice of ground level of potential, respectively.

The idealistic infinite precision data related to electric current-voltage boundary measurements are characterized by the {\em Neumann-to-Dirichlet (ND) map}
\begin{equation*}
\Lambda(\gamma) : L^2_\diamond({\partial\Omega})\to L^2_\diamond({\partial\Omega}), \qquad f \mapsto u|_{{\partial\Omega}},
\end{equation*}
where $u$ is the solution to \eqref{eq:cm}. 
Note that by density, knowledge of $\Lambda(\gamma)$ is tantamount to knowing every boundary current-voltage density pair 
\[
(f,u|_{\partial\Omega})\in H_\diamond^{-1/2}(\partial\Omega)\times H_\diamond^{1/2}(\partial\Omega)
\] 
connected via a more general formulation of \eqref{eq:cm} for less regular input current densities. The operator $\Lambda(\gamma)$ is linear and bounded, that is, it belongs to the space $\Ls(L_\diamond^2(\partial\Omega))$ of bounded linear operators from $L_\diamond^2(\partial\Omega)$ to itself. Furthermore, it is straightforward to show that $\Lambda(\gamma)$ is self-adjoint (see e.g.\ \cite{Borcea2002a}). On the other hand, the mapping 
\[
\gamma \mapsto \Lambda(\gamma) \colon L_+^\infty(\Omega) \to \Ls(L_\diamond^2(\partial\Omega))
\] 
is non-linear since clearly $\Lambda(c\gamma) = \Lambda(\gamma)/c$ for any constant $c > 0$. Despite being non-linear, the mapping is still very regular as it is analytic in $\gamma \in L_+^\infty(\Omega)$ (cf. Appendix~\ref{append:B}). 

Although the CM is feasible in proving interesting solvability and uniqueness results, it should be emphasized that the model is not very accurate in predicting real-life measurements. This is because, in practice, point-wise boundary current densities are out of reach and realistic electrodes cause a shunting effect which the CM does not directly account for \cite{Cheng1989,Somersalo1992}. Be that as it may, the operator $\Lambda(\gamma)\in \Ls(L_\diamond^2(\partial\Omega))$ as well as its linearization is compact due to the compact embedding $H^{1/2}(\partial\Omega)\subset\subset L^2(\partial\Omega)$. Therefore, it seems natural that they can be approximated accurately by finite-dimensional electrode model-based matrices. Before returning to this question in section \ref{sec:cem}, we recapitulate the ideas behind the monotonicity method.

\subsection{Monotonicity-based characterization of inclusions}

The following principle forms the basis of the monotonicity method in EIT \cite{Harrach10,Harrach13,Tamburrino2002}. 
\begin{proposition}\label{prop:mono}
For two arbitrary conductivities $\gamma,\tilde{\gamma} \in L_+^\infty(\Omega)$ it holds
%
\begin{equation}\label{eq:mono}
\int_\Omega \frac{\tilde{\gamma}}{\gamma}(\gamma - \tilde{\gamma})|\nabla \tilde{u}|^2 dx \leq \langle (\Lambda(\tilde{\gamma}) - \Lambda(\gamma))f,f \rangle \leq \int_\Omega (\gamma - \tilde{\gamma})|\nabla \tilde{u}|^2 dx
\end{equation}
where $f\in L_\diamond^2(\partial\Omega)$ is arbitrary and $\tilde{u} \in H^1_\diamond(\Omega)$ solves 
\[
\nabla \cdot (\tilde{\gamma}\nabla \tilde{u}) = 0 \ \ {\rm in} \ \Omega, \qquad \nu\cdot\tilde{\gamma}\nabla \tilde{u} = f \ \ {\rm on} \ \partial\Omega.
\]
\end{proposition}
In what follows, we focus on detecting (definite) conductivity inhomogeneities, or {\em inclusions}, lying in a known background. For the ease of presentation, we define the notion of an inclusion. 
\begin{definition}\label{def:inc}
Consider a conductivity distribution of the form $\gamma = \gamma_0 + \kappa\chi_D$, where $\gamma_0\in \cc^\infty(\overbar{\Omega})$, $\kappa \in L_+^\infty(\Omega)$, and $D$ is open with $\overbar{D} \subseteq \Omega$. The set $D$ is called a positive {\em inclusion} with respect to the background conductivity $\gamma_0$.   
\end{definition} 
In the rest of the paper, $\gamma$ is implicitly assumed to be as in Definition~\ref{def:inc} unless otherwise mentioned.

Proposition~\ref{prop:mono} gives rise to the following method for computing upper estimates for the inclusion $D$. 
Assume $0 < \beta \leq \kappa$ and let $B\subseteq \Omega$ be an arbitrary open ball. As a consequence of Proposition~\ref{prop:mono} we have
\begin{equation}\label{eq:cm-mono}
B \subseteq D \quad {\rm implies} \quad \Lambda(\gamma_0+\beta\chi_B) - \Lambda(\gamma) \geq 0
\end{equation}
in the sense of semidefiniteness. Here $\chi_B$ is the characteristic function of $B$. Denoting the collection of all admissible open balls by 
\begin{equation}\label{eq:R-cal}
\Mcal = \{B\subseteq\Omega \ {\rm open \ ball} \, \colon \Lambda(\gamma_0+\beta\chi_B) - \Lambda(\gamma) \geq 0 \},
\end{equation} 
we get the upper estimate
\begin{equation}\label{eq:R}
D \subseteq \cup \Mcal. 
\end{equation}
Accurate numerical approximation of $\Mcal$ can be costly, as it typically requires computation of a large number of forward solutions to \eqref{eq:cm}. 

Faster formulation of the semidefiniteness tests in \eqref{eq:R-cal} can be derived by linearizing the operator $\gamma \mapsto \Lambda(\gamma)$ around $\gamma = \gamma_0$. In fact, this modification also yields an upper bound for $D$ analogous to \eqref{eq:R}. To see this, suppose that $\gamma$ is as in Definition~\ref{def:inc}. As a consequence of Proposition~\ref{prop:mono} and the Fr\'echet derivative (see~\eqref{eq:fre-cm} in Appendix \ref{append:B}), we obtain
\begin{equation*}
\langle (\Lambda(\gamma_0) + \beta\Lambda'(\gamma_0)\chi_B - \Lambda(\gamma))f,f \rangle \geq \int_\Omega \bigg( \frac{\gamma_0\kappa}{\gamma}\chi_D - \beta\chi_B \bigg)|\nabla u_0|^2 dx
\end{equation*}
for any $f\in L_\diamond^2(\partial\Omega)$ and $u_0 \in H_\diamond^1(\Omega)$ solving 
\[
\nabla \cdot (\gamma_0\nabla u_0) = 0 \ \ {\rm in} \ \Omega, \qquad \nu\cdot\gamma_0\nabla u_0 = f \ \ {\rm on} \ \partial\Omega.
\]
In particular (cf.~\eqref{eq:cm-mono}) we deduce
\begin{equation}\label{eq:cm-mono-lin-2}
B \subseteq D \quad {\rm implies} \quad \Lambda(\gamma_0) + \beta\Lambda'(\gamma_0)\chi_B - \Lambda(\gamma) \geq 0
\end{equation}
provided that $0 < \beta \leq \gamma_0\kappa / \gamma$. By defining
\begin{equation}\label{eq:Mcal-dot}
\Mcal' = \left\{B\subseteq\Omega \ {\rm open \, ball} \, \colon \Lambda(\gamma_0) + \beta\Lambda'(\gamma_0)\chi_B - \Lambda(\gamma) \geq 0 \right\},
\end{equation}
we have
\begin{equation}\label{eq:R-dot}
D \subseteq \cup \Mcal'
\end{equation}
which is analogous to \eqref{eq:R}. Computational approximation of the set $\Mcal '$ is the main idea behind the reconstruction algorithm studied in this paper. A particular advantage of this approach is that the Fr\'echet derivative $\Lambda'(\gamma_0)$ can be computed ahead of time since it only depends on the (known) background conductivity $\gamma_0$ and object $\Omega$ but not on $B$ or $\beta$. More theoretical plausibility for the monotonicity method is given by the following result which, in a sense, complements relations \eqref{eq:cm-mono} and \eqref{eq:cm-mono-lin-2} \cite{Harrach13}.   
\begin{proposition}\label{prop:D=R}
Let $\gamma\in L_+^\infty(\Omega)$ be as in Definition~\ref{def:inc}. Assume that the background conductivity $\gamma_0$ is piecewise analytic and that the domain $\Omega$ has a $\cc^\infty$-regular boundary. 
Then, for any constant $\beta > 0$,
\begin{equation}\label{eq:co-mono}
\Lambda(\gamma_0 + \beta\chi_B) - \Lambda(\gamma) \geq 0 \ \ \ {\it or} \ \ \Lambda(\gamma_0) + \beta\Lambda'(\gamma_0)\chi_B - \Lambda(\gamma) \geq 0
\end{equation}
implies 
\[
B\subseteq D^\bullet = \overbar{\Omega} \setminus \cup \{U\subseteq \mathbb{R}^n\setminus D \ {\it open \ and \ connected} \,\colon U\cap \partial\Omega \neq \emptyset\}. 
\] 
Note that $D^\bullet$ corresponds to the smallest closed set containing $D$ and having connected complement. Consequently, by \eqref{eq:R-cal} and \eqref{eq:Mcal-dot}, the conditions
\begin{align}
& D \subseteq  \cup \Mcal \subseteq D^\bullet \quad {\it if} \quad 0 < \beta \leq {\rm ess}\inf \kappa, \label{eq:beta-nonlin} \\ 
& D \subseteq \cup \Mcal' \subseteq D^\bullet \quad {\it if} \quad 0 < \beta \leq {\rm ess}\inf \Big( \frac{\gamma_0\kappa}{\gamma} \Big). \label{eq:beta-lin}
\end{align}
hold true. Note also that if $D$ has connected complement, then $D^\bullet = \overbar{D}$.
\end{proposition}
\begin{proof}
The claim follows directly from \cite[Theorem 4.1 and Theorem 4.3]{Harrach13} which are proven using the theory of {\em localized potentials} \cite{Gebauer2008b}. The intuition behind \eqref{eq:co-mono} can roughly be described as follows. If $B\not\subseteq D^\bullet$, then $\gamma_0 + \beta\chi_B \geq \gamma$ in $B \setminus D^\bullet$. Consequently, each of the semidefiniteness conditions in \eqref{eq:co-mono} can be contradicted by constructing sequences of potentials having simultaneously very large energy in $B\setminus D^\bullet$ and very small energy in $D^\bullet$.  \qed
\end{proof}
Relations \eqref{eq:beta-nonlin} and \eqref{eq:beta-lin} indicate that both $\Mcal$ and $\Mcal'$ can contain significant information about $D$. However, due to ill-posedness, the associated semidefiniteness tests can be expected to be sensitive with respect to measurement noise and modelling error. Aiming for stable numerical implementation we introduce a regularized variant of the semidefiniteness tests.

\section{Regularized monotonicity-based reconstruction} \label{sec:reg-mono}

In practice, the infinite precision measurement $\Lambda(\gamma)$ is out of reach, and moreover, only approximate numerical models are available for computational monotonicity tests. We model a 
collection of abstract approximate forward models by a family of compact self-adjoint operators $\{\Lambda_h(\gamma)\}_{h>0}$ such that 
\begin{equation}\label{eq:h-Lambda}
\|\Lambda(\gamma) - \Lambda_h(\gamma)\|_{\Ls(L^2_\diamond(\partial\Omega))} \leq \omega(h)\|\gamma\|_{L^\infty(\Omega)}, \quad \lim_{h\to 0}\omega(h) = 0
\end{equation}
for any $\gamma\in L_+^\infty(\Omega)$, where $\omega$ is independent of $\gamma$. In addition to systematic modelling error, real-life measurements are also corrupted by noise caused by imperfections of the measurement device. We assume the following additive noise model
\begin{equation*}
\Lambda_{h}^{\delta}(\gamma) = \Lambda_h(\gamma) + N^\delta,
\end{equation*} 
where the noise is modelled by a family of compact and self-adjoint operators
\begin{equation}\label{eq:noise}
N^\delta \colon L^2_\diamond(\partial\Omega) \to L^2_\diamond(\partial\Omega), \qquad \|N^\delta\|_{\Ls(L^2_\diamond(\partial\Omega))} \leq \delta.
\end{equation} 
Note that, by symmetrizing if necessary, the self-adjointness assumption of the error operator can be made without loss of generality.

To facilitate reading, we use the following abbreviations for certain operators and the related infimal eigenvalues. For a given open set $B\subseteq \Omega$ we denote 
\begin{equation}\label{eq:T}
\begin{array}{rl}
T(B) & = \Lambda(\gamma_0 + \beta\chi_B) - \Lambda(\gamma),\\[3pt]
T_h(B) & = \Lambda_h(\gamma_0 + \beta\chi_B) - \Lambda_h(\gamma),\\[3pt] 
T_h^\delta(B) & = \Lambda_h(\gamma_0 + \beta\chi_B) - \Lambda_h^\delta(\gamma).
\end{array}
\end{equation} 
As all operators in \eqref{eq:T} are self-adjoint, their spectra are contained in $\R$. Moreover, by Hilbert--Schmidt theorem \cite{Reed1990}, the eigenvalues of any infinite dimensional, compact and self-adjoint Hilbert space operator $S$ accumulate at zero implying 
\begin{equation*}
-\|S\| \leq \inf\sigma(S) \leq 0
\end{equation*}
where $\sigma(S)$ denotes the spectrum of $S$. 

To guarantee a meaningful noisy reconstruction, we introduce a regularized version of $\Mcal$ in \eqref{eq:R-cal} defined by  
\begin{equation}\label{eq:Mcal-hd}
\Mcal_\alpha(T_h^\delta) = \{B\subseteq\Omega \ {\rm open \ ball} \, \colon  T_h^\delta(B) + \alpha\id \geq 0 \}
\end{equation} 
where $\alpha \in \R$ is a regularization parameter. Next we investigate in which sense and under which conditions the set $\Mcal_\alpha(T_h^\delta)$ converges to $\Mcal_0(T)$, as the error parameters $h$ and $\delta$ tend to zero. To establish an asymptotic relationship, we resort to the following lemma which is a consequence of spectral continuity \cite{Kato1995}.

\begin{lemma}\label{lem:cont-spec}
Let $S$ and $T$ be bounded self-adjoint operators on a Hilbert space $H$. Then 
\begin{equation}\label{eq:infST}
|\inf\sigma(S) - \inf\sigma(T)| \leq \|S - T\|_{\Ls(H)}
\end{equation}
\end{lemma}
\begin{proof}
We begin by noting that \eqref{eq:infST} trivially holds if its left-hand side vanishes. Hence, by symmetry we may assume that $\inf\sigma(S) < \inf\sigma(T)$. There exists a sequence $\{\lambda_j\}_{j=1}^\infty \subseteq \sigma(S)$ such that
\[
\lim_{j\to\infty} \lambda_j = \inf\sigma(S)
\]
and $\inf \sigma(S) \leq \lambda_j < \inf \sigma(T)$ for all $j$. By the continuity of the spectrum \cite[Chapter V \S 4, Theorem 4.10]{Kato1995}, we have 
\begin{equation*}
\textup{dist}(\lambda_j,\sigma(T)) \leq \sup_{\lambda \in \sigma(S)} \textup{dist}(\lambda,\sigma(T)) \leq \|S- T\|_{\Ls(H)}.
\end{equation*}
Consequently 
\[
\|S-T\|_{\Ls(H)} \geq \textup{dist} (\lambda_j, \sigma(T)) = \inf_{\mu \in \sigma(T)} |\lambda_j - \mu| = \inf \sigma(T)  - \lambda_j
\]
and hence, by taking the limit, we deduce
\[
\inf\sigma(T) - \inf\sigma(S) = \lim_{j\to\infty} \textup{dist}(\lambda_j,\sigma(T)) \leq \|S - T\|_{\Ls(H)}
\]
which concludes the proof. \qed
\end{proof}
As a special consequence of Lemma~\ref{lem:cont-spec}, \eqref{eq:h-Lambda}, and \eqref{eq:noise} we obtain
\[
|\inf\sigma(T_h^\delta(B)) - \inf\sigma(T(B))| \leq \omega(h)\left( \|\gamma_0\|_{L^\infty(\Omega)} + \beta + \|\gamma\|_{L^\infty(\Omega)} \right) + \delta
\]
which implies that 
\begin{equation}\label{eq:B-uni}
\lim_{h,\delta \to 0} \inf\sigma(T_h^\delta(B)) = \inf\sigma(T(B))
\end{equation}
uniformly in $B\subseteq \Omega$. The following theorem shows that, with a suitable sequence of regularization parameters $\alpha\in \R$, the set $\Mcal_0(T)$ can be, in a sense, stably approximated by the sequence $\Mcal_\alpha(T_h^\delta)$.
\begin{theorem}\label{thm:h-mono} 
Suppose that the regularization parameter $\alpha = \alpha(h,\delta) \in \R$ satisfies
\begin{equation}\label{eq:reg-param}
\delta - \alpha(h,\delta) \leq \inf_{B\in\Mcal_0(T)}\inf\sigma(T_h(B)) \quad {\it and} \quad \lim_{h,\delta \to 0} \alpha(h,\delta) = 0.
\end{equation}
Then for any given $\lambda > 0$ there exists an $\veps_\lambda > 0$ such that
\begin{equation}\label{eq:R-lim}
\Mcal_0(T) \subseteq \Mcal_{\alpha(h,\delta)}(T_h^\delta) \subseteq \Mcal_\lambda(T)
\end{equation}
for all $h,\delta \in (0,\veps_\lambda]$. 
\end{theorem}

\begin{proof}
Let us start by noting that according to \eqref{eq:B-uni} and Lemma~\ref{lem:A-1} of Appendix~\ref{append:A}, the conditions \eqref{eq:reg-param} are not contradictory, and thus, the set of admissible sequences of regularization parameters is not empty. To prove the left-hand set inclusion in \eqref{eq:R-lim}, let $B\in \Mcal_0(T)$ be an arbitrary open ball. First we note that, by a basic property of the infimal eigenvalue and \eqref{eq:reg-param}, we have 
\begin{equation}\label{eq:h-mono-1}
T_h(B) \geq \inf\sigma(T_h(B)) \id \geq (\delta - \alpha(h,\delta))\id.
\end{equation}
From \eqref{eq:noise} we obtain $\delta\id \geq N^\delta$, which together with \eqref{eq:h-mono-1} yields
\[
T_h^\delta(B) + \alpha(h,\delta)\id = T_h(B) - N^\delta + \alpha(h,\delta)\id \geq \delta \id - N^\delta \geq 0.  
\]
This shows that $B\in \Mcal_{\alpha(h,\delta)}(T_h^\delta)$ for all $h,\delta > 0$. In particular, the left-hand set inclusion in \eqref{eq:R-lim} holds.  

To prove the right-hand set inclusion in \eqref{eq:R-lim}, let $\lambda > 0$ be arbitrary. According to \eqref{eq:reg-param} and the uniform convergence \eqref{eq:B-uni}, there exists an $\veps_\lambda > 0$ such that 
\begin{equation*}
\alpha(h,\delta) \leq \frac{\lambda}{2},\quad \inf\sigma(T_h^\delta(B)) \leq \inf\sigma(T(B)) + \frac{\lambda}{2}
\end{equation*}
for all $h,\delta \in (0,\veps_\lambda]$ and open balls $B\subseteq \Omega$. Moreover, by definition \eqref{eq:Mcal-hd}, any open ball $B \in \Mcal_{\alpha(h,\delta)}(T_h^\delta)$ satisfies  
\begin{equation}\label{eq:lambda-ineq}
0 \leq \inf\sigma(T_h^\delta(B)) + \alpha(h,\delta) \leq \inf \sigma(T(B))+\lambda
\end{equation}
which implies that $\Mcal_{\alpha(h,\delta)}(T_h^\delta) \subseteq \Mcal_\lambda(T)$ for all $h,\delta \in (0,\veps_\lambda]$.
\qed
\end{proof}
Compared to $\Mcal_{\alpha}(T_h^\delta)$, the family $\Mcal_\lambda(T)$ has the favourable monotone decreasing property
\begin{equation}\label{eq:mono-decr}
0 < \lambda \leq \mu \ \ {\rm implies} \ \ \Mcal_\lambda(T) \subseteq \Mcal_\mu(T).
\end{equation}
This yields the set-theoretic limit (defined as in e.g.\ \cite{Resnick2014})
\begin{equation}\label{eq:set-limit}
\lim_{\lambda\to 0}\Mcal_\lambda(T) = \bigcap_{\lambda > 0}\Mcal_\lambda(T) = \Mcal_0(T)
\end{equation}
where the left and right equalities follow from the monotone decreasing property \eqref{eq:mono-decr} and the definition of $\Mcal_\lambda(T)$, respectively. As a consequence of \eqref{eq:set-limit} and Theorem~\ref{thm:h-mono}, we obtain a corresponding limit for $\Mcal_\alpha(T_h^\delta)$.
\begin{corollary}\label{cor:h-mono}
Let the regularization parameter be as in \eqref{eq:reg-param}. Then we have the set-theoretic limit 
\begin{equation}\label{eq:hd-set-limit}
\lim_{h,\delta \to 0} \Mcal_{\alpha(h,\delta)}(T_h^\delta) = \Mcal_0(T) = \bigcap_{h,\delta > 0} \Mcal_{\alpha(h,\delta)}(T_h^\delta).
\end{equation}
\end{corollary}
\begin{proof}
The fact that the set-theoretic limit exists and coincides with $\Mcal_0(T)$ is a direct consequence of the ``squeeze principle'' enforced by \eqref{eq:R-lim} and \eqref{eq:set-limit}. Note that the considered family of sets is not necessarily decreasing; hence, the right-hand equality in \eqref{eq:hd-set-limit} has to be proven separately. In the proof of Theorem~\ref{thm:h-mono} it is shown that $\Mcal_0(T) \subseteq \Mcal_{\alpha(h,\delta)}(T_h^\delta)$ for all $h,\delta > 0$, implying the ``$\supseteq$''-direction. The ``$\subseteq$''-direction follows from \eqref{eq:lambda-ineq} by letting $h,\delta \to 0$ and recalling \eqref{eq:B-uni}.  
\qed
\end{proof}
\begin{remark}\label{rem:h-mono}
Suppose that the inclusion $D$ has a connected complement, and that $|D| = |\overbar{D}|$. Assume further that the approximate operator family satisfies the monotonicity property: $\gamma \geq \tilde{\gamma}$ implies $\Lambda_h(\tilde{\gamma}) - \Lambda_h(\gamma) \geq 0$ for all $\gamma, \tilde{\gamma} \in L_+^\infty(\Omega)$. By Proposition~\ref{prop:D=R} we have $\cup\Mcal_0(T) \subseteq \overbar{D}$. Hence, by monotonicity and choosing $\beta$ as in \eqref{eq:beta-nonlin} yields
\[
T_h(B) \geq T_h(\cup\Mcal_0(T)) \geq T_h(\overbar{D}) = T_h(D) \geq 0
\]  
for any $B\in \Mcal_0(T)$. Consequently, the choice $\alpha(h,\delta) = \delta$ is in this case sufficient for obtaining \eqref{eq:hd-set-limit}.
\end{remark}

\begin{remark}
The results indicated by Theorem \ref{thm:h-mono}, Corollary \ref{cor:h-mono}, and Remark \ref{rem:h-mono} straightforwardly adapt to the linearized version \eqref{eq:R-dot}. To complete the proofs it is sufficient, in addition to \eqref{eq:h-Lambda}, to assume that we have
\begin{equation}\label{eq:h-linLambda}
\|\Lambda'(\gamma)\eta - \Lambda'_h(\gamma)\eta\|_{\Ls(L^2_\diamond(\partial\Omega))} \leq \omega(h)\|\gamma\|_{L^\infty(\Omega)}\|\eta\|_{L^\infty(\Omega)}
\end{equation}
for all $\eta\in L^\infty(\Omega)$. Furthermore, the definitions \eqref{eq:T} have to be modified accordingly.

\end{remark}

While the above asymptotic results are formulated between the somewhat abstract ball collections, a relevant question is whether an analogue of \eqref{eq:hd-set-limit} holds for their unions which --- according to Proposition~\ref{prop:D=R} --- can be directly compared to conductivity inclusions. In particular, such a result would imply 
$$
D\subseteq \lim_{h,\delta \to 0}\cup\Mcal_{\alpha(h,\delta)}(T_h^\delta) \subseteq D^\bullet
$$ 
with the convergence in the sense of the Lebesgue measure. With straightforward modifications, the proof of Theorem~\ref{thm:h-mono} can be adapted to the case where the ball collections are replaced with the unions 
$$ 
\cup \Mcal_\lambda(T) \ \ {\rm and} \ \  \cup\Mcal_\alpha(T_h^\delta).
$$ 
Moreover, the former family is monotonously decreasing (cf.~\eqref{eq:mono-decr}) allowing a set-theoretic limit and convergence in the Lebesgue measure. However, currently we are not aware of a non-trivial relationship between (cf~\eqref{eq:set-limit}) 
$$
\bigcap_{\lambda > 0} \cup\Mcal_\lambda(T) \ \ {\rm and} \ \ \cup\Mcal_0(T).
$$
To demonstrate the present difficulty, pick a point $x$ that belongs to $\cup\Mcal_\lambda(T)$ for all $\lambda > 0$. Then for each $\lambda > 0$ there exists an open ball $B_\lambda \in \Mcal_\lambda(T)$ which contains $x$. However, without further specifications, this does not imply the existence of a {\em fixed} ball $B$ that would contain $x$ and lie in $\Mcal_\lambda(T)$ for all $\lambda > 0$.

\section{Approximating infinite-precision data by realistic models} \label{sec:cem}

The most widely used model for real-life EIT measurements is the CEM which is capable of predicting measurement data up to instrument precision \cite{Cheney1999,Cheng1989,Somersalo1992}. In this section, we define the CEM and review it's fundamental monotonicity property analogous to \eqref{eq:mono}. Subsequently, we point out that --- under some reasonable regularity assumptions --- the CEM can be used to construct sequences of approximate operators of type \eqref{eq:h-Lambda} and \eqref{eq:h-linLambda} for both $\Lambda(\gamma)$ and $\Lambda'(\gamma)$, respectively. 

\subsection{Complete electrode model}

The CEM is formally defined by the boundary value problem
\begin{equation}
\label{eq:cem}
\begin{array}{ll}
\displaystyle{\nabla\cdot(\gamma\nabla v) = 0, }\quad  &{\rm in}\;\; \Omega, \\[5pt] 
{\displaystyle{\nu\cdot\gamma\nabla v} = 0, }\quad &{\rm on}\;\;{\partial\Omega}\setminus \bigcup_{j=1}^k\overbar{E_j},\\[5pt] 
{\displaystyle v+z_j{\nu\cdot\gamma\nabla v} = V_j, } \quad &{\rm on}\;\; E_j, \\ 
{\displaystyle \int_{E_j}\nu\cdot\gamma\nabla v\, dS = I_j},\phantom{\enskip} \quad & j=1,2,\ldots k, \\[2pt]
\end{array}
\end{equation} 
where the open, connected and mutually disjoint sets $E_j \subseteq {\partial\Omega}$ model the electrode patches attached to the outer boundary of the object. For a given conductivity $\gamma\in L_+^\infty(\Omega)$, contact impedance $z \in \R_+^k$, and net input current pattern 
$$
I\in \R_\diamond^k = \Bigg\{W\in\R^k: \sum_{j=1}^k W_j = 0\Bigg\},
$$ 
a unique weak solution pair
\[
(v,V) \in H^1(\Omega)\oplus \R_\diamond^k
\] 
to problem \eqref{eq:cem} exists \cite{Somersalo1992,Hyvonen2004}. The electrode measurement data related to $\gamma$ (and $z$) are fully characterized by the bounded, linear, and self-adjoint \cite{Somersalo1992} {\em measurement map}
\begin{equation*}
R(\gamma) : \R_\diamond^k \to \R_\diamond^k, \qquad I \mapsto V.
\end{equation*}
Here $I$ and $V$ model the net input currents and the net voltages perceived by the electrodes, respectively. Note that $R(\gamma)$ depends also on the contact impedance $z$. However, this dependence will be omitted except in Proposition~\ref{prop:cem-mono} where its values may vary.

Compared to the CM \eqref{eq:cm}, the CEM has a more complicated mathematical formulation \eqref{eq:cem}. Partly due to this, CEM-based EIT is not theoretically well-understood. In particular, virtually all known uniqueness and stability results --- and the related reconstruction techniques --- for the inverse conductivity problem are formulated in terms of the CM. However, by the following proposition, the monotonicity principle \eqref{eq:mono} extends quite naturally to the CEM framework \cite{Harrach15}. 
\begin{proposition}\label{prop:cem-mono}
Consider two arbitrary conductivities $\gamma, \tilde{\gamma} \in L_+^\infty(\Omega)$ and contact impedances $z,\tilde{z} \in \R_+^k$. Denote
\begin{align*}
	c_0 &= \int_\Omega \frac{\tilde{\gamma}}{\gamma}(\gamma - \tilde{\gamma})|\nabla \tilde{v}|^2 dx + \sum_{j=1}^k \int_{E_j}\frac{z_j}{\tilde{z}_j}\Big(\frac{1}{z_j} - \frac{1}{\tilde{z}_j}\Big)|\tilde{v} - \tilde{V}_j|^2 dS, \\
	c_1 &= \int_\Omega (\gamma - \tilde{\gamma})|\nabla \tilde{v}|^2 dx + \sum_{j=1}^k \int_{E_j}\Big(\frac{1}{z_j} - \frac{1}{\tilde{z}_j}\Big)|\tilde{v} - \tilde{V}_j|^2 dS,
\end{align*}
where $(\tilde{v},\tilde{V})$ is the solution pair to \eqref{eq:cem} corresponding to the input current $I\in\mathbb{R}_\diamond^k$, conductivity $\tilde{\gamma}$, and contact impedance $\tilde{z}\in \R_+^k$. Then it holds 
\begin{equation*}
 c_0 \leq I^{\rm T} (R(\tilde{\gamma},\tilde{z}) - R(\gamma,z))I \leq 
c_1.
\end{equation*}
\end{proposition}
Proposition~\ref{prop:cem-mono} implies, in particular, that counterparts of \eqref{eq:cm-mono} and \eqref{eq:cm-mono-lin-2} hold in the CEM framework.

\subsection{Approximating $\Lambda(\gamma)$ using the measurement map $R(\gamma)$}

The relationship between $\Lambda(\gamma)\in \Ls(L_\diamond^2(\partial\Omega))$ and $R(\gamma) \in \Ls(\R_\diamond^k)$
has been studied e.g.~in \cite{Lechleiter2008a,Hyvonen09}. In what follows, we review the approach in \cite{Hyvonen09} because it is simple to formulate, and it gives a good error estimate. The method relies on the concept of open and mutually disjoint {\em extended electrodes} $\{E_j^+\}_{j=1}^k$ that are assumed (together with the actual electrodes) to satisfy  
\begin{equation}\label{eq:ext-el}
E_j \subseteq E_j^+ \subseteq \partial\Omega, \quad \bigcup_{j=1}^k \overbar{E_{\hspace{-1pt}j}^+} = \partial\Omega, \quad \min_{j=1,\ldots,k} \frac{|E_j|}{|E_j^+|} \geq c
\end{equation} 
where $c>0$ is a constant independent of the set of electrodes in question. The mappings $\Lambda(\gamma)$ and $R(\gamma)$ can be compared with the help of the adjoint pair $Q\colon \R^k \to L^2(\partial\Omega)$ and $\Qast\colon L^2(\partial\Omega) \to \R^k$ defined via
\begin{equation*}
QW = \sum_{j=1}^k W_j \chi_j^+, \quad (\Qast f)_j = \int_{E_j^+} f dS \label{eq:QL-def}
\end{equation*}
where \smash{$\chi_j^+$} denotes the characteristic function of \smash{$E_j^+$}. A thorough motivation of the notion of extended electrodes can be found in \cite{Hyvonen09}.

As indicated by the following theorem, under reasonable assumptions on the regularity of \smash{$E_j^+$} and $\partial\Omega$, the infinite precision EIT data can be approximated using the CEM with an error directly proportional to the maximal extended electrode width. 
\begin{theorem}\label{thm:cem-cm} Let $P \colon L^2(\partial\Omega) \to \R^k$ and $L \colon L^2(\partial\Omega) \to L^2_\diamond(\partial\Omega)$ be the orthogonal projectors\footnote{The former in the sense in which $\R^k$ is identified with the subspace of $L^2(\partial\Omega)$ consisting of piecewise constant functions of form $a = \sum_j a_j\chi_j$ where $\chi_j$ is the characteristic function of $E_j$.} defined via
\begin{equation}\label{eq:P-proj}
(P f)_j := \Xint-_{E_j} f\, dS, \quad L f := f - \Xint-_{\partial\Omega}f dS,
\end{equation}
and denote the maximal extended electrode diameter by \smash{$h = \max_j {\rm diam}(E_j^+)$}. Assume that the extended electrodes are regular enough\footnote{Estimate \eqref{eq:I-LP} could be enforced by assuming \smash{$E_j^+$} are regular enough to allow a ``scaling argument'' resulting in a Poincar\'e inequality with a constant bounded by \smash{$C\,{\rm diam}(E_j^+)$} \cite{Hyvonen09,Lieb2003,Bebendorf2003}. } so that the Poincar\'e inequality-type estimate
\begin{equation}\label{eq:I-LP}
\|(\id - QP)f\|_{L^2(\partial\Omega)} \leq Ch \|\nabla f\|_{L^2(\partial\Omega)}, 
\end{equation}
holds for all $f\in H^1(\partial\Omega)$ with a constant $C > 0$ independent of $h$. Then we have  
\begin{equation}\label{eq:cem-cm}
\|\Lambda(\gamma) - L Q(R(\gamma)-Z)\Qast\,\|_{\Ls(L_\diamond^2({\partial\Omega}))} \leq C h\|\gamma\|_{L^\infty(\Omega)}
\end{equation}
for any $\gamma$ as in Definition~\ref{def:inc} with a constant $C > 0$ independent of $\gamma$ and $h$. Here $Z\in\R^{k\times k}$ is the diagonal matrix with the non-zero entries $Z_{jj} := z_j/|E_j|$.
\end{theorem}
\begin{proof}
The claim follows from \cite[Proof of Theorem 4.1]{Hyvonen09} with the following minor modifications. First of all, note that the projector $\Qast$ and the matrix $Z$ are defined slightly differently here because \cite{Hyvonen09} formulates \eqref{eq:cem} in terms of mean electrode currents instead of the total currents used in this paper. Moreover, for the proof, it is essential that $QZQ^\ast$ is uniformly bounded with respect to $h$. This is ensured by assuming the rightmost condition in \eqref{eq:ext-el}. Finally, in \cite{Hyvonen09} the forward models are formulated in the quotient space framework, that is, choice of ground level potential is circumvented by letting $\Lambda(\gamma)$ and $R(\gamma)$ take values in spaces $L^2(\partial\Omega)/\R$ and $\R^k/\R$, respectively. With the above modifications to $Q$ and $Z$, it is straightforward to see that the whole difference operator in \eqref{eq:cem-cm} equals its counterpart in \cite{Hyvonen09} up to an isometry between $L^2(\partial\Omega)/\R $ and $L^2_\diamond(\partial\Omega)$.
\end{proof} 


Above it was pointed out that, under suitable assumptions, the CEM provides a linearly convergent approximation to the CM as the number of electrodes grows in a suitable manner. Next we point out that the same holds true also for the linearized versions of the models. 

\begin{theorem}\label{thm:cem-cm-lin}
Let $\eta\in L^\infty(\Omega)$ be compactly supported in $\Omega$. Under the same assumptions as in Theorem \ref{thm:cem-cm}, there holds 
\begin{equation*}
\|\Lambda'(\gamma)\eta - LQ(R'(\gamma)\eta)\Qast\,\|_{\Ls(L_\diamond^2({\partial\Omega}))} \leq C h \|\gamma\|_{L^\infty(\Omega)} \|\eta\|_{L^\infty(\Omega)},
\end{equation*}
where $h$ is the maximal diameter of $E_j^+$ and $C>0$ is a constant independent of $h$, $\gamma$, and $\eta$.
\end{theorem}
\begin{proof}
Let $u$ be the solution to \eqref{eq:cm} with boundary condition $f\in L_\diamond^2({\partial\Omega})$ and let $(v,V)$ solve \eqref{eq:cem} with the input current $I= \Qast f$. Denote by $u'$ the solution to the sensitivity problem \eqref{eq:fre-cm} and by $(v',V')$ the solution to \eqref{eq:fre-cem}. In the following, the operator $L$ is also considered an operator from $H^1(\Omega)$ to $H_\diamond^1(\Omega)$ in the sense of subtracting the mean of the trace; cf. \eqref{eq:P-proj}. 

By the boundary conditions of \eqref{eq:fre-cem}, we have 
\begin{equation}\label{eq:LV}
QV' = \sum_{j=1}^k \chi^+_j\Xint-_{E_j}(v' + z_j \nu\cdot\gamma\nabla v')  dS = \sum_{j=1}^k \chi^+_j \Xint-_{E_j} v' dS =   QPv'.
\end{equation}
By \eqref{eq:LV}, the triangle inequality, and by applying the trace theorem and the fact $\|L\|_{\Ls(L^2(\partial\Omega))} = 1$, we get
\begin{align}
\|u' - LQV'\|_{L^2(\partial\Omega)} &\leq \|u'-Lv'\|_{L^2(\partial\Omega)} + \|L(\id - QP)v'\|_{L^2(\partial\Omega)} \notag\\[4pt]
&\leq  C\|u'-Lv'\|_{H^1(\Omega)} + \|(\id - QP)v'\|_{L^2(\partial\Omega)}.\label{eq:cem-cm-lin-1}
\end{align}
Next we estimate the first term on the right side of \eqref{eq:cem-cm-lin-1}. By coercivity of the bilinear form associated with \eqref{eq:cm}, we have
\begin{equation}\label{eq:bundle-1}
\|u' - Lv'\|_{H^1(\Omega)}^2 \leq C \int_\Omega \gamma |\nabla (u'-v')|^2 \, dx.
\end{equation}
Furthermore, using the variational formulations of the sensitivity problems \eqref{eq:fre-cm} and \eqref{eq:fre-cem}, we obtain
\begin{align}\label{eq:bundle-2}
\int_\Omega \gamma |\nabla (u'-v')|^2 \, dx & = \int_\Omega \eta \nabla (Lv-u) \cdot \nabla (u'-Lv') \, dx \nonumber\\ 
 &\phantom{{}={}}+ \sum_{j=1}^k \int_{E_j}\frac{1}{z}(v'-V'_j)(u'-Lv') \, dS.
\end{align}
Note that, in the three instances above, $v$ has been replaced by $Lv$ and $v'$ by $Lv'$, respectively. While in the interior term the replacements are trivially justified, the boundary term follows from the fact that --- by the boundary conditions \eqref{eq:fre-cem} --- we have
\[
\int_{E_j}\frac{1}{z}(v' - V_j')c \, dS = c\int_{E_j} \nu\cdot\gamma\nabla v' dS = 0 
\]
for all scalars $c$. Inserting \eqref{eq:bundle-2} into \eqref{eq:bundle-1}, and applying Cauchy--Schwartz inequality and trace theorem to the right-hand quantity results in 
\begin{equation}\label{eq:cem-cm-lin-2}
\|u' - Lv'\|_{H^1(\Omega)} \leq C \Big(\|\eta\|_{L^\infty(\Omega)}\|u-Lv\|_{H^1(\Omega)} + \sum_{j=1}^k\|v' - V_j'\|_{L^2(E_j)}\Big).
\end{equation}
The first term on the right side of \eqref{eq:cem-cm-lin-2} can be estimated suitably by using 
\begin{equation}
\|u-Lv\|_{H^1(\Omega)} \leq Ch\|\gamma\|_{L^\infty(\Omega)}\|f\|_{L^2(\partial\Omega)} \label{eq:uv-est}
\end{equation} 
which follows from the proof of Theorem \ref{thm:cem-cm}. Since by \eqref{eq:LV} we have
\[
\|v'-V_j'\|_{L^2(E_j)} = \|v'-QV'\|_{L^2(E_j)} \leq \|(\id - QP)v'\|_{L^2(\partial\Omega)},
\]
the second term on the right side of \eqref{eq:cem-cm-lin-2} can be handled by working out an appropriate upper bound for $\|(\id - QP)v'\|_{L^2(\partial\Omega)}$.

At this point, to guarantee sufficient regularity for $v'$, we need the assumptions that $\gamma$ is as in Definition \ref{def:inc} and $\eta$ is compactly supported in $\Omega$ (cf.~\eqref{eq:fre-cem}). Under these hypotheses, an analogous argument to \cite[Theorem 2.1]{Hyvonen2014} --- ultimately based on elliptic regularity theory \cite{Lions1972} and the boundary conditions \eqref{eq:fre-cem} --- implies that there exists a relatively open connected set $U\subseteq \overbar{\Omega}$ such that 
\begin{equation}\label{eq:v'32}
\partial\Omega \subseteq \partial U, \quad \overbar{U}\cap ({\rm supp}\, \eta \cup \overbar{D}) = \emptyset, \quad \|\nu\cdot \gamma\nabla v'\|_{L^2(\partial U)} \leq C \|v'\|_{H^1(\Omega)}.
\end{equation} 
Consequently, applying \eqref{eq:I-LP}, the definition of the $H^1$-norm, the trace theorem for quotient spaces \cite{Hyvonen2004}, continuous dependence on Neumann-data \cite[\S 1 Theorem 7.4]{Lions1972}, the rightmost estimate of \eqref{eq:v'32}, and continuity properties of \eqref{eq:fre-cem} leads us to the estimate 
\begin{align} 
\|(\id - QP)v'\|_{L^2(\partial\Omega)} & \leq Ch \|\nabla v'\|_{L^2(\partial\Omega)} \notag\\[8pt]
& \leq Ch \inf_{c\in \R}\|v' + c\|_{H^1(\partial\Omega)} \notag\\[4pt]
& \leq Ch \inf_{c\in \R}\|v' + c\|_{H^{3/2}(U)} \notag \\[4pt] 
& \leq C h \|\nu\cdot\gamma\nabla v'\|_{L^2(\partial U)} \notag\\[8pt]
& \leq C h \|v'\|_{H^1(\Omega)} \notag \\[8pt]
& \leq Ch\|\gamma\|_{L^\infty(\Omega)}\|\eta\|_{L^\infty(\Omega)}\|f\|_{L^2(\partial\Omega)}. \label{eq:I-LP-est}
\end{align}
In conclusion, using \eqref{eq:cem-cm-lin-1}--\eqref{eq:I-LP-est} we have shown that
\[
\begin{split}
&\|(\Lambda(\gamma)'\eta)f - LQ(R'(\gamma)\eta)Q^*f\|_{L^2(\partial\Omega)} \\[4pt]
&\qquad\qquad\qquad\qquad\leq C \|u'-Lv'\|_{H^1(\Omega)} +  \|(\id - QP)v'\|_{L^2(\partial\Omega)}\\[4pt] 
&\qquad\qquad\qquad\qquad \leq C\left( \|\eta\|_{L^\infty(\Omega)}\|u-Lv\|_{H^1(\Omega)} +  \|(\id - QP)v'\|_{L^2(\partial\Omega)} \right)\\[4pt]
&\qquad\qquad\qquad\qquad \leq Ch\|\gamma\|_{L^\infty(\Omega)}\|\eta\|_{L^\infty(\Omega)}\|f\|_{L^2(\partial\Omega)}
\end{split}
\]
and the proof is concluded. \qed
\end{proof}
\begin{remark}\label{rem:cem-cm-lin}
Analogous argumentation can be used to generalize Theorem~\ref{thm:cem-cm-lin} for higher order Fr\'echet derivatives; cf.~Proposition \ref{prop:fre} in Appendix B.
\end{remark}

According to Theorems \ref{thm:cem-cm}--\ref{thm:cem-cm-lin}, the operators of the form
\begin{equation*}
	\Lambda_h(\gamma) = LQ(R(\gamma)-Z)Q^*, \quad \Lambda_h'(\gamma) = LQR'(\gamma)\Qast
\end{equation*}
satisfy \eqref{eq:h-Lambda} and \eqref{eq:h-linLambda}, respectively.
Given a noisy measurement map $R^\delta(\gamma)$, the semidefiniteness test applied to the operators
\[
	LQ(R(\gamma_0+\beta\chi_B)-R^\delta(\gamma))Q^*, \quad LQ(R(\gamma_0) + \beta R'(\gamma_0)\chi_B + R^\delta(\gamma))\Qast
\]
where $B\subseteq \Omega$ is an open ball, satisfies the asymptotic characterization property of Theorem~\ref{thm:h-mono}. Note also that the noise does not get amplified since the operator norms of $Q$ (and $Q^*$) and $L$ are obviously bounded by $\max_{j=1,\dots,k} |E^+_j|^{1/2}$ and $1$, respectively. 

The following lemma shows that equivalent semidefiniteness tests can be carried out without constructing the projection operators, but instead, just using the electrode measurement map. 
\begin{lemma} \label{prop:posdef-equiv}
Let $A : \R_\diamond^k \to \R_\diamond^k$ be arbitrary. Then for any $f,g\in L_\diamond^2(\partial\Omega)$ we have
\begin{equation}\label{eq:inner-equiv}
	\inner{LQAQ^*f,g} = AQ^*f \cdot Q^*g.
\end{equation}
As a consequence, the conditions
\begin{equation*}
	LQAQ^* \geq 0 \ \ {\it and} \ \  A \geq 0
\end{equation*}
are equivalent.
\end{lemma}
\begin{proof}
	Clearly, $L$ is self-adjoint when considered an operator from $L^2(\partial\Omega)$ to itself. Moreover, since $L|_{L_\diamond^2(\partial\Omega)} = \id $, we have
\[
	\inner{LQAQ^*f,g} = AQ^*f\cdot Q^*Lg = AQ^*f\cdot Q^*g. 
\]
for any $f,g\in L^2_\diamond(\partial\Omega)$, that is, \eqref{eq:inner-equiv}. In particular, it follows that $A \geq 0$ implies $LQAQ^* \geq 0$. 
	
To show the converse, assume that $LQAQ^* \geq 0$. As a consequence of \eqref{eq:inner-equiv}, it is sufficient to show that for any $I\in \R_\diamond^k$ there exists an $f_I\in L^2_\diamond(\partial\Omega)$ such that $I = Q^*f_I$. For example, the function defined by
\[
	f_I = \sum_{j=1}^k \frac{\chi_j^+}{|E_j^+|}I_j,
\]
has the desired property. \qed
\end{proof}
Let us finish the section with a short recap. We have demonstrated that, under the assumption \eqref{eq:I-LP}, the CEM can be used to construct a sequence of monotonicity reconstructions that converge to the infinite precision counterpart in the sense of Theorem~\ref{thm:h-mono}. Although the assumptions of Theorem~\ref{thm:h-mono} do not require the approximative operators to satisfy a monotonicity principle, Proposition~\ref{prop:cem-mono} shows that the CEM measurement map nevertheless has this favourable property (cf. Remark~\ref{rem:h-mono}). We also emphasize that, by Proposition~\ref{prop:posdef-equiv}, the operators $Q$ and $L$ do not have a practical role in terms of implementing the algorithm --- they are only needed for forming a theoretical connection between the CEM and CM-based semidefiniteness tests. 

\section{Algorithmic implementation} \label{sec:alg-impl}

In this section we formulate two slightly different algorithms for reconstructing strictly positive (or negative) conductivity inclusions. The modifications that enable reconstructing indefinite inclusions are left for future studies. To highlight the fact that in practice the ``true'' conductivity is unknown, we denote the noisy measurement data by $R^\delta$ omitting the dependence on $\gamma$. Moreover, let us remind that we make the arguably restrictive assumption that the background conductivity $\gamma_0$ is known a priori.

\begin{algorithm}[Monotonicity method for the CEM]\label{alg:1}
\quad 
\begin{itemize}
\item[0.] Fix the collection of balls $\mathcal{B}$, choose the regularization parameter $\alpha > 0$, and compute $R(\gamma_0)$ and $R'(\gamma_0)$.\\[-4pt]
\item[1.] For all $B = B(x,r)\in \mathcal{B}$, construct the indicator 
\begin{equation}\label{eq:ind-1}
{\rm Ind}(x) := \max\{0,\min\sigma(R(\gamma_0) + \beta R'(\gamma_0)\chi_B - R^\delta + \alpha\id)\}
\end{equation}

\item[2.] Return ${\rm Ind}$.

\end{itemize}
\end{algorithm}
A logical choice for the probing scalar is
\begin{equation*}
\beta = {\rm ess}\inf \Big( \frac{\gamma_0\kappa}{\gamma} \Big).
\end{equation*} 
In practice a scalar satisfying $\beta \leq {\rm ess}\inf ( \gamma_0\kappa/\gamma )$ can be constructed from knowledge of lower and upper bounds on $\kappa$.
Moreover, we propose constructing the regularization parameter in the form
\begin{equation}\label{eq:reg-param-num}
\alpha = \alpha(h,\delta) = -\mu \min \sigma(R(\gamma_0) - R^\delta), \quad \mu \approx 1.
\end{equation}
The value of $\mu$ is to be tuned --- according to our numerical results typically very close to $1$. The idea behind this choice comes from the following heuristic argument. According to the adaptation of Theorem~\ref{thm:h-mono} to the linearized monotonicity method, the following holds. If the regularization parameter satisfies
\[
\alpha(h,\delta) \geq \delta - \inf_{B\in\Mcal_0'} \inf \sigma(\Lambda_h(\gamma_0) + \beta \Lambda_h'(\gamma_0)\chi_B - \Lambda_h(\gamma)),
\]
then the reconstruction forms an upper bound for $\Mcal_0'$. However, recalling that $\beta\Lambda_h'(\gamma_0)\chi_B \leq 0$ \cite[Corollary 3.4]{Harrach13} and $N^\delta \leq \delta$, we see that  
\begin{align}\label{eq:reg-ineq}
-\inf\sigma(\Lambda_h(\gamma_0) - \Lambda_h^\delta(\gamma)) 
&\leq \delta - \inf\sigma(\Lambda_h(\gamma_0) - \Lambda_h(\gamma)) \nonumber\\[4pt]
& \leq \delta - \inf_{B\in\Mcal_0'}\inf\sigma(\Lambda_h(\gamma_0) + \beta\Lambda_h'(\gamma_0)\chi_B - \Lambda_h(\gamma)).
\end{align}
The idea is to barely reverse the inequality \eqref{eq:reg-ineq} by multiplying with a suitable $\mu$ such that the resulting regularization parameter $\alpha$ satisfies \eqref{eq:reg-param} while being small enough. Rigorous association of \eqref{eq:reg-param-num} with the reversal of \eqref{eq:reg-ineq} would of course require taking the operators $L$, $Q$ and $Q^*$ into the consideration. Further elaboration of the choice of the regularization parameter are left for future studies.

According to our numerical tests, the output of Algorithm \ref{alg:1} is very sensitive to the choice of the regularization parameter: Choosing slightly too large $\alpha$ tends to result in crudely overestimated supports, whereas even a bit too small $\alpha$ yields vanishing reconstructions. An idea for increasing the flexibility of the method is to fix sufficiently large $\alpha$, and probe monotonicity using various values of $\beta$. Intuitively this seems natural, as increasing $\beta$ will tighten the monotonicity test and in turn (hopefully) sharpen up the reconstruction. This gives rise to the following algorithm.
%

\begin{algorithm}[A flexible monotonicity method for the CEM]\label{alg:2}
\quad 
\begin{itemize}
\item[0.] Fix the collection of balls $\mathcal{B}$ and choose the regularization parameter $\alpha > 0$, and compute $R(\gamma_0)$ and $R'(\gamma_0)$. Fix a collection $\{\beta_j\}_{j=1}^m \subset \R$ of increasing values for probing semidefiniteness, and set $j = 1$.\\[-4pt]
\item[1.] For all $B = B(x,r)\in \mathcal{B}$, construct  
\begin{equation}\label{eq:ind-2}
{\rm Ind}_j(x) := 
\begin{cases}
1, &{\it if} \  \min\sigma(R(\gamma_0) + \beta_j R'(\gamma_0)\chi_B - R^\delta + \alpha\id) \geq 0,\\
0, &{\it otherwise}.
\end{cases}
\end{equation}
\item[2.] Set $j \to j + 1 $, redefine $\mathcal{B} \to \mathcal{B} \setminus \{B(x,r) \in \mathcal{B} \,\colon {\rm Ind}_j(x) = 0\} $, and go back to step 1.\\[-4pt]
\item[3.] If $\mathcal{B} = \emptyset$, return the indicator
\[
{\rm Ind} := \sum_{j=1}^m{\rm Ind}_j.
\]
\end{itemize}
\end{algorithm}
Note that the idea of ``Step 2'' in Algorithm \ref{alg:2} is to speed up the computations by discarding excess test balls --- this is justified by the monotonicity property of the CEM measurement map. Typically, if $\alpha > 0$ is not very large and the inclusions are small, the running times of the Algorithms 1--2 are essentially the same.

\section{Numerical experiments} \label{sec:numerical}

We proceed with four numerical examples which test the implementation of Algorithms~\ref{alg:1}--\ref{alg:2} from the following point of views. The idea of the first example is to test whether the reconstruction of a fixed (non-convex) inclusion sharpens up as the number of measurement electrodes increase. Here, for comparison, adaptations of Algorithms~\ref{alg:1}--\ref{alg:2} are applied also in the CM framework. In the second example, the algorithms are applied in a two-dimensional geometrical setting to synthetic data with and without additive artificial random noise. In the third example Algorithm~\ref{alg:2} is applied to real-life data measured on a cylindrically symmetric water tank phantom. Due to the symmetry, the reconstruction is carried out in two spatial dimensions. The last example is an application of Algorithm~\ref{alg:1} to synthetic exact data simulated in three spatial dimensions. 

In all of the two-dimensional experiments the object $\Omega$ is unit disk-shaped, and the electrodes are equispaced and of equal length, covering in total half of the boundary. In the three-dimensional experiment, the domain is a unit ball with spherical cap-shaped electrodes that are placed approximately equidistantly. Moreover, in all simulated examples, the contact impedance is given a constant value $z_j = 0.1$, $j = 1,2,\ldots,k$.

The numerical implementation of Algorithms~1--2 is based on the following linear algebra. Let $\{I^{(j)}\}_{j=1}^{k-1}$ be a basis of $\R_\diamond^k$ and denote 
$$
{\bf I} = [I^{(1)},I^{(2)},\ldots,I^{(k-1)}].
$$ 
Then a matrix representation of $R(\gamma_0) + \beta R'(\gamma_0)\chi_B - R(\gamma)$ in this basis is given by 
\begin{equation*}
{\bf A} = {\bf I}^\dagger ({\bf X + V}) \in \R^{(k-1)\times (k-1)},
\end{equation*}
where ${\bf I}^\dagger = ({\bf I}^{\rm T} {\bf I})^{-1}{\bf I}^{\rm T}$ is the Moore--Penrose pseudoinverse of ${\bf I}$,
$$
{\bf X} = [X^{(1)},X^{(2)},\ldots,X^{(k-1)}], \quad {\bf V} = [V^{(1)},V^{(2)},\ldots,V^{(k-1)}],
$$
with $X^{(j)} = (R(\gamma_0) + \beta R'(\gamma_0)\chi_B)I^{(j)}$ and $V^{(j)} = R(\gamma)I^{(j)}$.  Evaluation of the indicator functions in \eqref{eq:ind-1} and \eqref{eq:ind-2} is based on computing the smallest eigenvalue of 
\begin{equation*}
{\bf A}^{\hspace{-2pt}\delta} = {\bf I}^\dagger ({\bf X} + {\bf V}^\delta),
\end{equation*} 
where 
\begin{equation*}
{\bf V}^\delta = {\rm Sym}(\widetilde{\bf V}^\delta{\bf I}^\dagger){\bf I}
\end{equation*}
and $\widetilde{\bf V}^\delta$ is the noisy measurement data. Here ${\rm Sym}$ denotes the symmetric part, and it is applied to ensure that the underlying noisy data comprise a symmetric matrix. In addition, each column of ${\bf \widetilde{V}^\delta}$ and ${\bf V}^\delta$ is enforced to be in $\R_\diamond^k$. Noise is simulated by setting 
\begin{equation}\label{eq:noisemodel}
\widetilde{\bf V}^\delta = {\bf V} + {\bf N},
\end{equation} 
where each entry of ${\bf N}$ is given by $N_{ij} = V_i^{(j)} Y_{ij}$ and $Y_{ij}$ is a drawn from normal distribution of mean zero and standard deviation $5\cdot 10^{-3}$. In the simulated noisy examples the ratio 
\begin{equation*}
\|{\bf V - V^\delta}\|_F/\|{\bf V}\|_F
\end{equation*}
of Frobenius norms is called the ``relative error''. 

The measurement maps and their Fr\'echet derivatives are approximated by a standard finite element method \cite{Kaipio2000} using piecewise quadratic ($\mathbb{P}_2$) and piecewise affine ($\mathbb{P}_1$) elements, in two and three spatial dimensions, respectively. Moreover, the conductivity distribution is discretized by simplex-wise constant ($\mathbb{P}_0$) elements. The computational domain is a polygonal discretization of a unit disk/ball. In the two-dimensional examples, the meshes used in simulations and reconstructions consist of approximately $2.7\times 10^6$ nodes and $1.3\times 10^6$ triangles, and $7.5\times 10^5$ nodes and $3.7\times 10^5$ triangles, respectively. In the three-dimensional examples the simulations are performed on a mesh with $1.2\times 10^5$ nodes and $6.9\times 10^5$ tetrahedrons, while the reconstructions are computed using a mesh with $4.9\times 10^4$ nodes and $2.9\times 10^5$ tetrahedrons. With the above dimensions, an average reconstruction computation time using a laptop with two Intel Core 2 processors with CPU clock rate 2.4 Ghz was around five seconds.

As our FEM model is based on $\mathbb{P}_0$-discretization of the conductivity, the characteristic functions of the test sets $B$ are approximated by 
\begin{equation}\label{eq:chiB}
\chi_B \approx \sum_{K\subseteq B} \chi_K,
\end{equation}
where $K$ are simplices in the mesh. While carrying out the computations, some subtlety regarding to the choice of the test sets was observed. In particular, using very small balls yields artifactual reconstructions. This is not surprising since, for a fixed mesh, \eqref{eq:chiB} is a bad approximation when the radius of $B$ is small. To ease visualization of the results, we do not use disks/balls as computational test sets. In two-dimensions, the test sets are chosen from a regular hexagonal tiling of the plane. In the three-dimensional computations, we use voxels. We emphasize that convergence result analogous to Theorem~\ref{thm:h-mono} can be generalized to various types of measurable subsets of $\Omega$.

\begin{example}\label{ex:1}

In this numerical example we compute linearized monotonicity reconstructions with respect to different numbers of electrodes. For comparison, we use both CEM and discretized CM as forward models. No extra artificial noise is added to the synthesized data. 

The (noiseless) discretized CM is formulated as a truncated matrix approximation as follows. For a set of linearly independent boundary current densities $\{f^{(m)}\}_{m=1}^p\subseteq L_\diamond^2(\partial\Omega)$, we set
\begin{equation*}
A_{\ell m} = \int_{\Omega} (\gamma_0 - \beta\chi_B) \nabla u_0^{(\ell)}\cdot \nabla u_0^{(m)} dx - \inner{\Lambda(\gamma)f_\ell,f_m}
\end{equation*}
where $u_0^{(m)}$ solves \eqref{eq:cm} for the conductivity $\gamma_0$ and the current density $f^{(m)}$, $m=1,2,\ldots,p$. By \eqref{eq:cm}, \eqref{eq:fre-cm}, and Green's formula, the matrix ${\bf A} = \{A_{\ell m}\}_{\ell,m=1}^p$ is a discretization of $\Lambda(\gamma_0) + \beta\Lambda'(\gamma_0)\chi_B - \Lambda(\gamma)$.
\begin{figure}[h]
\begin{tabular}{lll}
\hspace{.9cm}{\scriptsize {\bf Target}}&
\hspace{.7cm}{\scriptsize {\bf Noiseless~1}}&
\hspace{.7cm}{\scriptsize {\bf Noiseless~2}}\\
\includegraphics[width=0.3\textwidth]{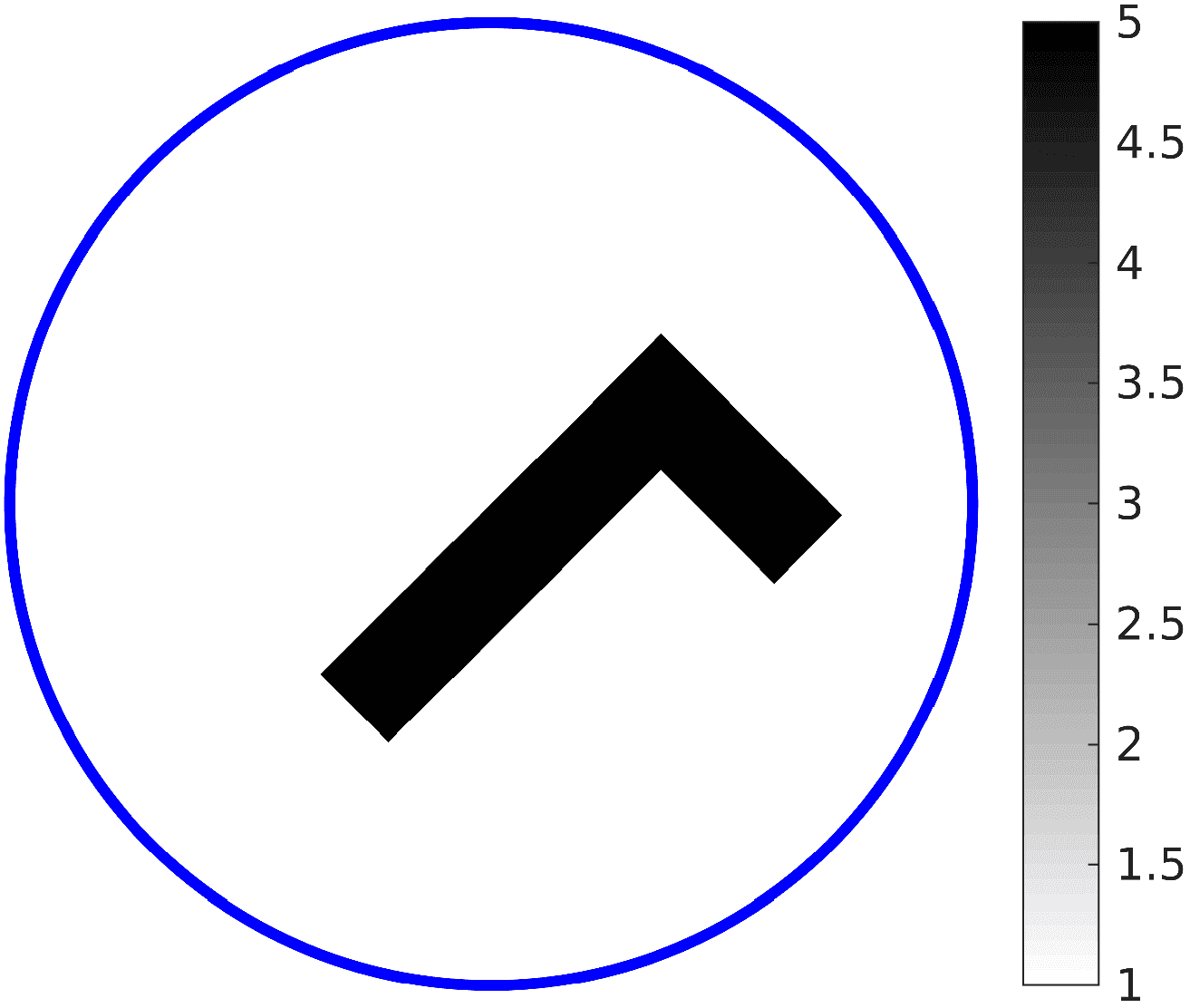}&
\includegraphics[width=0.32\textwidth]{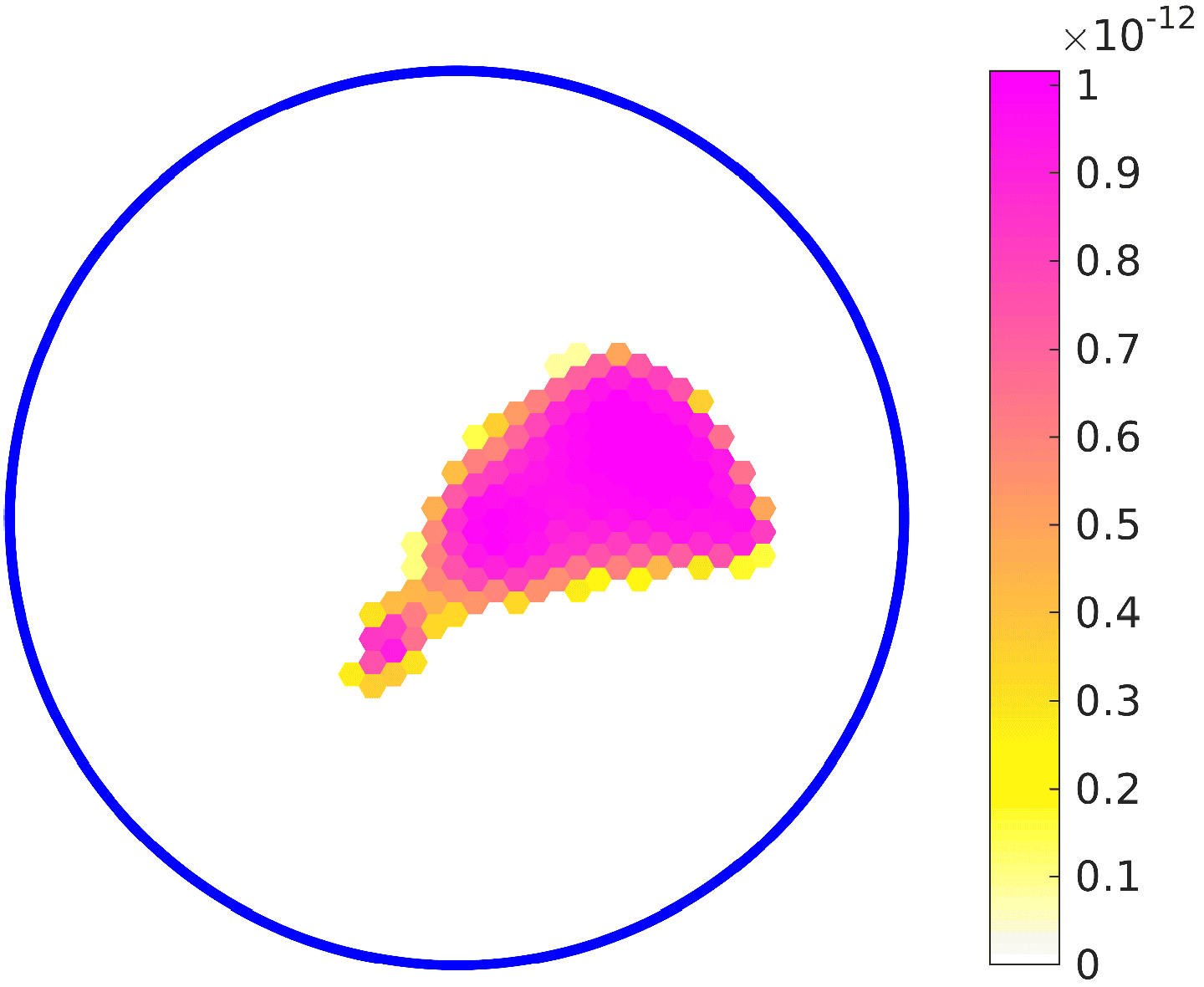}&
\includegraphics[width=0.31\textwidth]{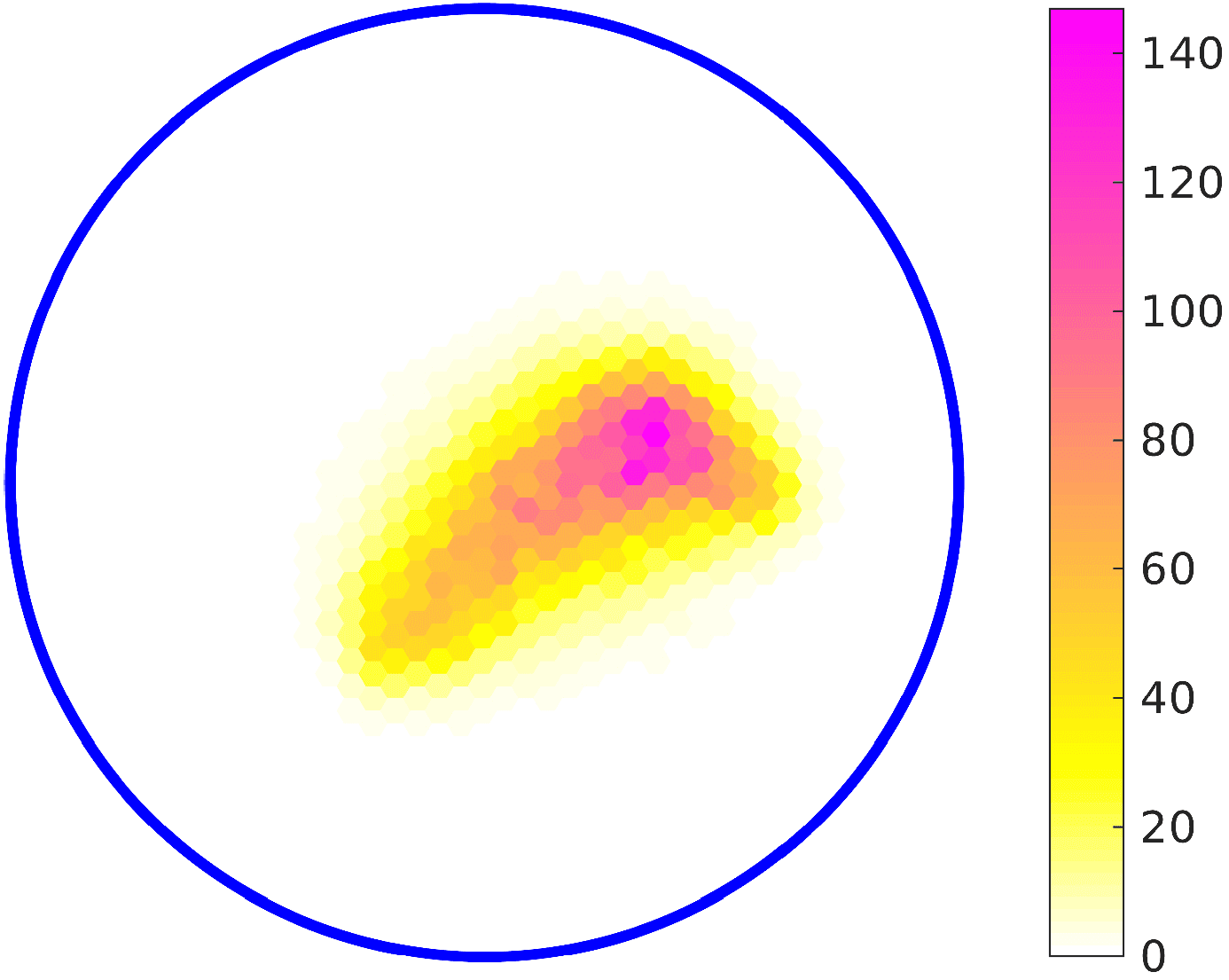}\\
\end{tabular}
\caption{\label{fig:cm-Lshape} Two-dimensional CM reconstructions from synthetic exact datum using $64$ linearly independent boundary current densities. Reconstructions ``{\bf Noiseless 1}'' and ``{\bf 2}'' are calculated using Algorithms 1 and 2, respectively. For more information on the FE mesh, see the beginning of section~\ref{sec:numerical}.}
\begin{center}
\def\arraystretch{1.3}
\begin{tabular}{ccc}
\hline
Parameter & {\bf Noiseless 1} (CM) & {\bf Noiseless 2} (CM)\\\hline
${\rm diam}(B)$ & $0.053$ & $0.053$ \\
$\beta$ & $0.8$ & $0.1 + 0.5\N$ \\
$\mu$ & $1.00003$ & $1.01$ \\\hline
\end{tabular}
\end{center}
{\bf Tbl. 1}\, \  Parameter values used in the computations; ${\rm diam}(B)$ is the diameter of the hexagons in the hexagonal reconstruction mesh \eqref{eq:chiB}, $\beta$ is the probing scalar(s) in the semidefiniteness test, and $\mu$ is the regularization parameter \eqref{eq:reg-param-num}.
\end{figure} 
\begin{figure}[h]
\begin{tabular}{m{1.2cm}m{2.14cm}m{2.14cm}m{2.14cm}m{2.14cm}}
&
\hspace{0.6cm}$k = 8$&
\hspace{0.53cm}$k = 16$&
\hspace{0.53cm}$k = 32$&
\hspace{0.53cm}$k = 64$\\
{\scriptsize {\bf Noiseless~1}}&
\includegraphics[width=0.20\textwidth]{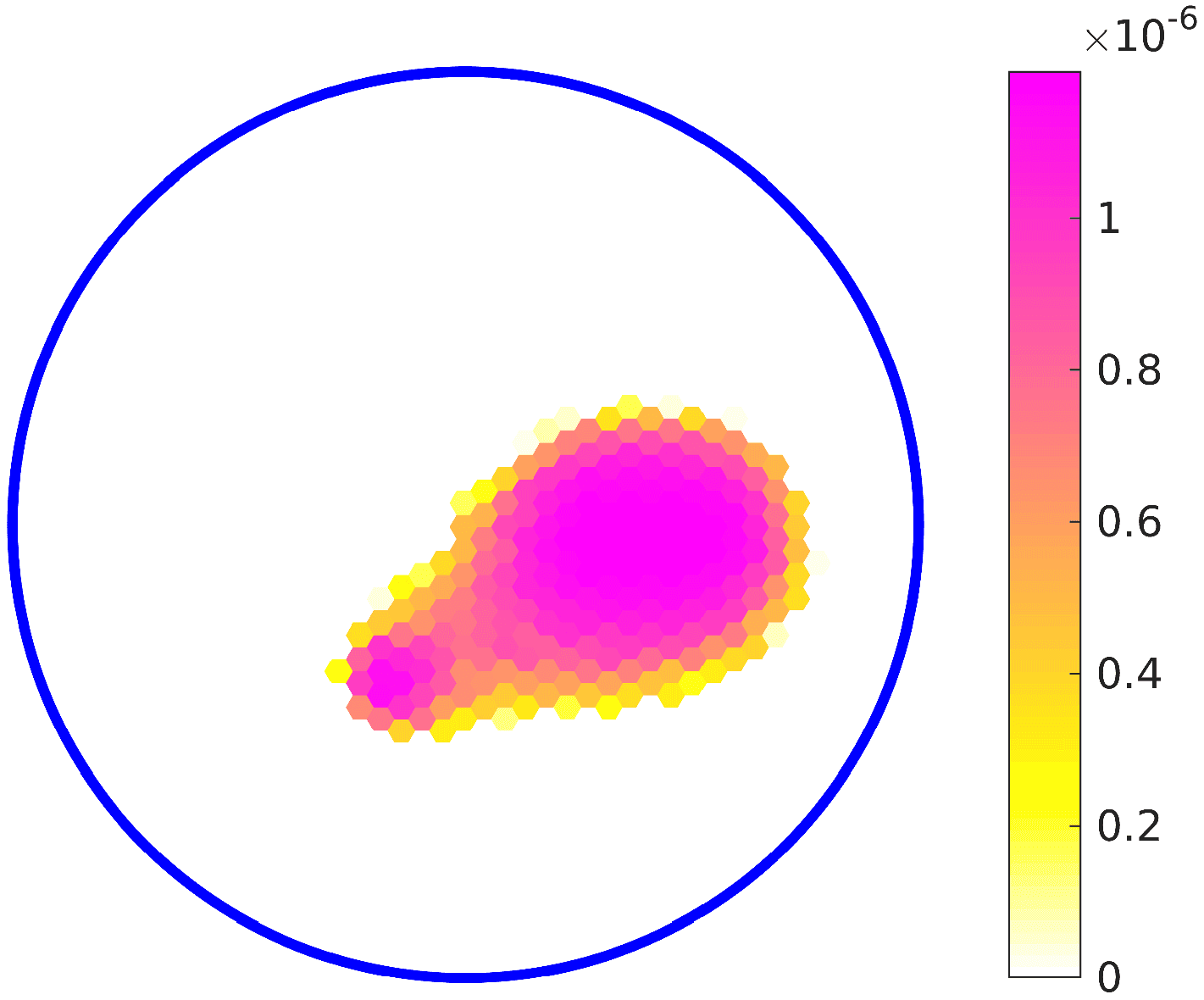}&
\includegraphics[width=0.20\textwidth]{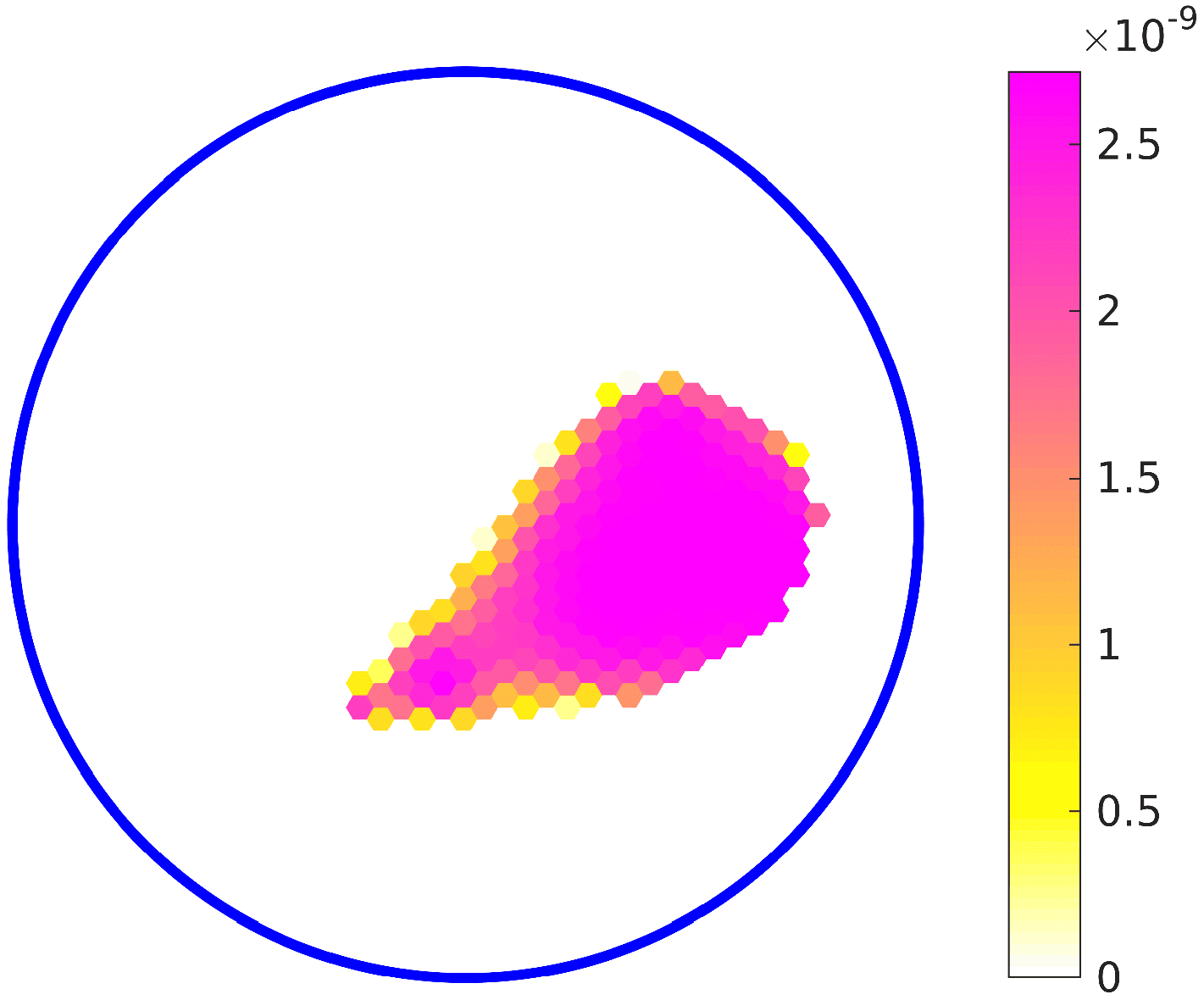}&
\includegraphics[width=0.20\textwidth]{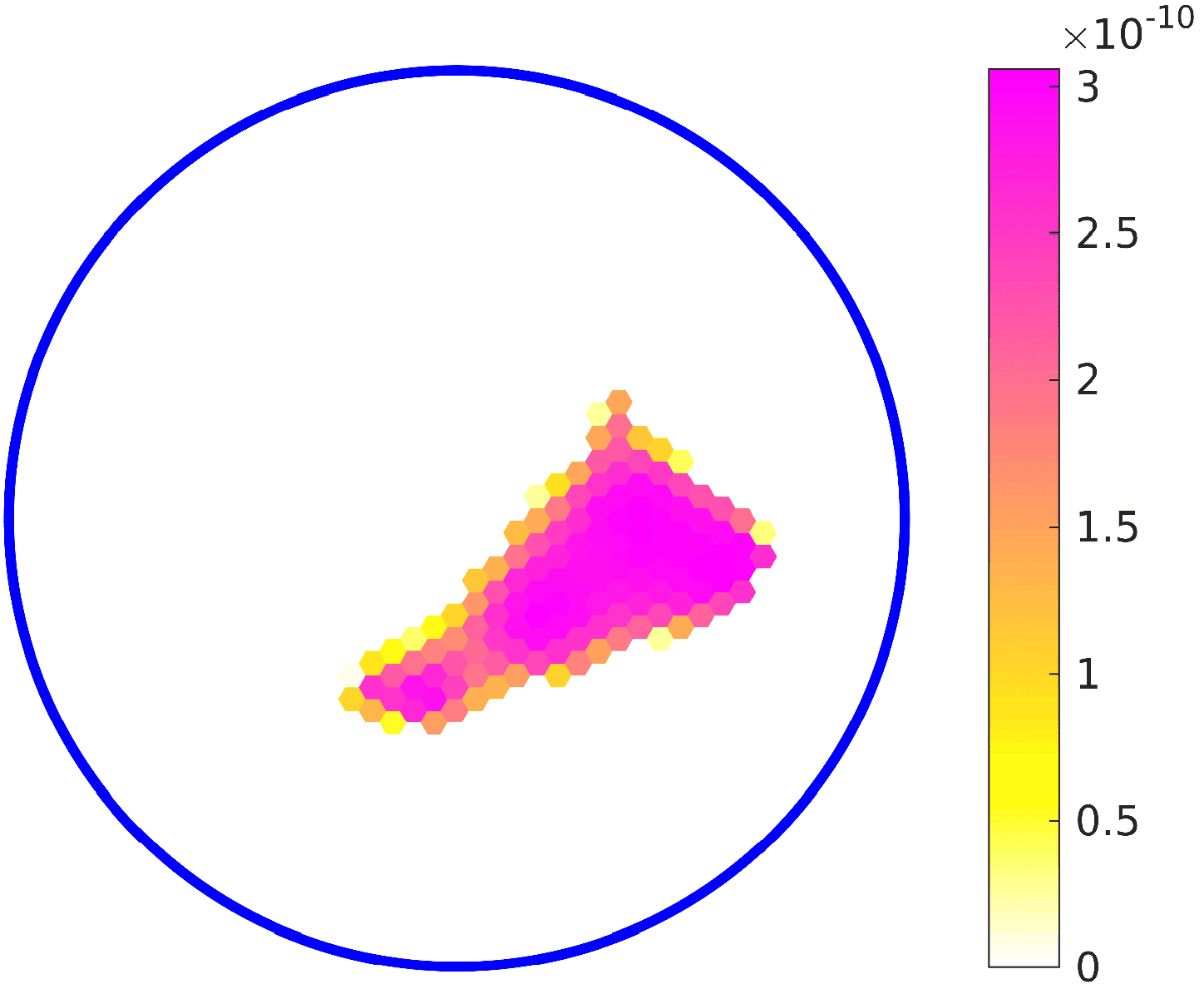}&
\includegraphics[width=0.20\textwidth]{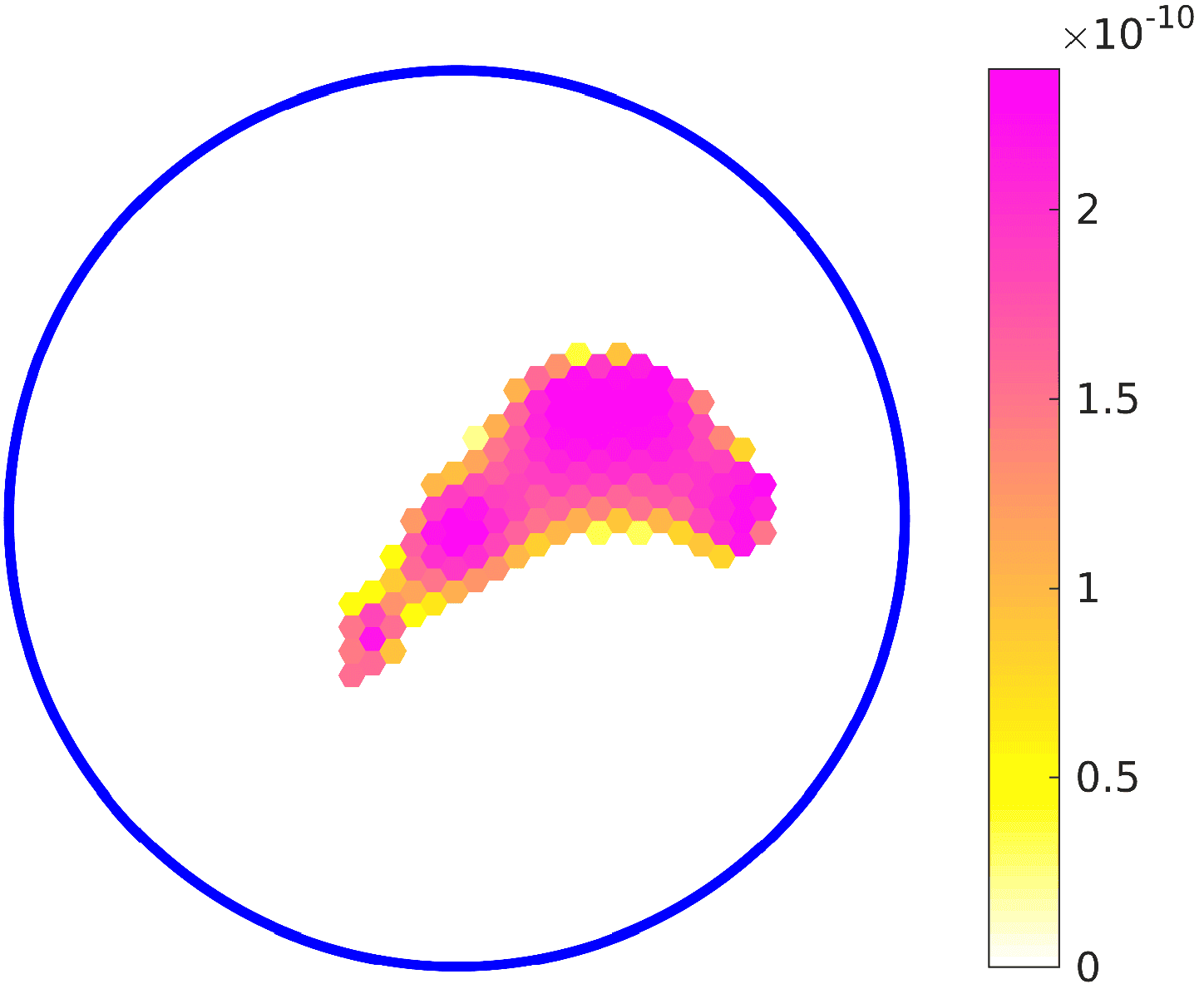}\\
{\scriptsize {\bf Noiseless~2}}&
\includegraphics[width=0.19\textwidth]{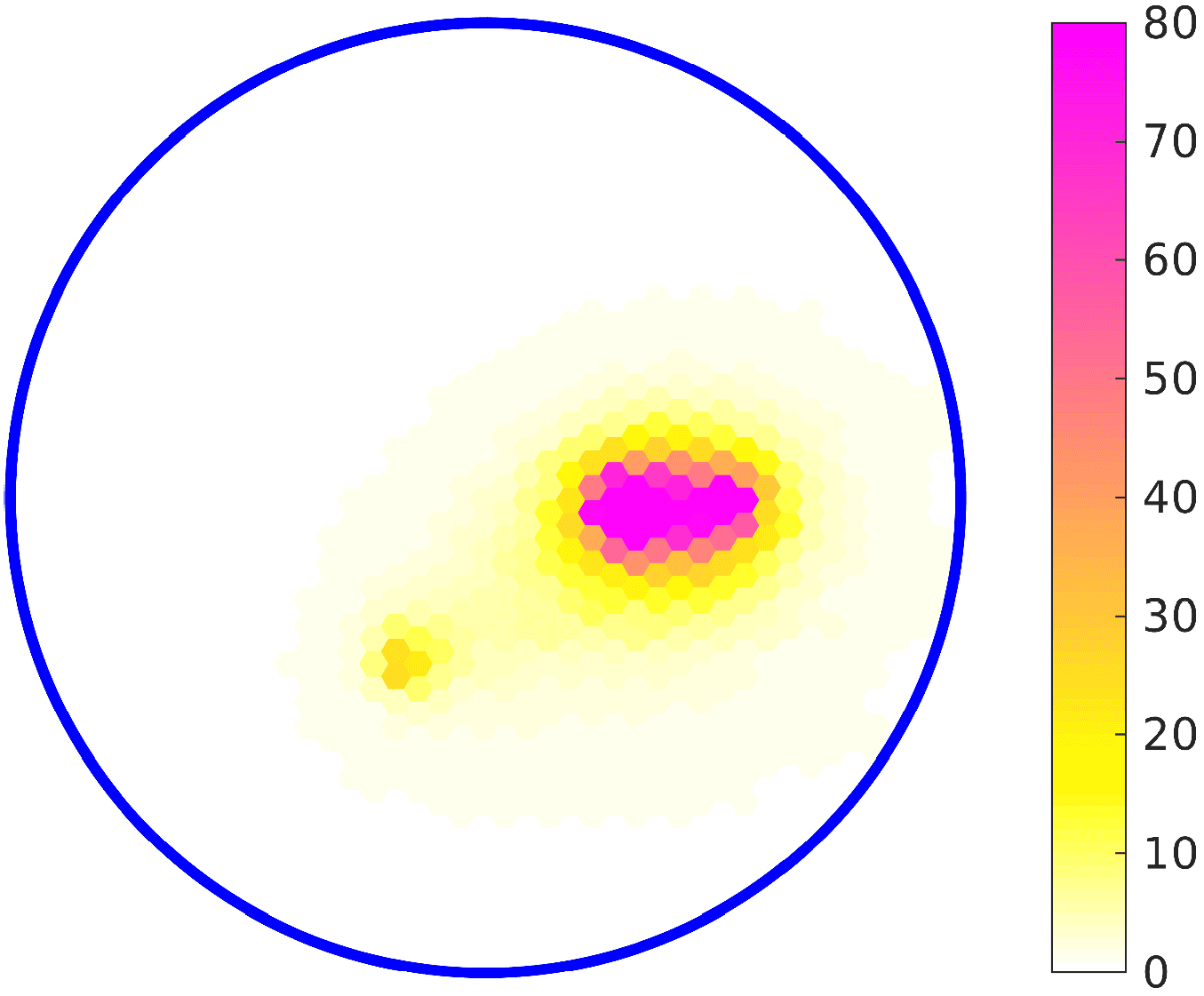}&
\includegraphics[width=0.19\textwidth]{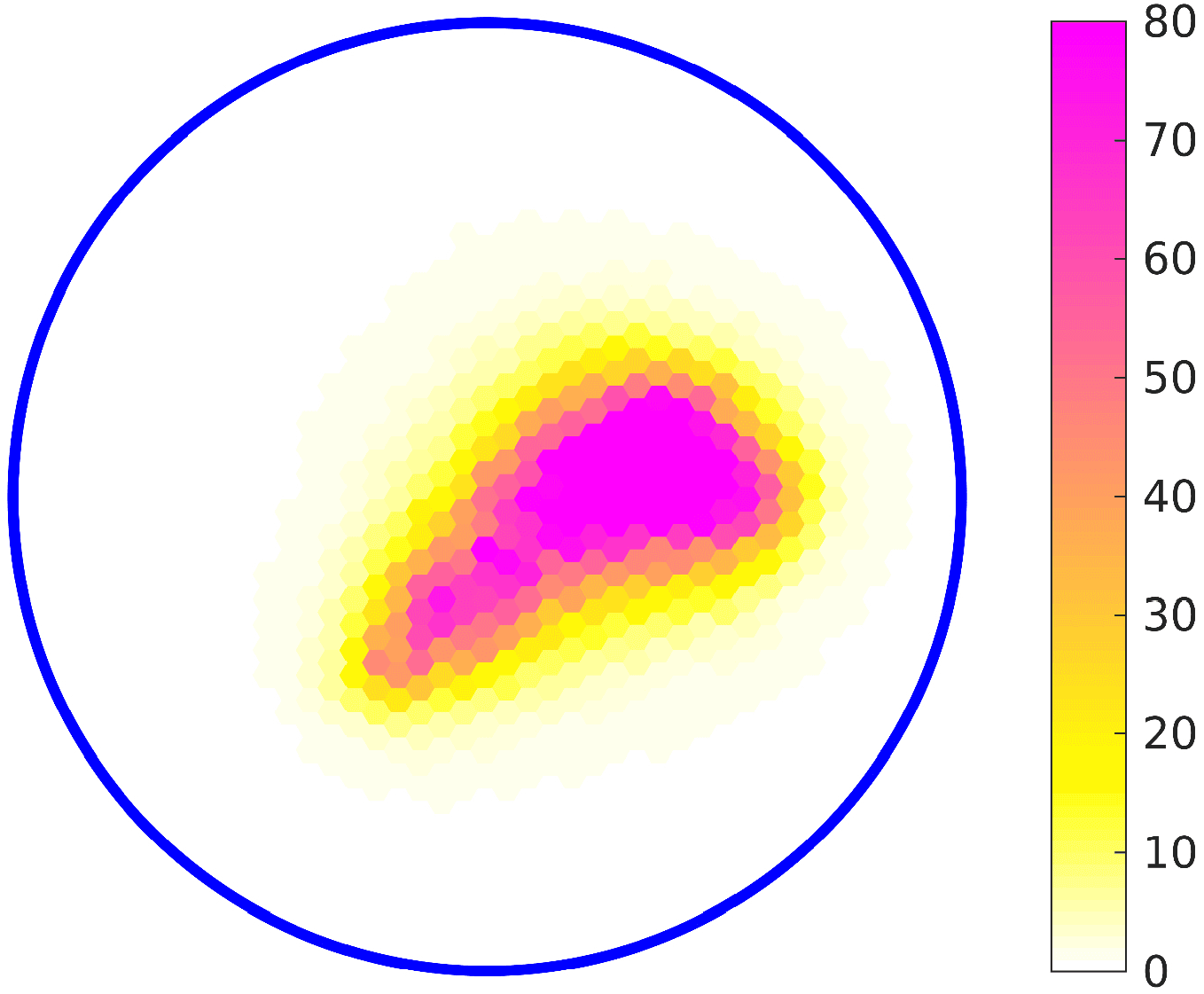}&
\includegraphics[width=0.19\textwidth]{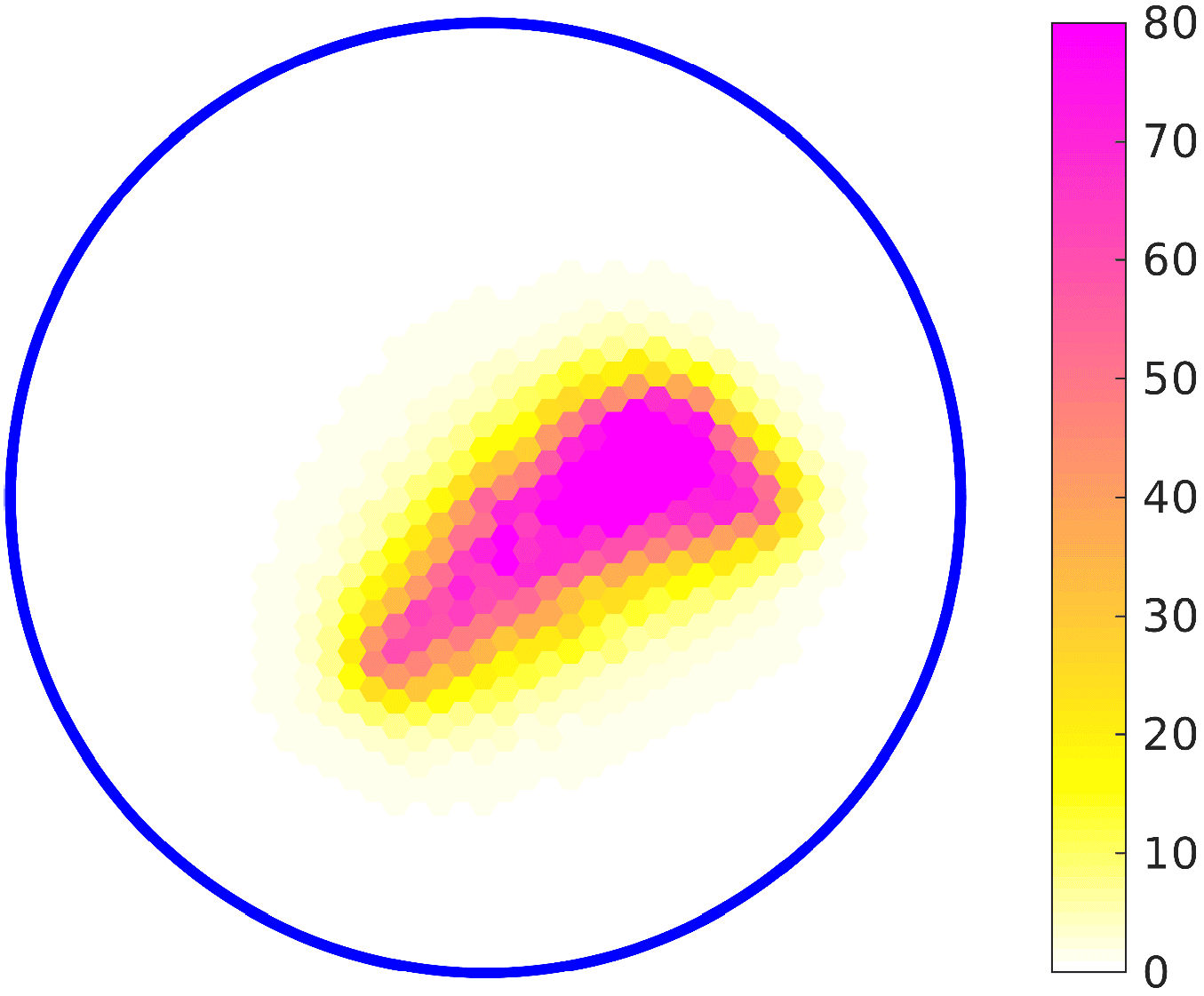}&
\includegraphics[width=0.19\textwidth]{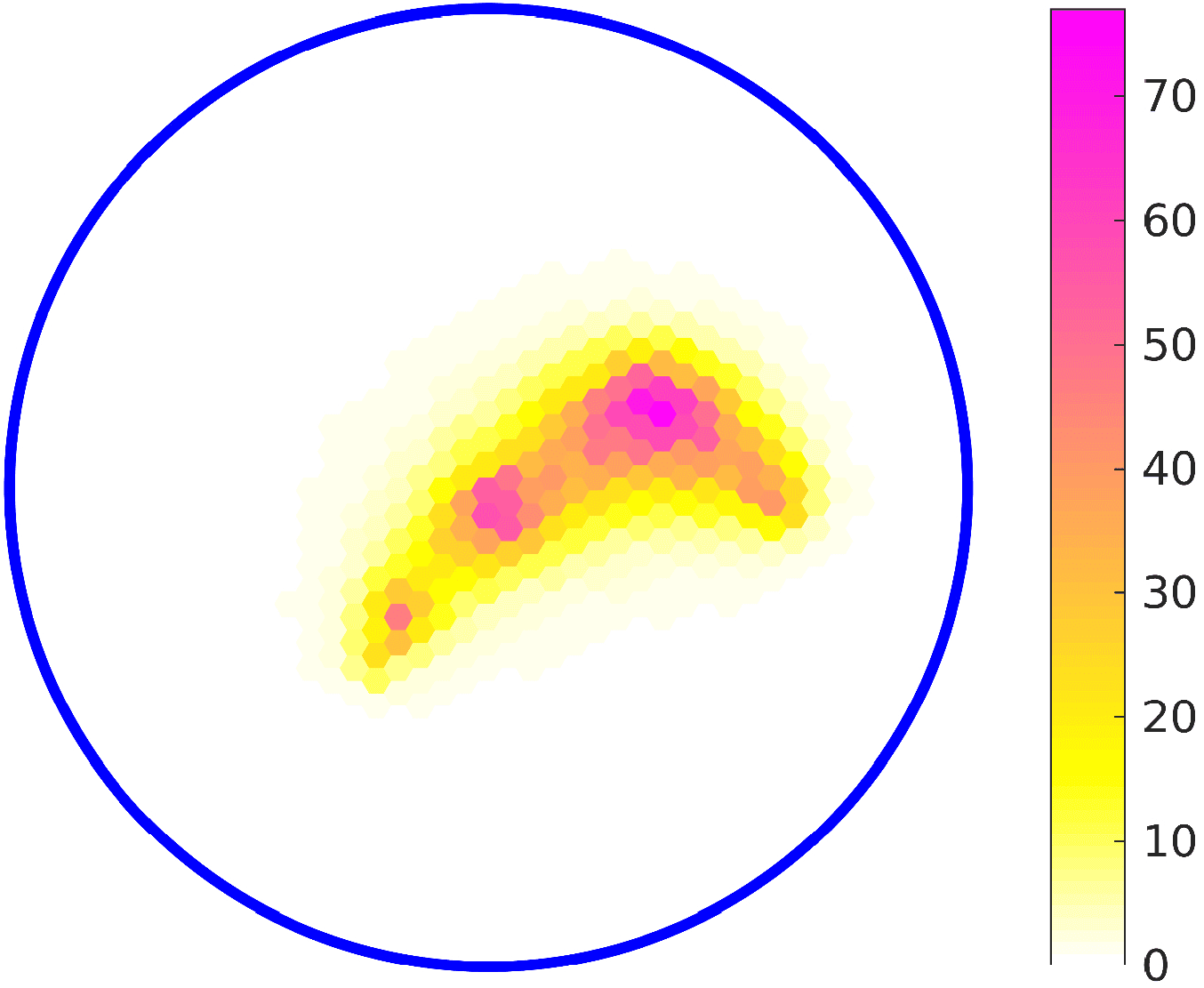}\\
\end{tabular}
\caption{\label{fig:cem-Lshape} Two-dimensional CEM reconstructions from synthetic noiseless data using $k = 8, 16, 32$ and $64$ equispaced electrodes of equal width. Reconstructions in the table rows ``{\bf Noiseless 1}'' and ``{\bf 2}'' are calculated using Algorithms 1 and 2, respectively. For more information on the FE mesh, see the beginning of section~\ref{sec:numerical}.}
\vspace{0.4cm}
\hspace{1.1cm}
\def\arraystretch{1.3}
\begin{tabular}{l}
\\[.07cm]
$
k = 8\hspace{4pt}\left\{
\begin{array}{c}
\\ 
\\
\\
\end{array}
\right.
$
\\
$
k = 16\left\{
\begin{array}{c}
\\ 
\\
\\
\end{array}
\right.
$
\\
$
k = 32\left\{
\begin{array}{c}
\\ 
\\
\\
\end{array}
\right.
$
\\
$
k = 64\left\{
\begin{array}{c}
\\ 
\\
\\
\end{array}
\right.
$
\end{tabular}
\hspace{-0.4cm}
\begin{tabular}{ccc}
\hline
 Parameter & {\bf Noiseless 1} & {\bf Noiseless 2} \\\hline
 ${\rm diam}(B)$ &  $0.053$ & $0.053$ \\
 $\beta$ & $0.8$ & $0.1 + 0.5\N$ \\
 $\mu$ & $1.4$ & $1.4$ \\\hdashline[2pt/2pt]
  ${\rm diam}(B)$ & $0.053$ & $0.053$ \\
 $\beta$ & $0.8$ & $0.1 + 0.5\N$ \\
 $\mu$ & $1.001$ & $1.01$ \\\hdashline[2pt/2pt]
  ${\rm diam}(B)$ & $0.053$ & $0.053$ \\
 $\beta$ & $0.8$ & $0.1 + 0.5\N$ \\
 $\mu$ & $1.00001$ & $1.01$ \\\hdashline[2pt/2pt]
  ${\rm diam}(B)$ & $0.053$ & $0.053$ \\
 $\beta$ & $0.8$ & $0.1 + 0.5\N$ \\
 $\mu$ & $1.00001$ & $1.0001$ \\\hline
\end{tabular}\\[0.2cm]

{\bf Tbl. 2}\, \  Parameter values used in the computations; ${\rm diam}(B)$ is the diameter of the hexagons in the hexagonal reconstruction mesh \eqref{eq:chiB}, $\beta$ is the probing scalar(s) in the semidefiniteness test, and $\mu$ is the regularization parameter \eqref{eq:reg-param-num}.
\end{figure}
Figure~\ref{fig:cm-Lshape} shows the reconstructions produced by Algorithms 1--2 using the discretized CM model. No random noise is added to the data which are simulated using a target conductivity with an L-shaped conductive inclusion. The model uses a fairly high number $p=64$ linearly independent boundary current density inputs of form 
$$
f_m(\theta) = \frac{1}{\sqrt{\pi}} 
\left\{
\begin{array}{lrr}
\cos(m\theta), & \ m & = 1,2,\ldots, p/2, \\
\sin((m-p/2)\theta), & \ m - p/2 & = 1,2,\ldots, p/2,
\end{array}
\right.,
\quad \theta\in [0,2\pi).
$$ 

Figure~\ref{fig:cem-Lshape} displays noiseless reconstructions using the CEM with increasing numbers of electrodes. In this example, current inputs of form
\begin{equation*}
I^{(m)}_j = 
\left\{
\begin{array}{lrrl}
\cos(m2\pi j/k), & \ m &=& 1,2,\ldots,k/2, \\
\sin((m-k/2)2\pi j/k), & \ m - k/2 &=& 1, 2, \ldots, k/2-1,
\end{array}
\right.
\end{equation*} 
are used. Intuitively, the results should not depend too much on the choice of the current basis if the noise level is low. We observe that increasing the number of electrodes seems to improve the reconstruction to the extent that the non-convexity of the target is revealed.

\end{example}

\begin{example}\label{ex:2}

The second example considers three different test objects in which the conductivity consists of a constant background with various convex-shaped inclusions. Both noisy and noiseless measurements are simulated using $k = 16$ electrodes. The reconstructions computed with both Algorithms~1--2 are displayed in Figure~\ref{fig:incs}. In the rightmost target, the inclusions are less conductive than the background, i.e., $\gamma = \gamma_0 - \kappa\chi_D$ where $\gamma_0 > \kappa > 0$. In this case, the suitable semidefiniteness test is 
\[
\min \sigma(R^\delta + \beta R'(\gamma_0)\chi_B - R(\gamma_0) + \alpha\id) \geq 0.
\]
Using the right-hand side of \eqref{eq:mono}, one can deduce that a sufficient condition for the probe constant is $\beta \leq {\rm ess}\inf (-\kappa)$. In the noiseless case we observe that \eqref{eq:reg-param-num} yields a negative $\alpha$. However, we emphasize that negative regularization parameters are allowed by both \eqref{eq:reg-param} and its adaptation to the case with resistive inclusions. The results indicate that interesting information about the inclusion locations can be retrieved in a relatively realistic simulated setting. Moreover, we notice that Algorithm~2 is more flexible in the sense that it enables compensating bad choices of $\alpha$ by increasing $\beta$. 
 
\begin{figure}[tbh]
\begin{tabular}{m{1.3cm}m{3.0cm}m{3.0cm}m{3.0cm}}
{\scriptsize{\bf Target}}&
\includegraphics[width=0.25\textwidth]{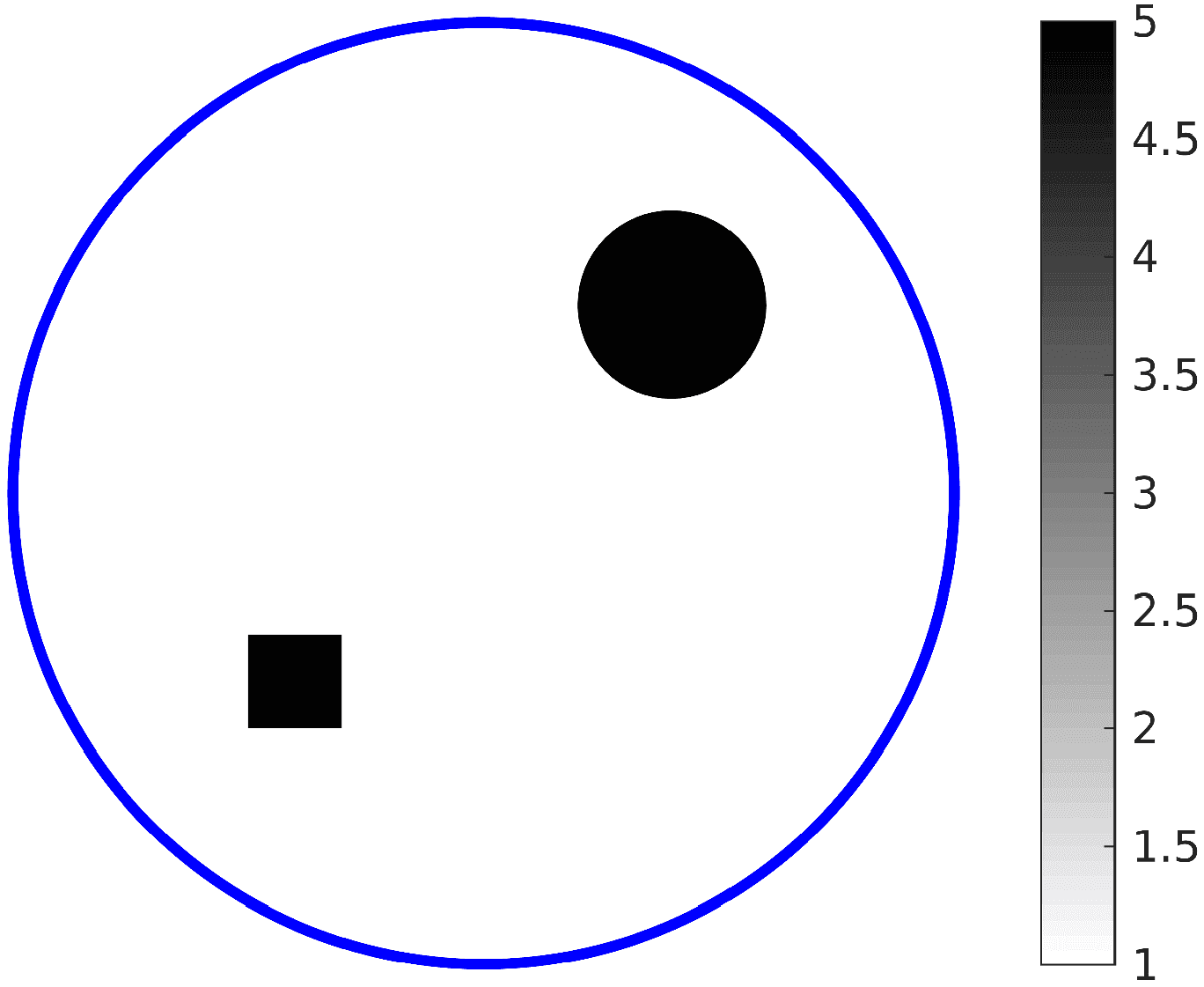}&
\includegraphics[width=0.25\textwidth]{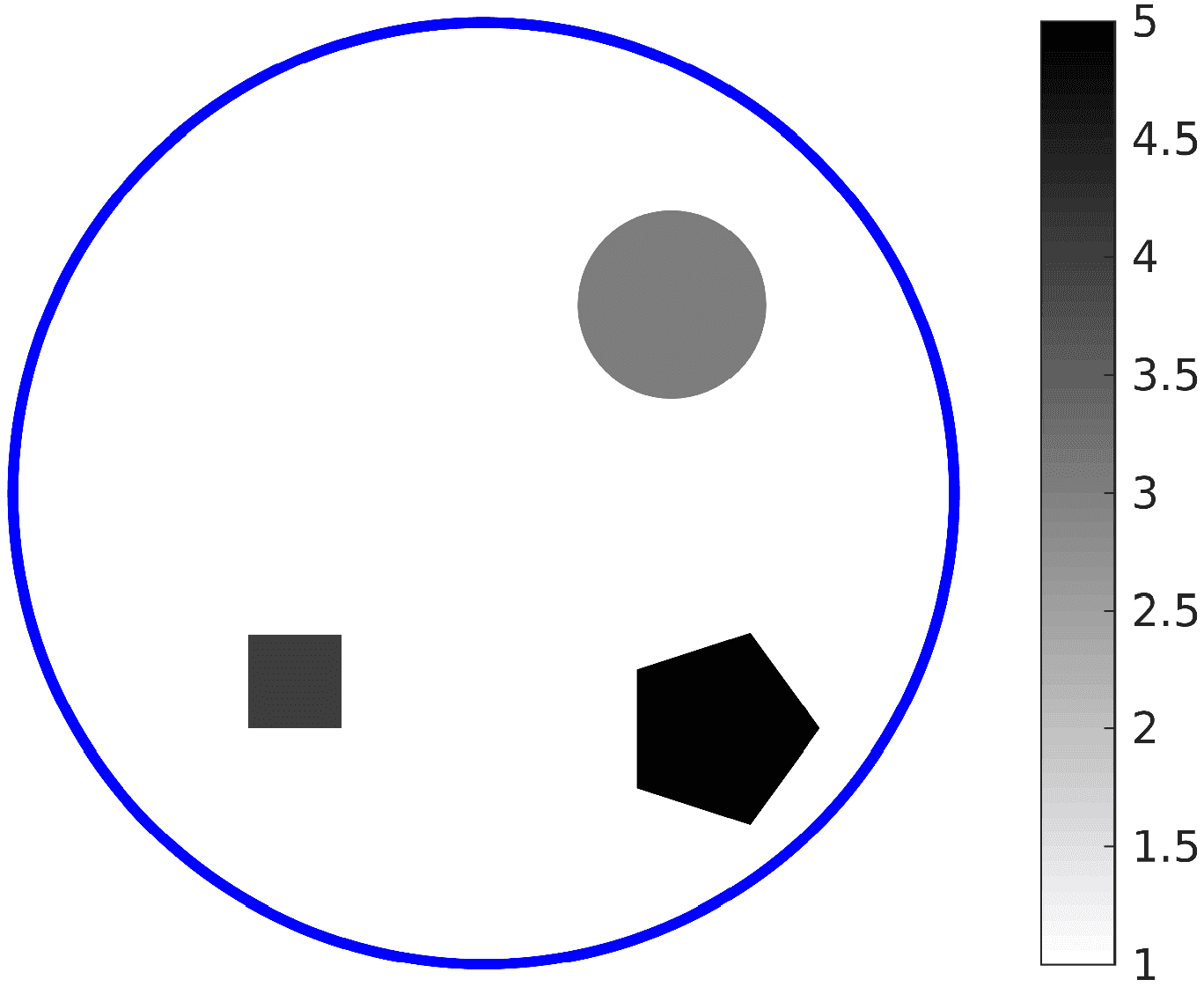}&
\includegraphics[width=0.25\textwidth]{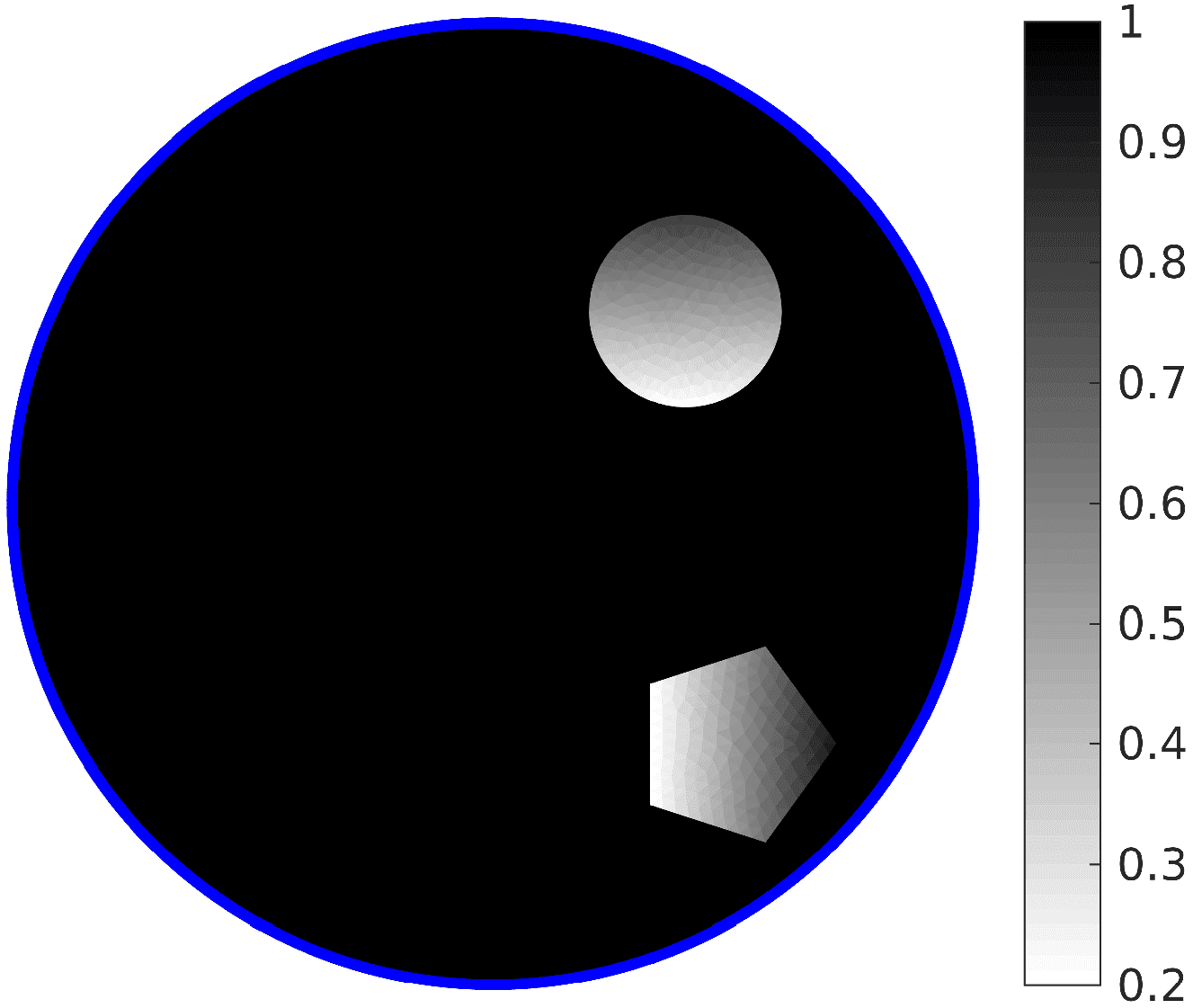}\\ 
{\scriptsize{\bf Noiseless~1} (CEM)} & 
\includegraphics[width=0.26\textwidth]{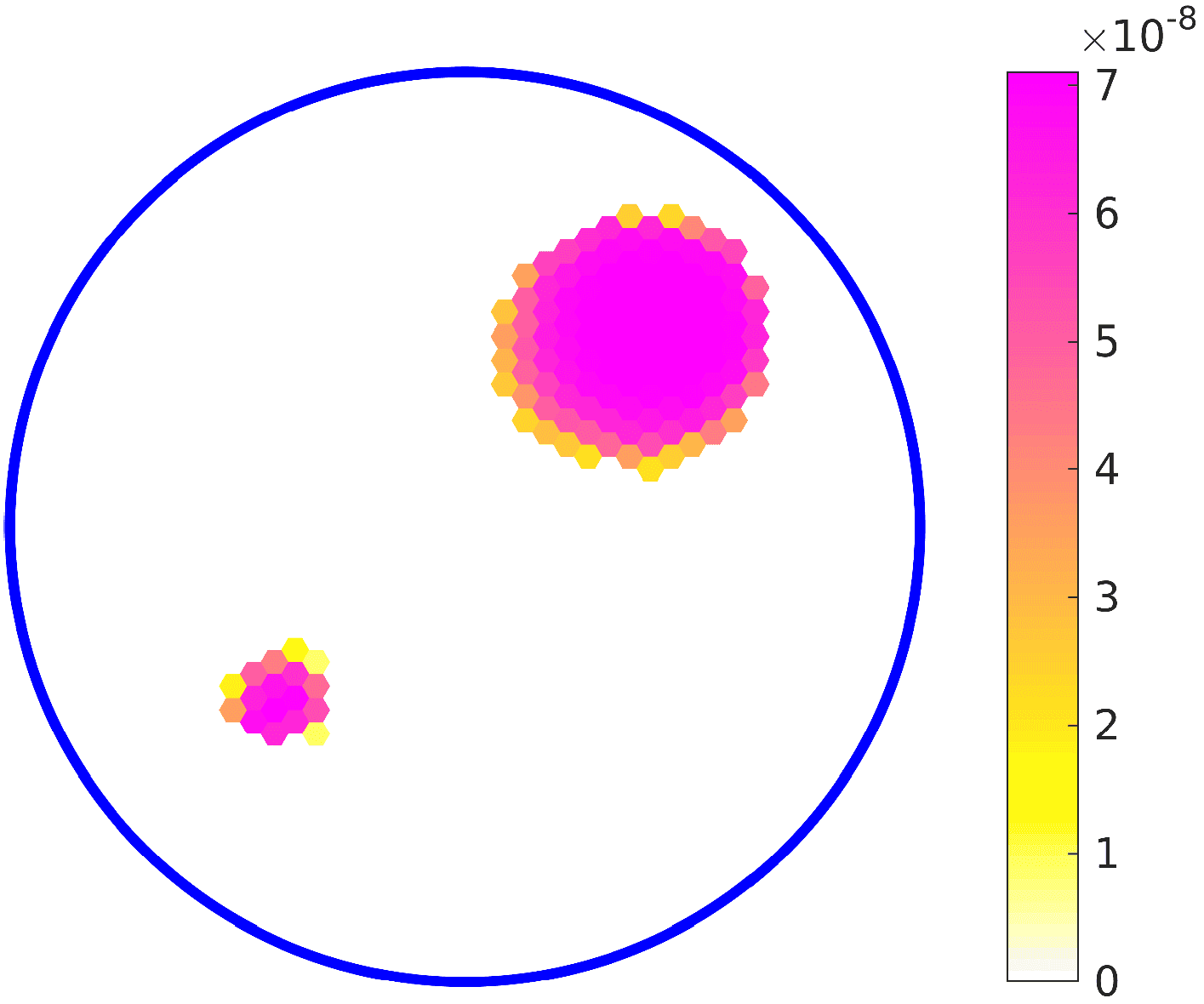}&
\includegraphics[width=0.26\textwidth]{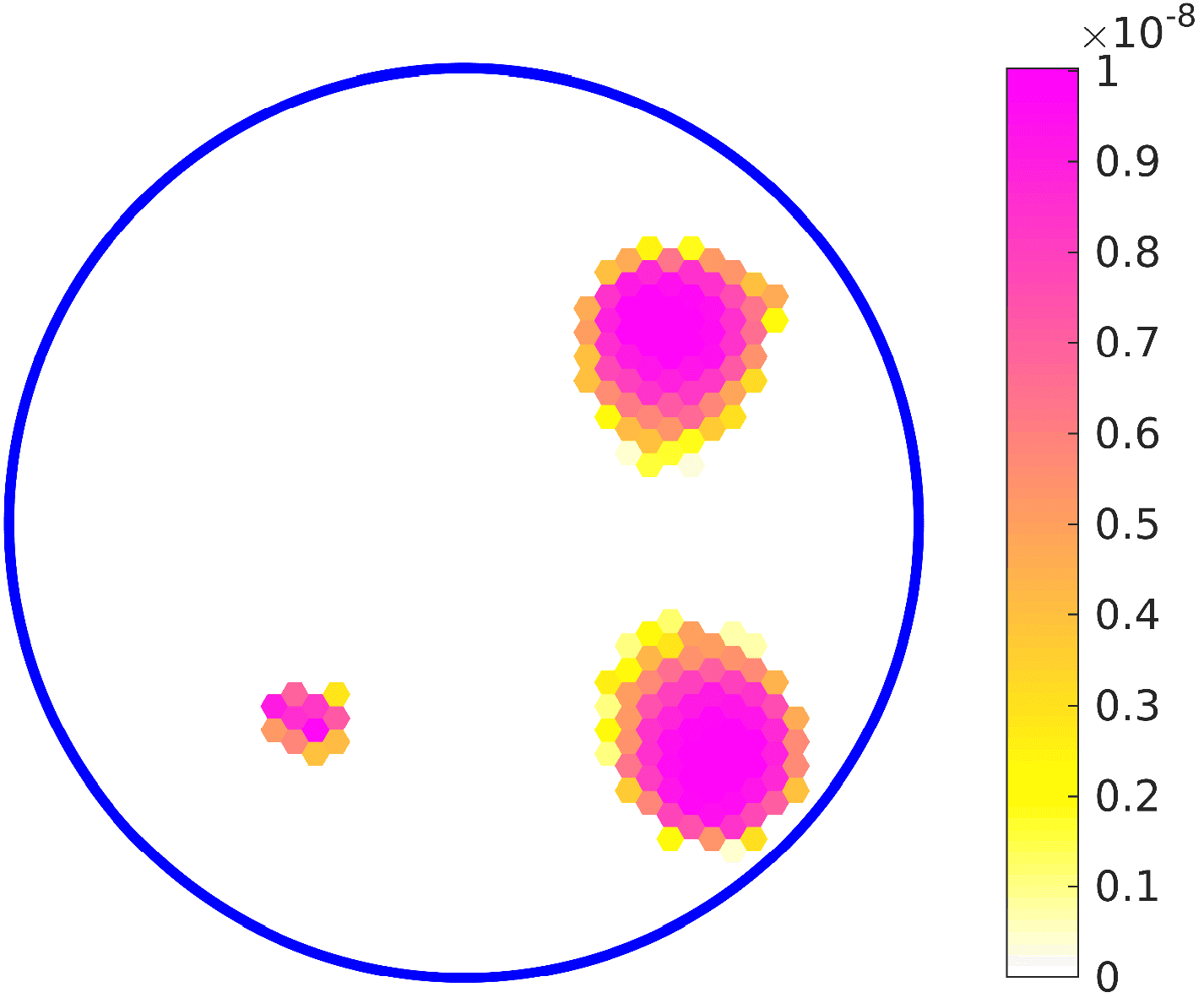}&
\includegraphics[width=0.26\textwidth]{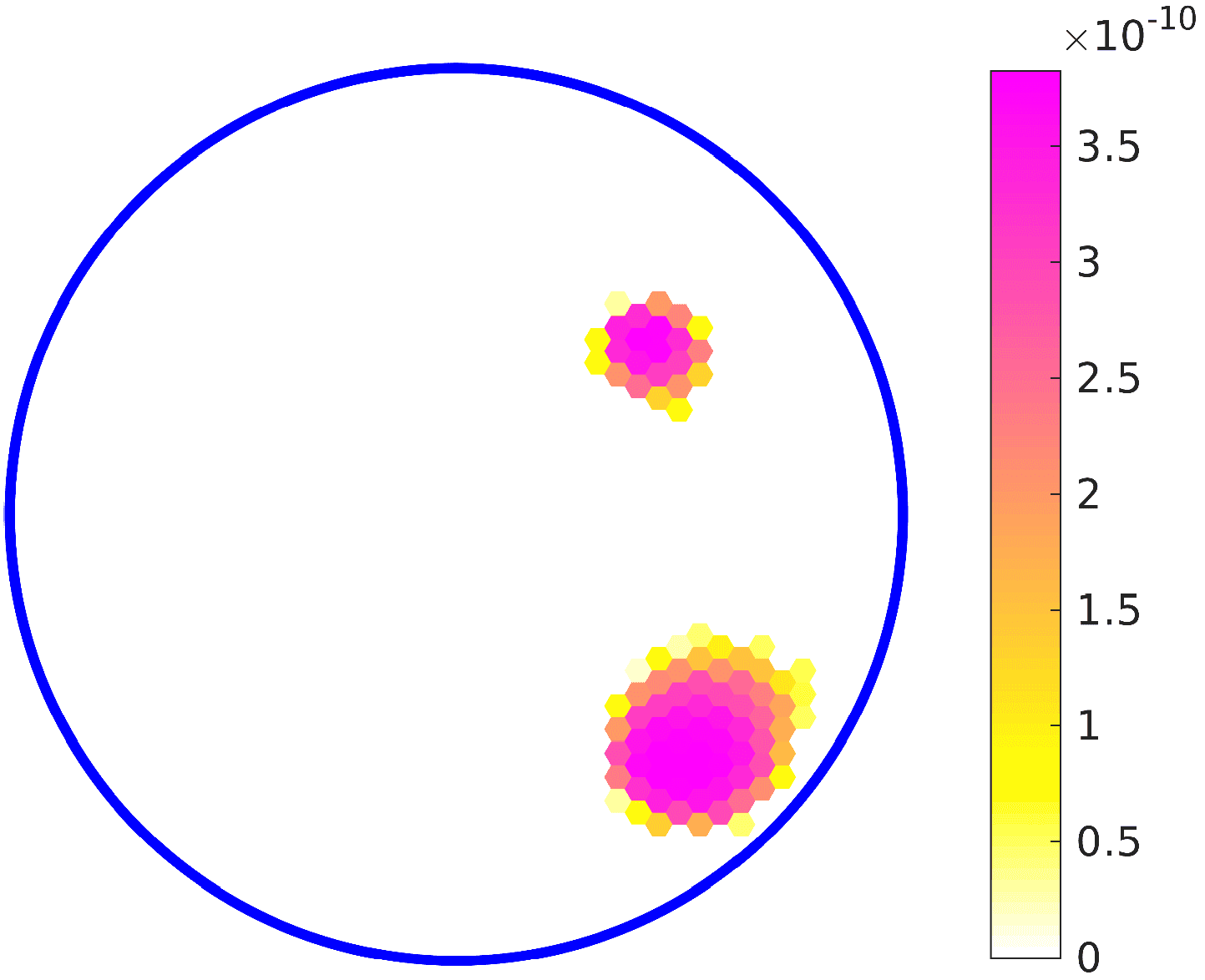}\\
{\scriptsize{\bf Noiseless~2} (CEM)} & 
\includegraphics[width=0.25\textwidth]{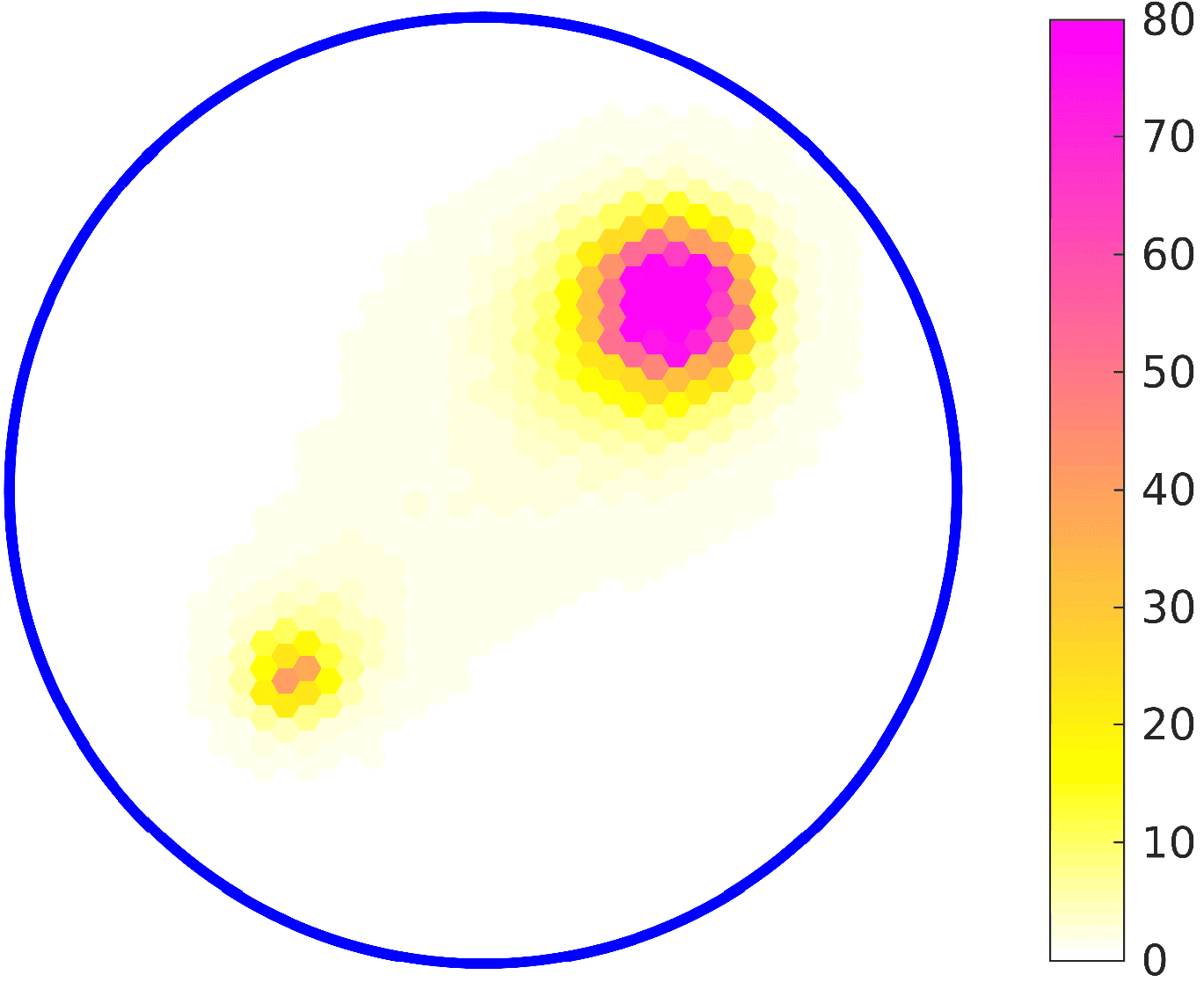}&
\includegraphics[width=0.25\textwidth]{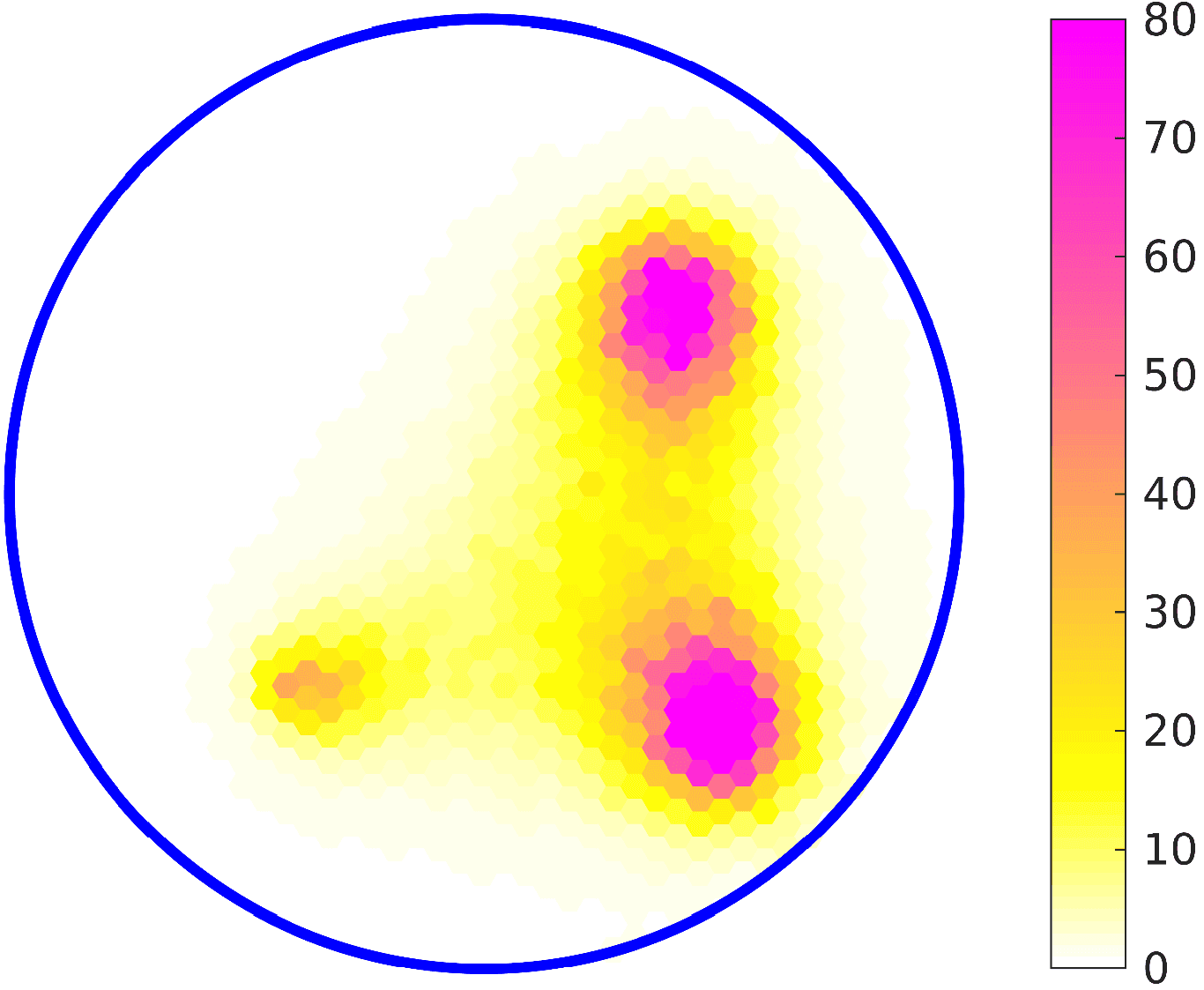}&
\includegraphics[width=0.25\textwidth]{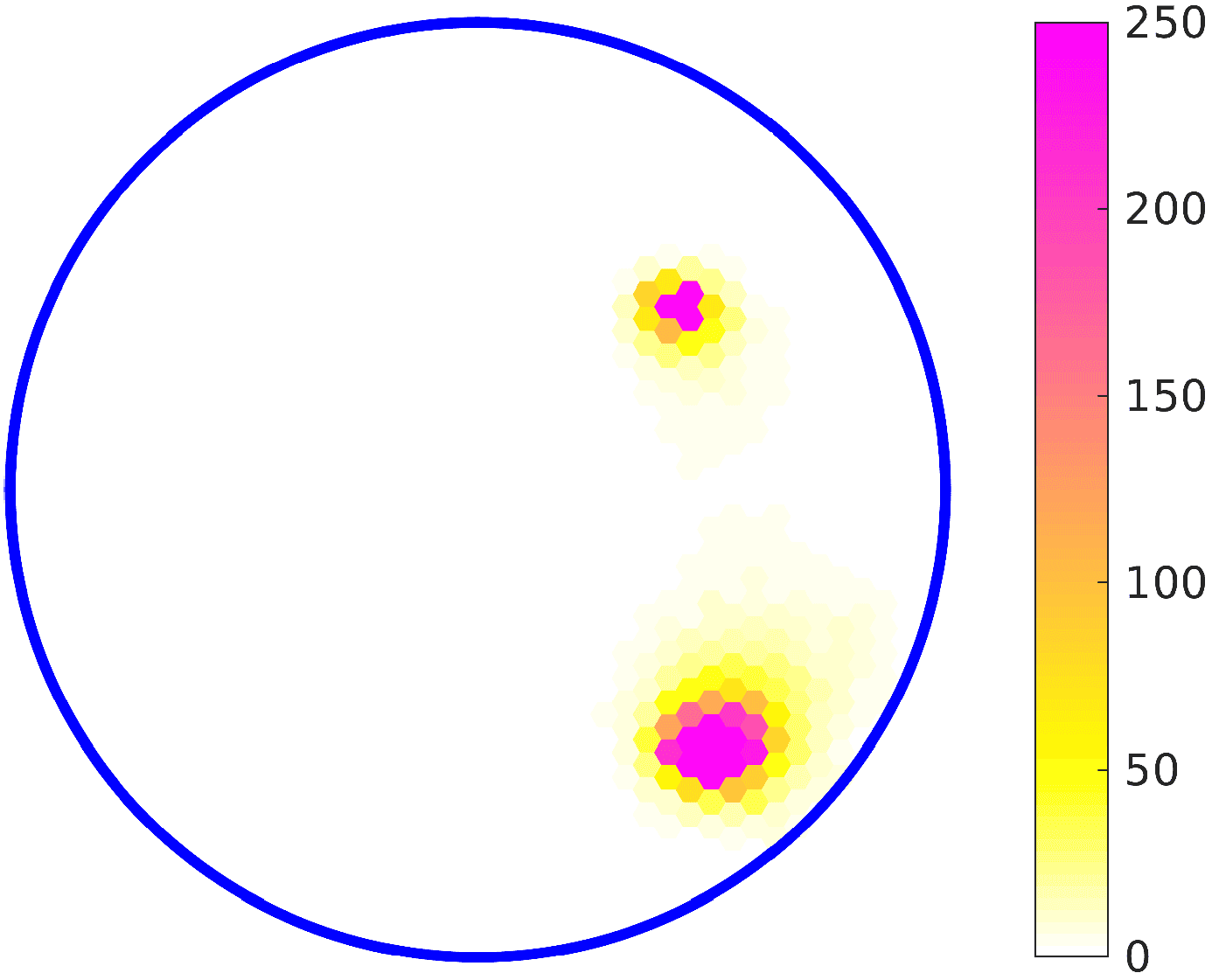}\\
{\scriptsize{\bf Noisy{\color{white}.}1} (CEM)} &
\includegraphics[width=0.26\textwidth]{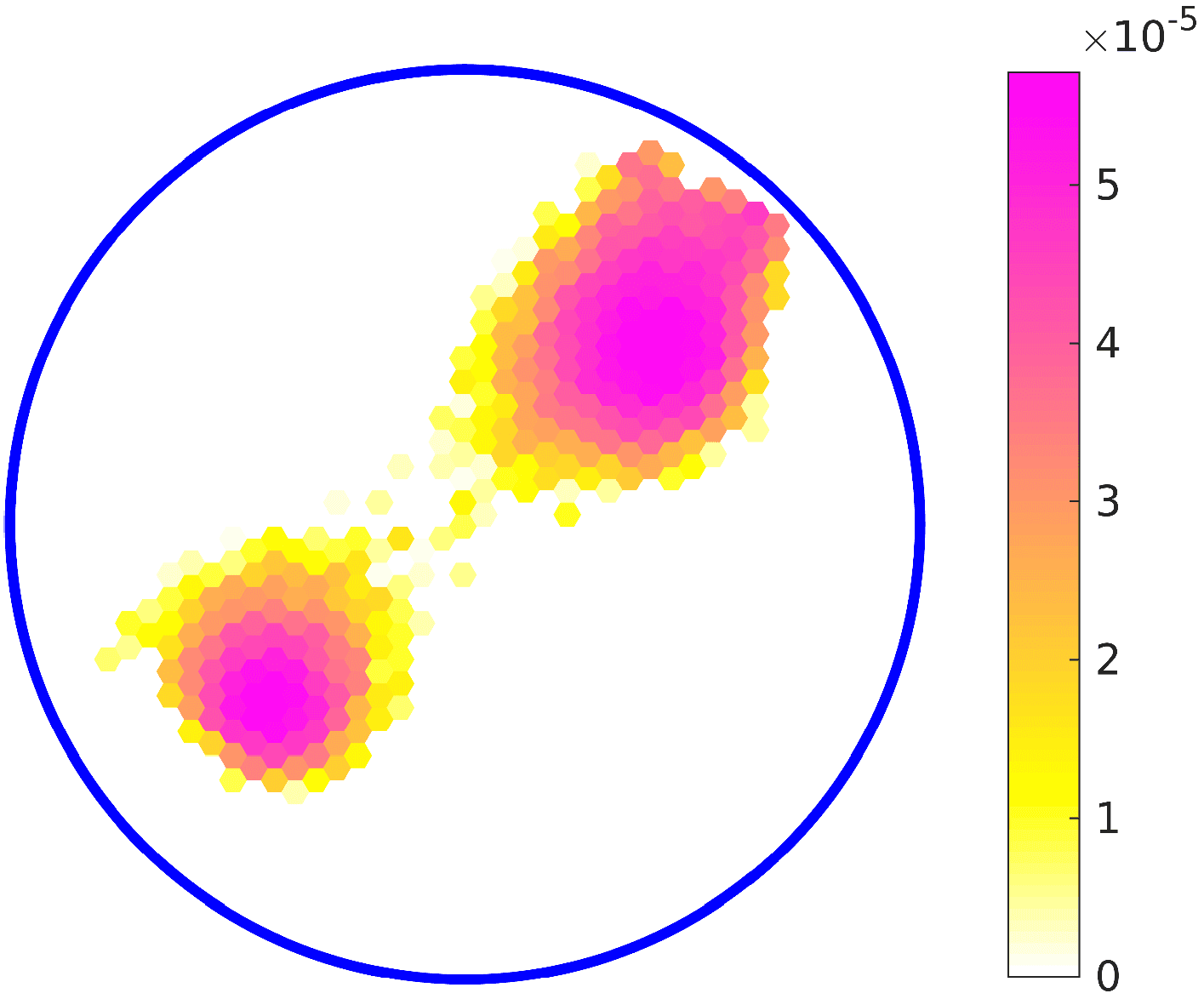}&
\includegraphics[width=0.26\textwidth]{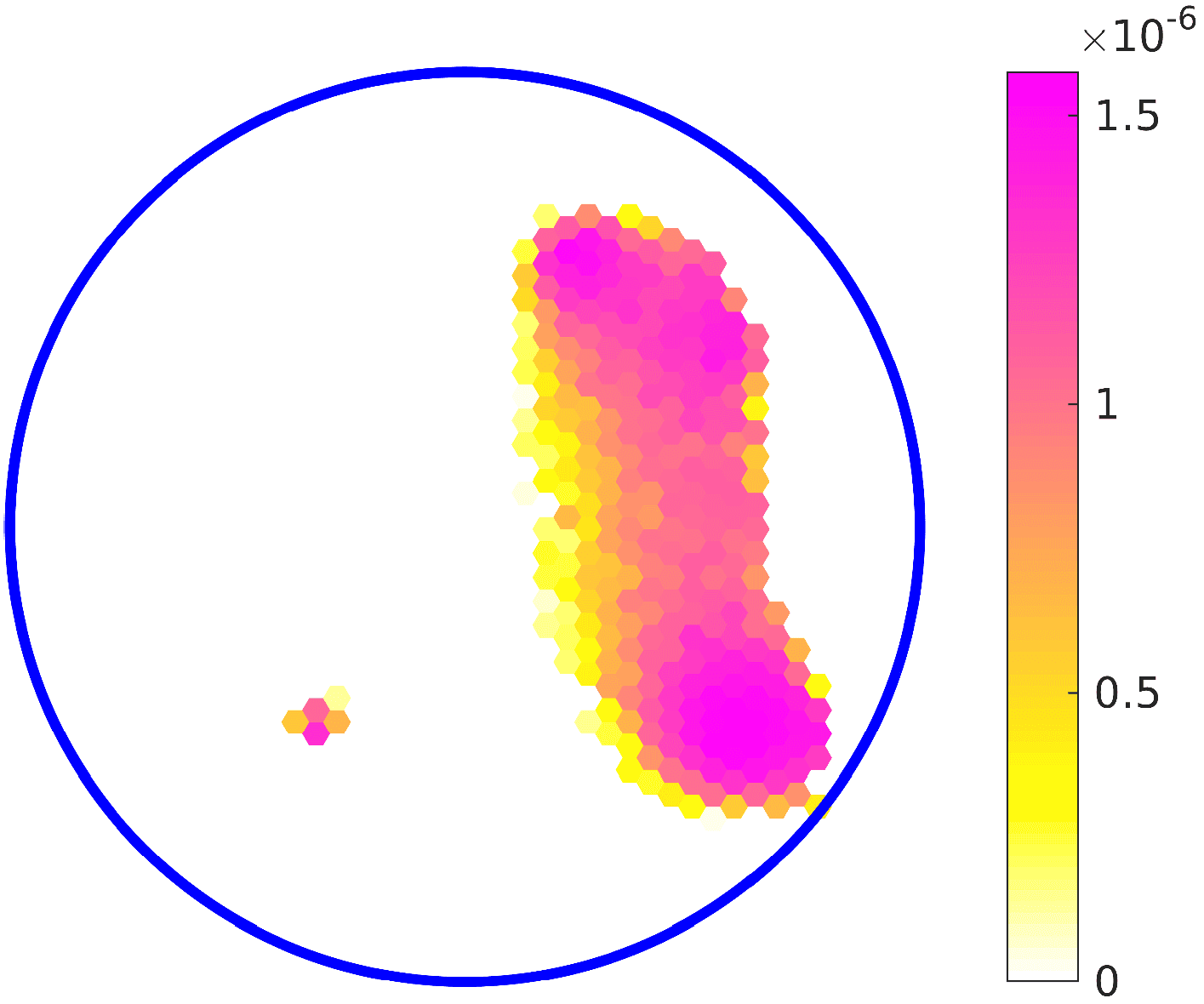}&
\includegraphics[width=0.26\textwidth]{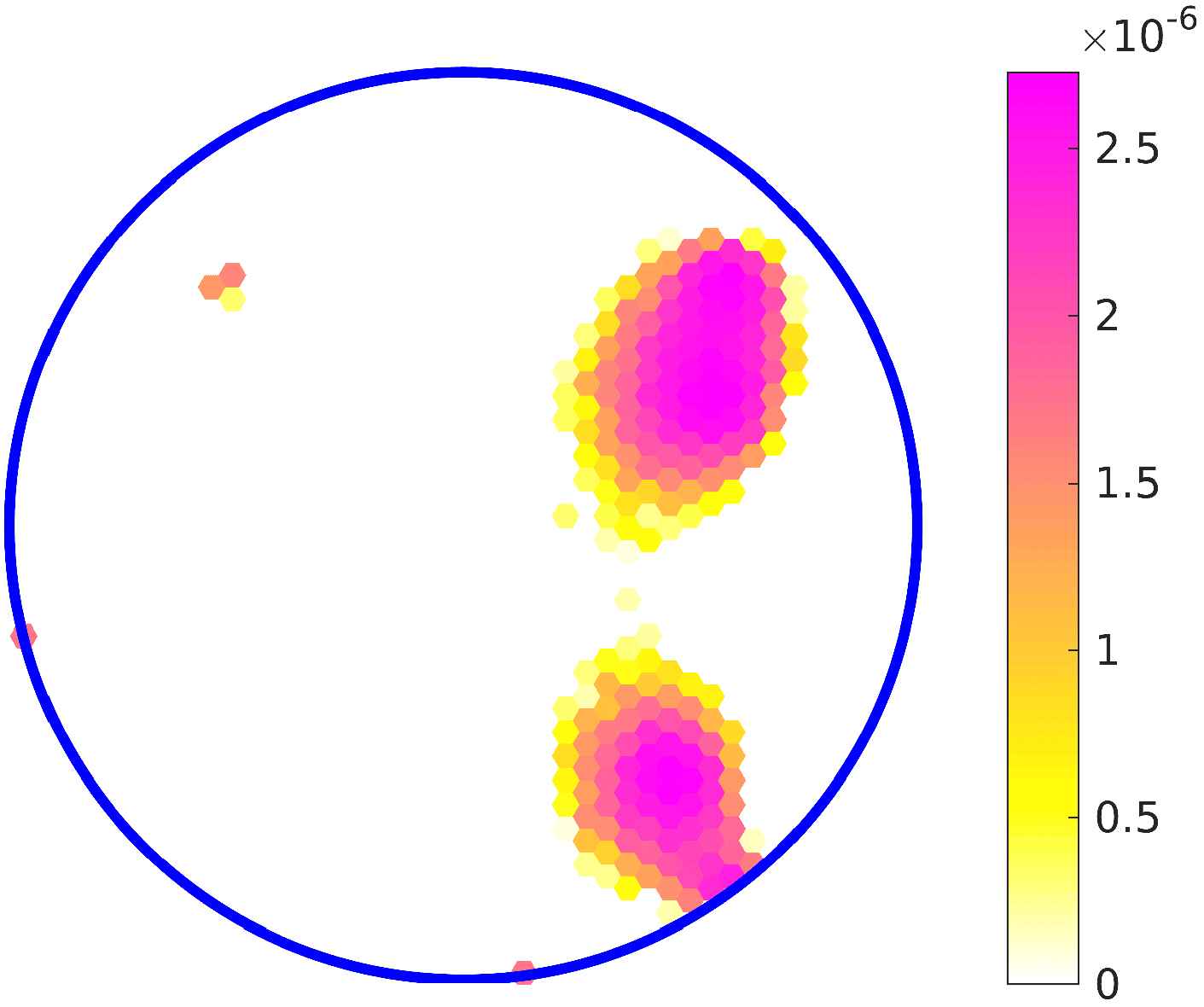}\\
{\scriptsize{\bf Noisy{\color{white}.}2} (CEM)} &
\includegraphics[width=0.25\textwidth]{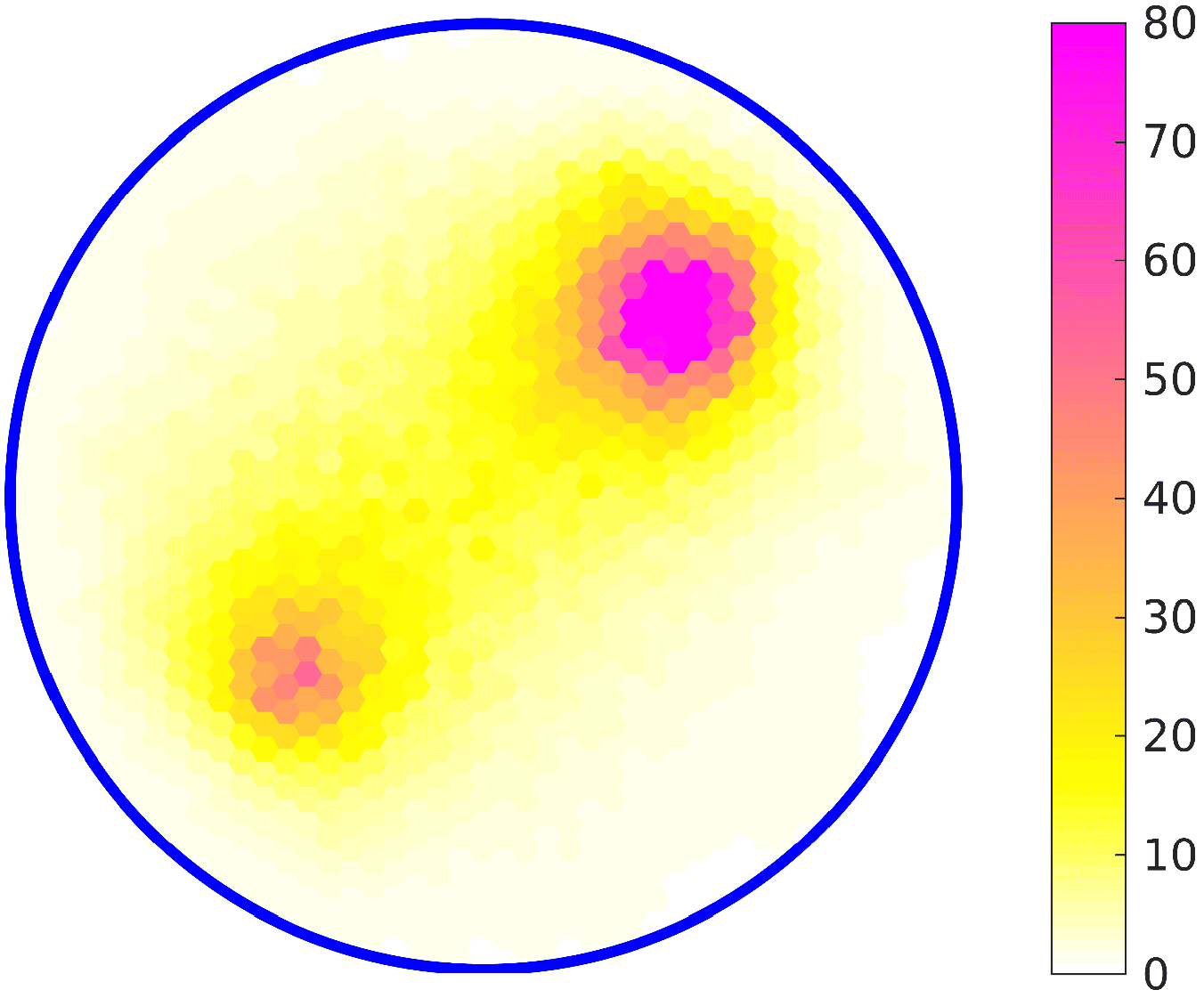}&
\includegraphics[width=0.25\textwidth]{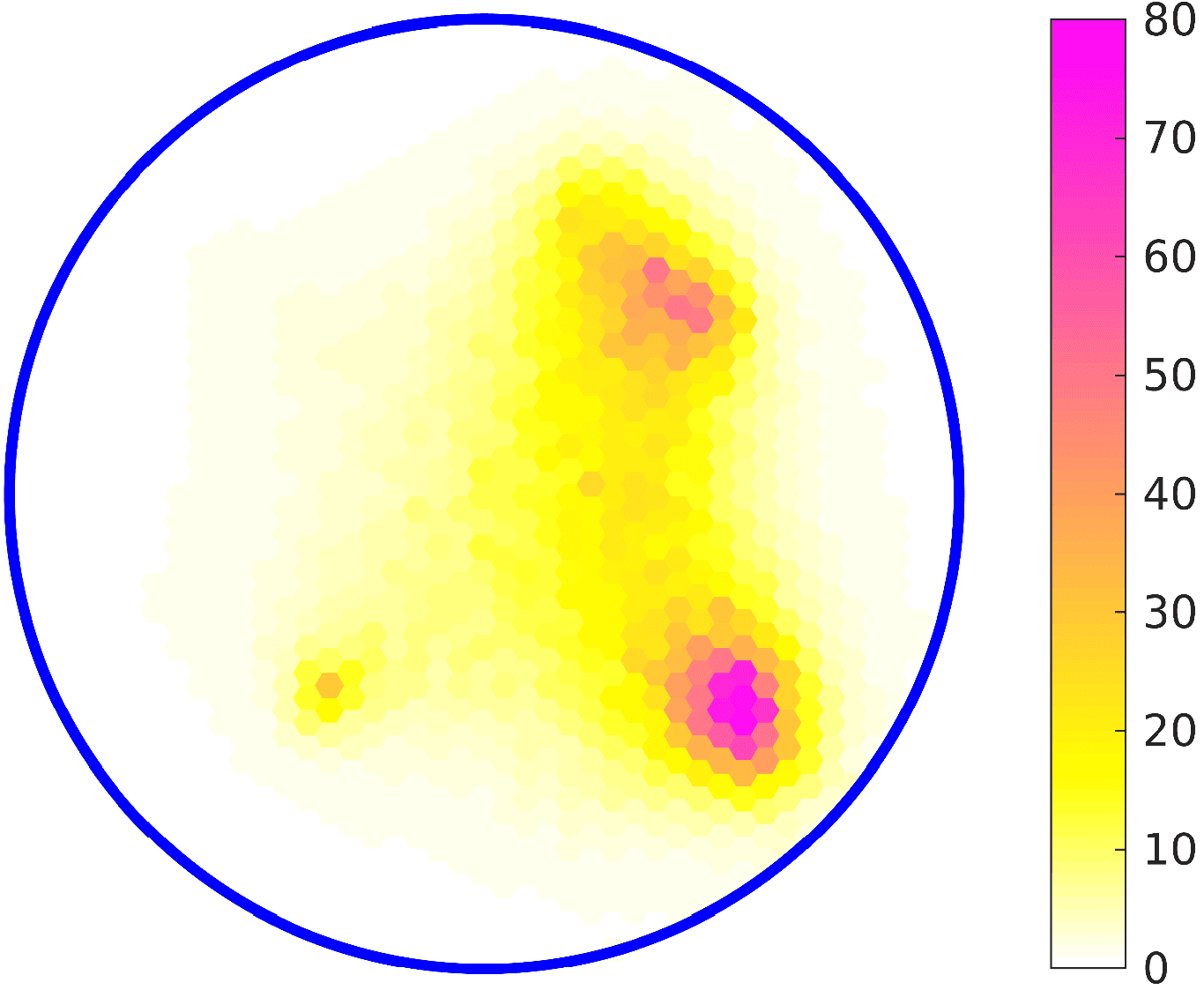}&
\includegraphics[width=0.25\textwidth]{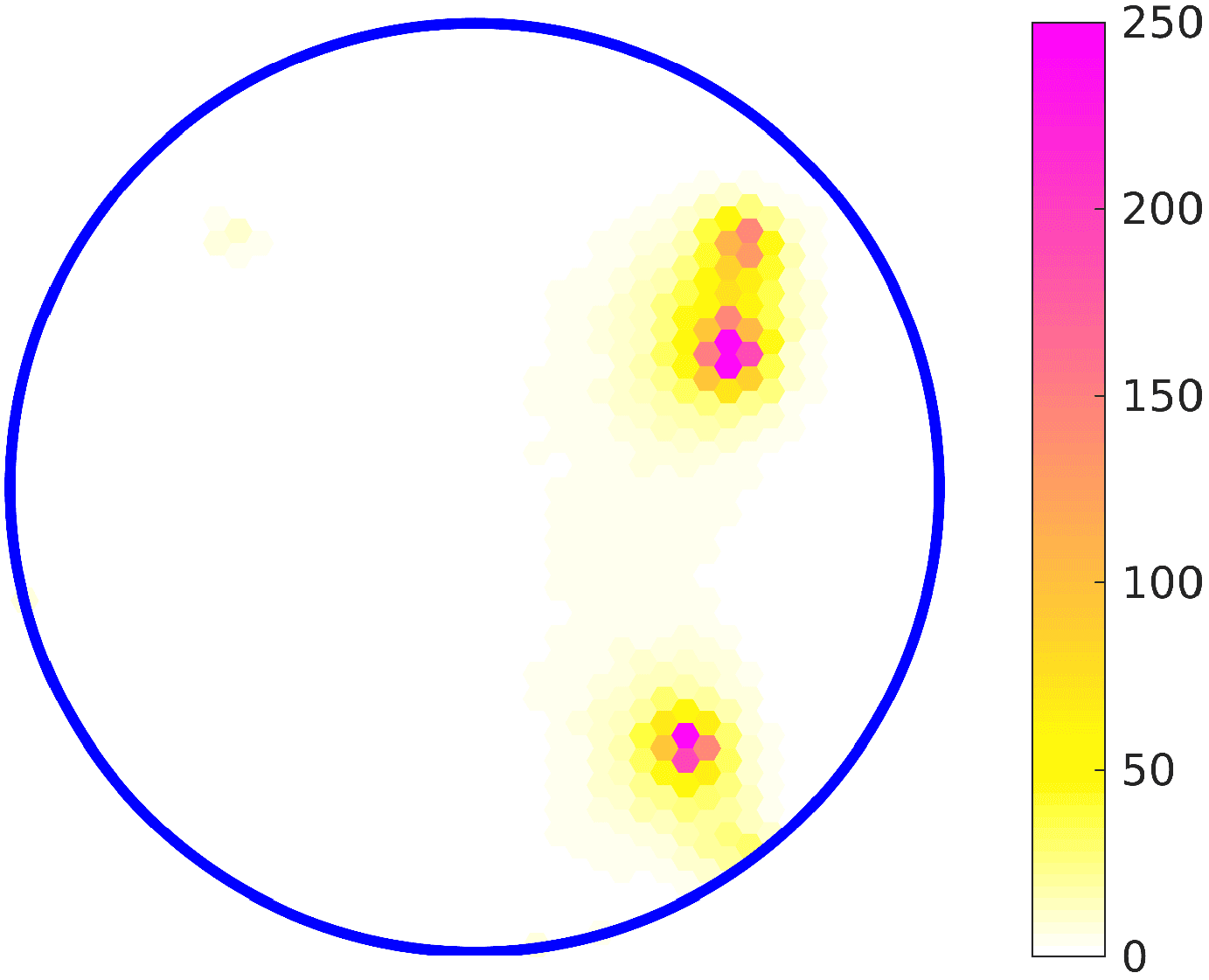}
\end{tabular}
\caption{\label{fig:incs}Two-dimensional reconstructions on a unit disk with $16$ equispaced and equidistant electrodes that cover 50\% of the boundary. Reconstructions in the table rows ``{\bf Noiseless/Noisy 1}'' and ``{\bf 2}'' are calculated using Algorithms 1 and 2, respectively. In each column, the noisy reconstructions are computed from a single dataset that contains around $0.5\%$ pseudorandom noise; see \eqref{eq:noisemodel} for details on the noise simulation. For more information on the FE mesh, see the beginning of section~\ref{sec:numerical}.}
\vspace{0.4cm}
\hspace{-0.1cm}
\def\arraystretch{1.3}
\begin{tabular}{l}
\\[.07cm]
\hspace{0.28cm}
$
{\rm Left}
\left\{
\begin{array}{c}
\\ 
\\
\\
\end{array}
\right.
$
\\
$
{\rm Middle}
\left\{
\begin{array}{c}
\\ 
\\
\\
\end{array}
\right.
$
\\
\hspace{0.08cm}
$
{\rm Right}
\left\{
\begin{array}{c}
\\ 
\\
\\
\end{array}
\right.
$
\end{tabular}
\hspace{-0.4cm}
\begin{tabular}{ccccc}
\hline
 Parameter & {\bf Noiseless 1} & {\bf Noiseless 2} & {\bf Noisy 1} & {\bf Noisy 2} \\\hline
 ${\rm diam}(B)$ &  $0.053$ & $0.053$ & $0.053$ & $0.053$ \\
 $\beta$ & $0.8$ & $0.1 + 0.5\N$ & $0.8$ & $0.1 + 0.5\N$ \\
 $\mu$ & $1.001$ & $1.01$ & $1.01$ & $1.01$ \\\hdashline[2pt/2pt]
  ${\rm diam}(B)$ & $0.053$ & $0.053$ & $0.053$ & $0.053$ \\
 $\beta$ & $0.66$ & $0.1 + 0.5\N$ & $0.66$ & $0.1 + 0.5\N$ \\
 $\mu$ & $1.0002$ & $1.01$ & $1.01$ & $1.01$ \\\hdashline[2pt/2pt]
  ${\rm diam}(B)$ & $0.053$ & $0.053$ & $0.053$ & $0.053$ \\
 $\beta$ & $-0.1$ & $-0.01 - 0.02\N$ & $-0.01$ & $-0.01 - 0.02\N$ \\
 $\mu$ & $0.99998$ & $0.99998$ & $1.001$ & $1.001$ \\\hline

\end{tabular}\\[0.2cm]

{\bf Tbl. 3}\, \  Parameter values used in the computations; ${\rm diam}(B)$ is the diameter of the hexagons in the hexagonal reconstruction mesh \eqref{eq:chiB}, $\beta$ is the probing scalar(s) in the semidefiniteness test, and $\mu$ is the regularization parameter \eqref{eq:reg-param-num}.
\end{figure}

\end{example}

\begin{example}\label{ex:3}

In this example we verify that, in principle, Algorithm~2 can provide reasonable reconstructions from real-life measurement data. The measurement data are gathered at the laboratory of the Applied Physics Department at University of Eastern Finland in Kuopio. The test object consists of a cylinder tank filled with regular tap water and iron objects. There are 16 rectangular electrodes on the lateral surface of the tank, and they are homogeneous along the symmetry axis of the tank. Moreover, the water surface is set along the top edge of the electrodes. Thus, the measurement geometry is essentially two-dimensional. The radius of the tank cross-section is 14.0 cm, and the electrode width and height are 2.5 cm and 7.0 cm, respectively. The measurements were performed using a dipole current basis $I^{(m)} = ({\rm e}^{(1)} - {\rm e}^{(m+1)}) \times {\rm 1.0 \ mA}$, $m=1,2,\ldots,k-1$.

\begin{figure}[tbh]
\begin{center}
\begin{tabular}{cc}
{\bf Target}&
{\bf Reconstrution~2} (CEM)\\[0.3cm]
\includegraphics[width=0.4\textwidth]{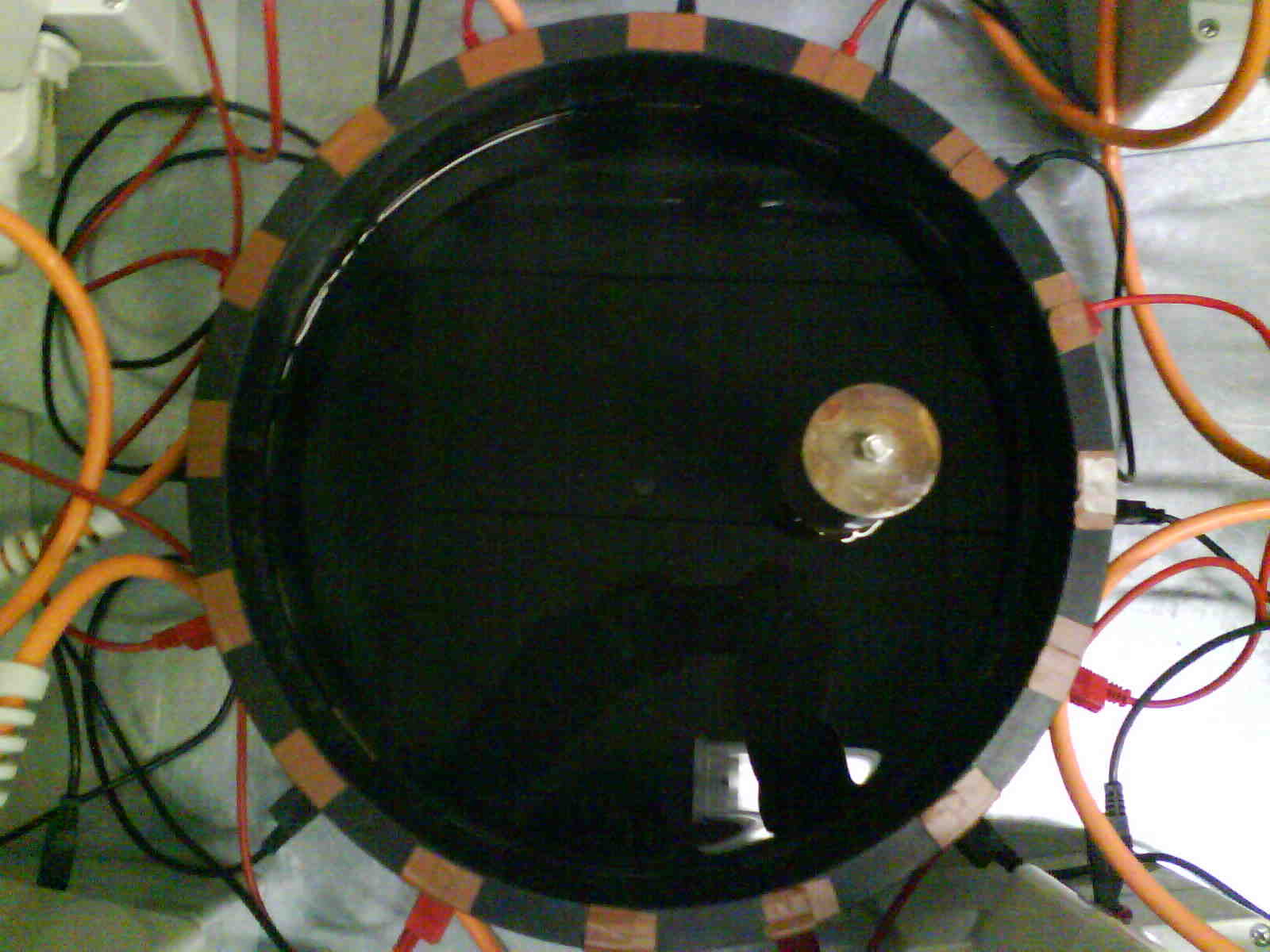}& 
\raisebox{0.3cm}{\includegraphics[width=0.3\textwidth]{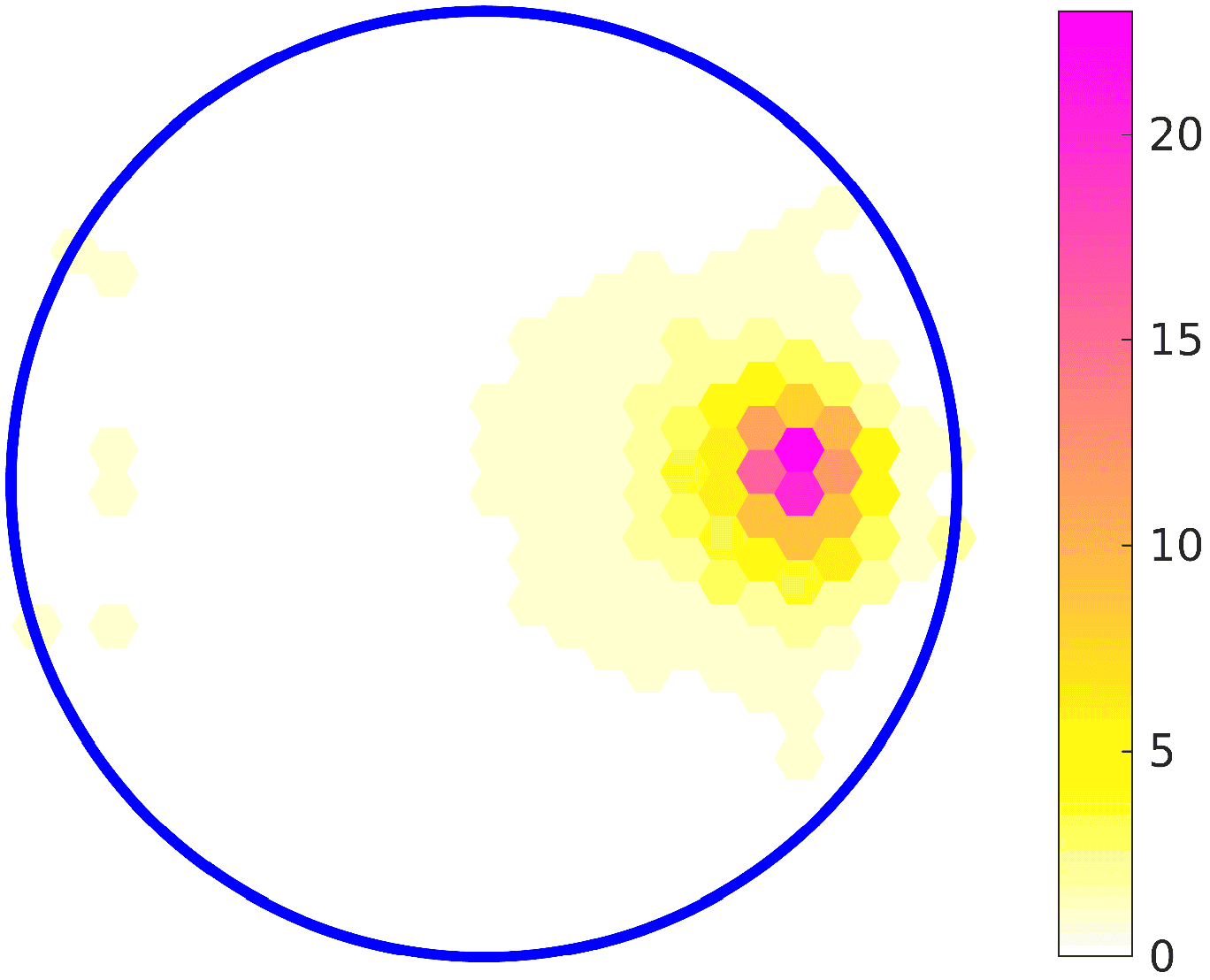}}
\\
\includegraphics[width=0.4\textwidth]{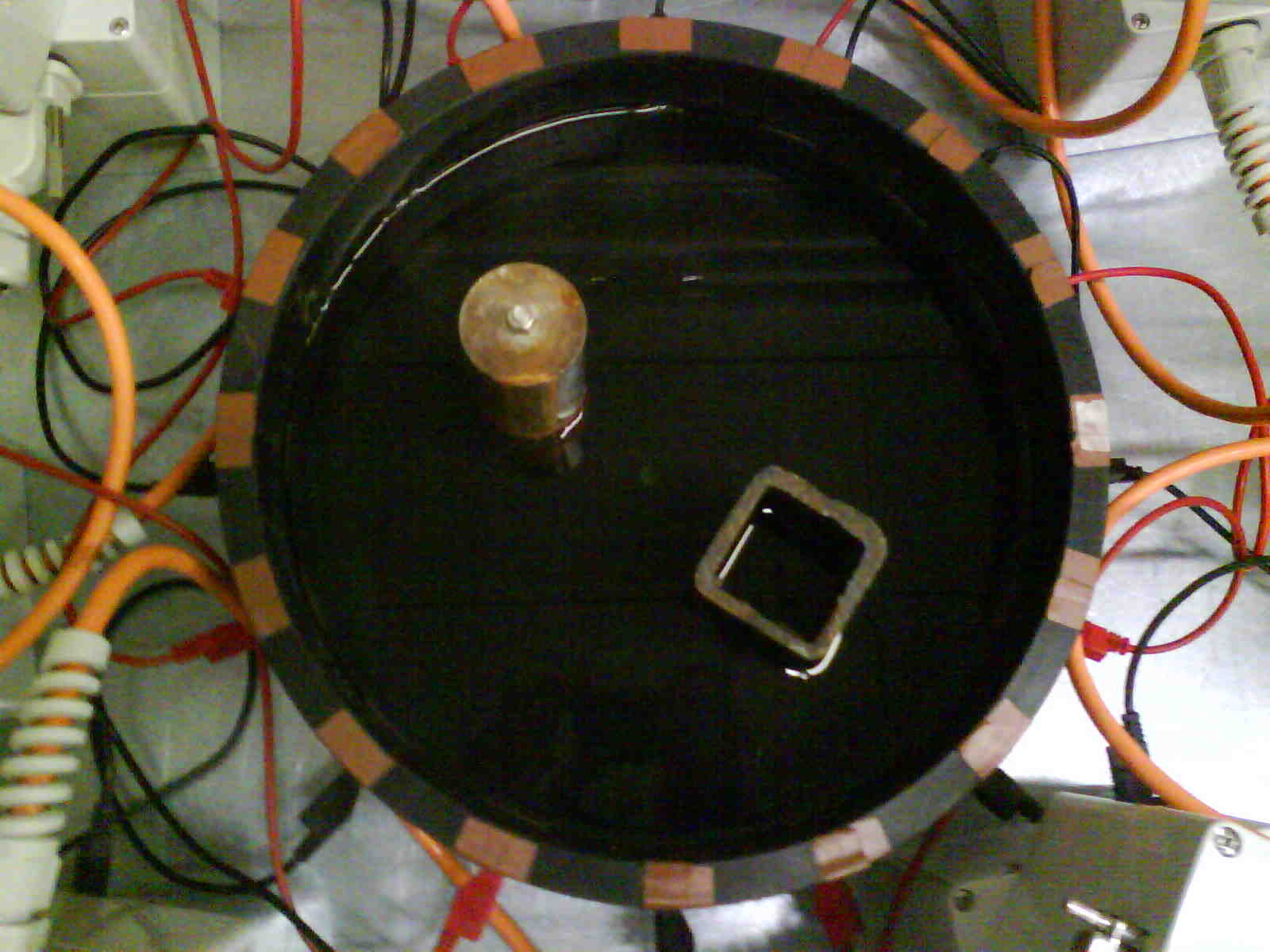}&
\raisebox{0.3cm}{\includegraphics[width=0.3\textwidth]{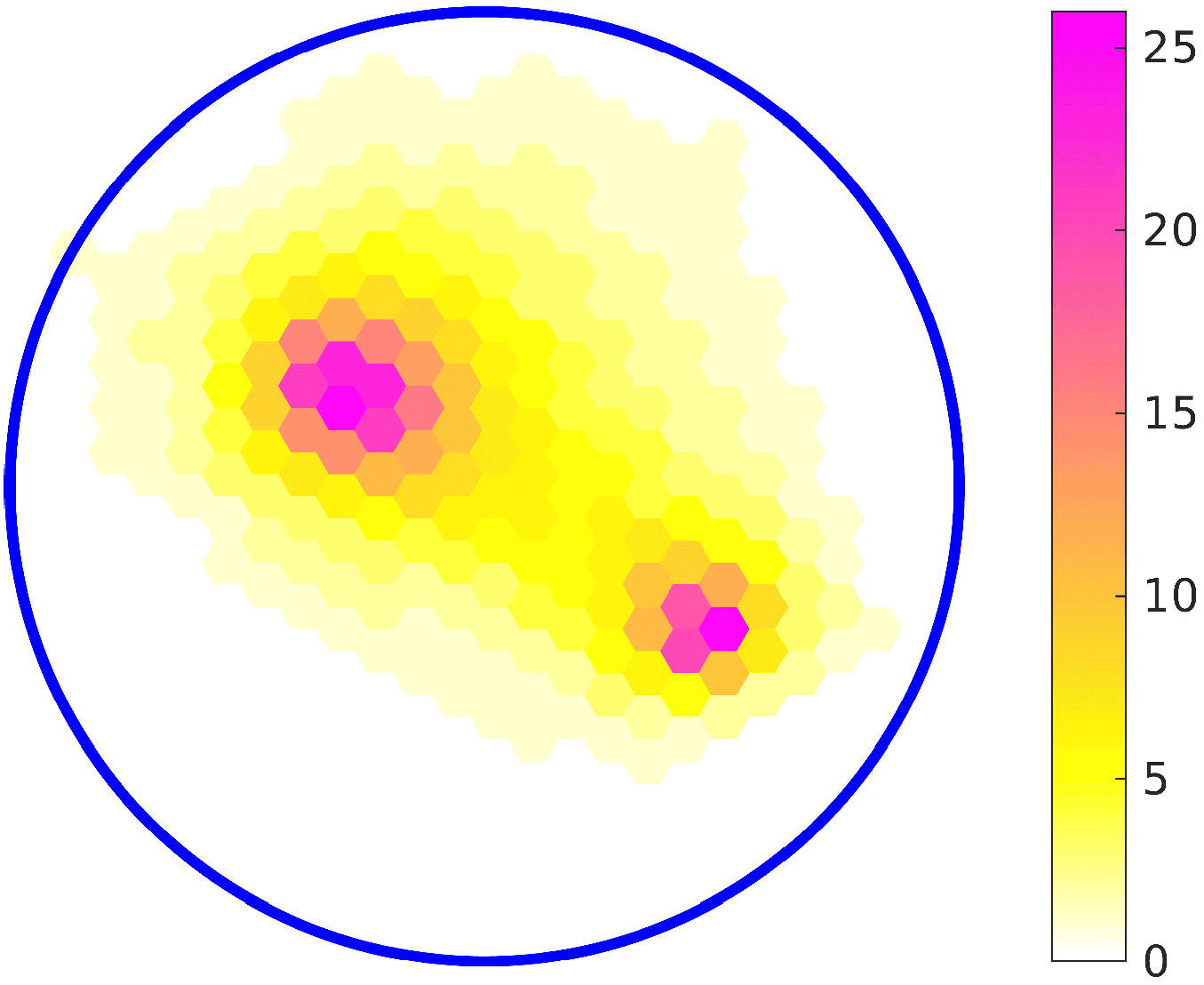}}
\end{tabular}
\end{center}
\caption{\label{fig:real}Reconstructions from water tank measurement data computed using a planar computational model and Algorithm~\ref{alg:2}. The tank is cylindrically symmetric with cross-sectional radius 14 cm. The measurements are done using 16 identical rectangular equispaced electrodes of 2.5 cm width and 7 cm height, and the depth of the water is 7 cm. For more information on the FE mesh, see the beginning of section~\ref{sec:numerical}.}

\begin{center}
\def\arraystretch{1.3}
\begin{tabular}{cc}
\hline
 Parameter & {\bf Reconstruction 2} \\\hline
 ${\rm diam}(B)$ & $0.093$ \\
 $\beta$ & $0.1 + 0.1\N$ \\
 $\mu$ & $1.02$ \\\hline
\end{tabular}
\end{center}
{\bf Tbl. 4}\, \  Parameter values used in the computations; ${\rm diam}(B)$ is the diameter of the hexagons in the hexagonal reconstruction mesh \eqref{eq:chiB}, $\beta$ is the probing scalar(s) in the semidefiniteness test, and $\mu$ is the regularization parameter \eqref{eq:reg-param-num}.
\end{figure}

The results obtained with a planar computational model are presented in Figure~\ref{fig:real}. Both reconstructions were computed using Algorithm~2 with the trial-and-error estimated values 
\[
\gamma_0 = 0.0243 \, {\rm S/m}, \ \ z = 0.005 \, {\rm m^2/S}
\] 
for the background conductivity and contact impedance, respectively.

It was observed that the reconstruction procedure is very sensitive with respect to the values of $\gamma_0$ and $z$ --- a few percent perturbation in their values was enough to ruin the whole reconstruction. This is not surprising as the measurement operator (when interpreted as a function of the conductivity and the contact resistance) satisfies 
\begin{equation}\label{eq:Rc}
\frac{1}{c}R(\gamma_0,z) = R(c\gamma_0,z/c)
\end{equation}
for any constant $c > 0$. By \eqref{eq:Rc}, a few percent error on $\gamma_0$ and $z$ can cause a few percent error on the measurement operator. Due to ill-posedness, such error levels are enough to suppress the signal entirely. However, the results show that with sufficiently good estimates for $\gamma_0$ and $z$, the proposed method can yield reconstructions that are comparable in quality to virtually any existing EIT reconstruction method. 

\end{example}

\begin{example}\label{ex:4}

The final example is a simulated three-dimensional example. The object is a unit ball, and there are $k=32$ approximately equidistantly placed electrodes which all are spherical caps of radius 0.1. An orthonormal current basis $\{I^{(m)}\}_{m=1}^{k-1}$, defined by
\begin{equation}\label{eq:I-3d}
	I^{(m)}_j = \begin{cases}
		\sqrt{\frac{1}{m(m+1)}} & j = 1,2,\dots,m, \\
		-\sqrt{\frac{m}{m+1}} & j=m+1, \\
		0 & j=m+2,m+3,\dots, k,
	\end{cases}
\end{equation}
is used. Note that \eqref{eq:I-3d} is the Gram-Schmidt orthonormalization of the standard $e^{(1)}-e^{(m+1)}$ basis used in many measurement setups, including the setup in Example \ref{ex:3}.

The results are shown in Figure~\ref{fig:three-dim}; for the ease of presentation only Algorithm~1 is considered. It is observed that, with suitable choices of regularization parameters, the algorithm can separate the reconstructed inclusions. We stress that, after precomputing an approximation of $R'(\gamma_0)$ and fixing the $\chi_B$'s, the implementation of the method independent of the spatial dimension. 

\begin{figure}[tbh]
\begin{center}
\begin{tabular}{ll}
\hspace{1.3cm}{\scriptsize {\bf Target}} & \hspace{0.7cm}{\scriptsize {\bf Noiseless~1} (CEM)}\\ 
\includegraphics[width=0.4\textwidth, trim = 4cm 3.5cm 1cm 3cm, clip=true]{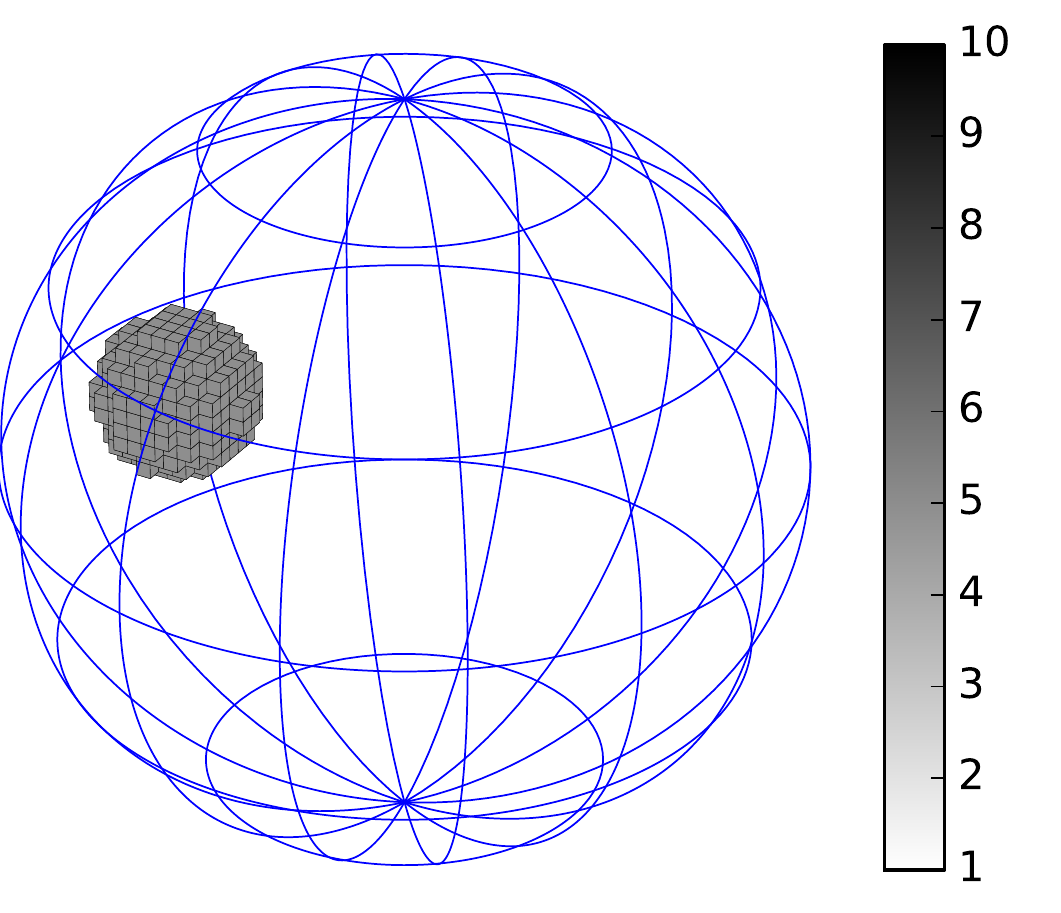}& 
\includegraphics[width=0.4\textwidth, trim = 4cm 3.5cm 1cm 3cm, clip=true]{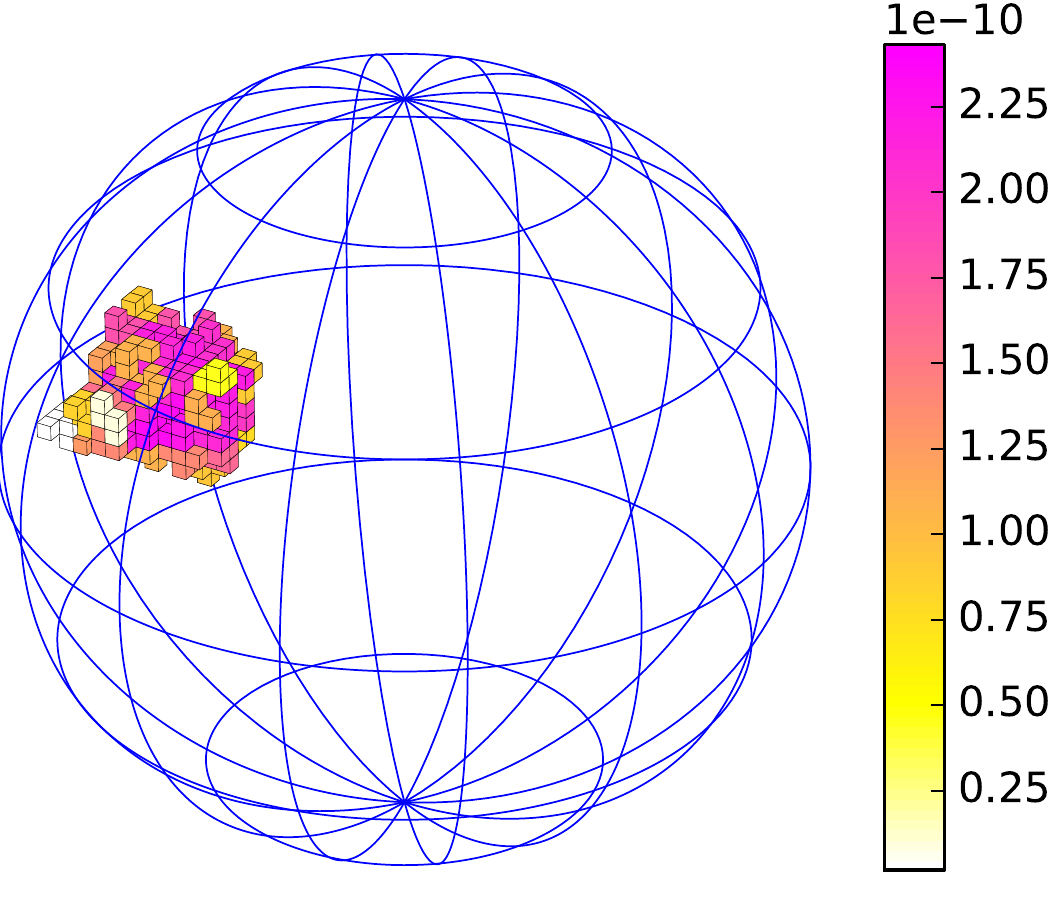}
\\
\includegraphics[width=0.4\textwidth, trim = 4cm 3.5cm 1cm 3cm, clip=true]{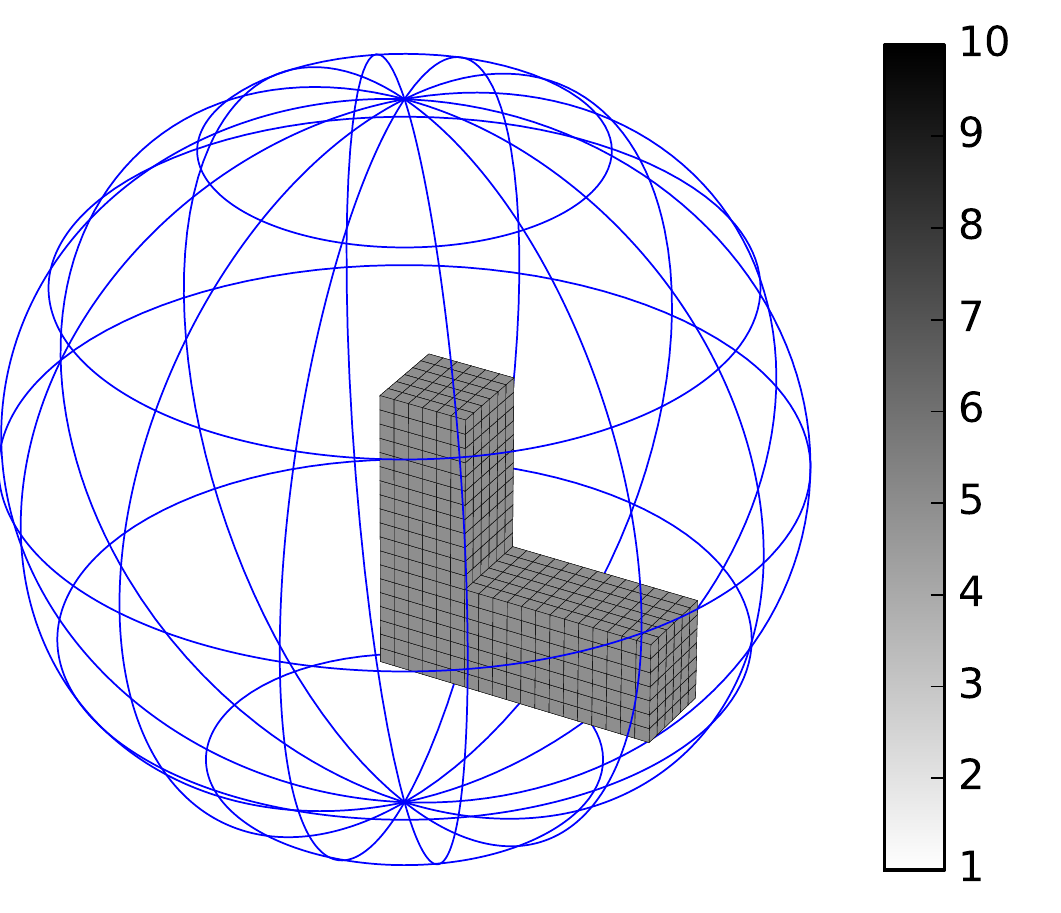} & 
\includegraphics[width=0.4\textwidth, trim = 4cm 3.5cm 1cm 3cm, clip=true]{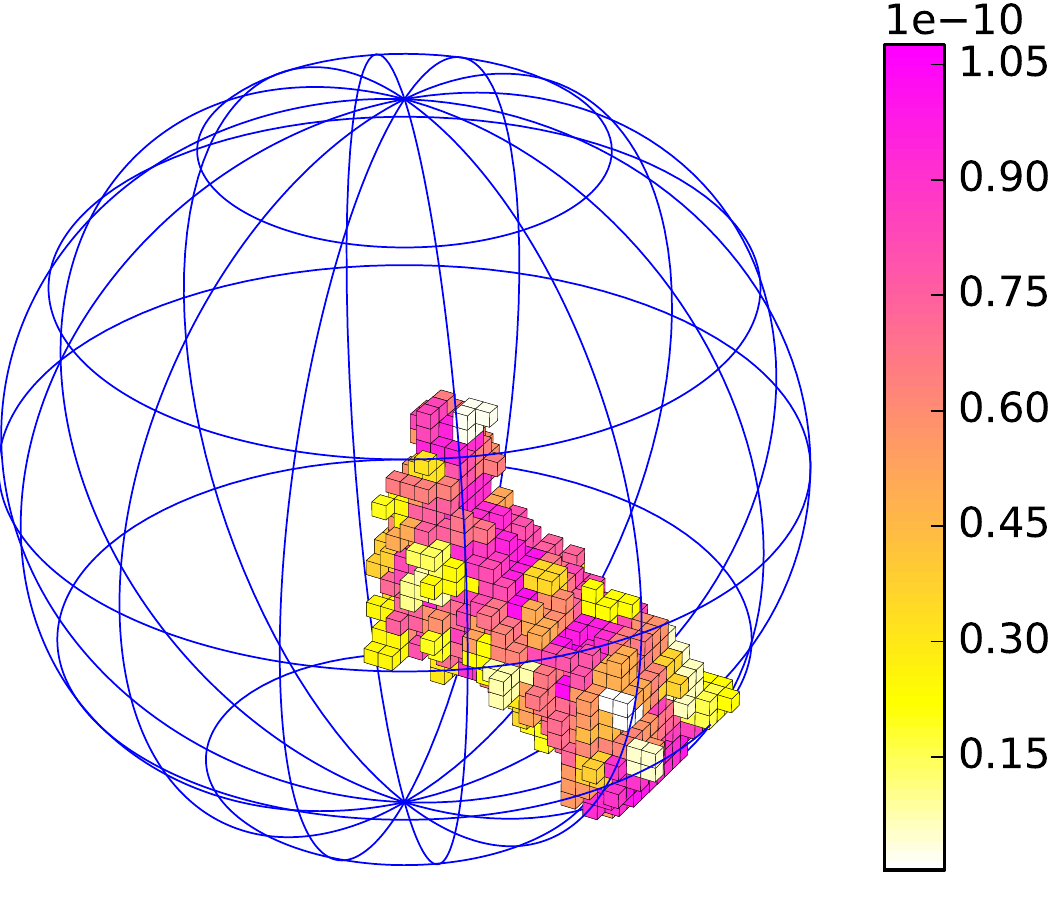}
\\
\includegraphics[width=0.4\textwidth, trim = 4cm 3.5cm 1cm 3cm, clip=true]{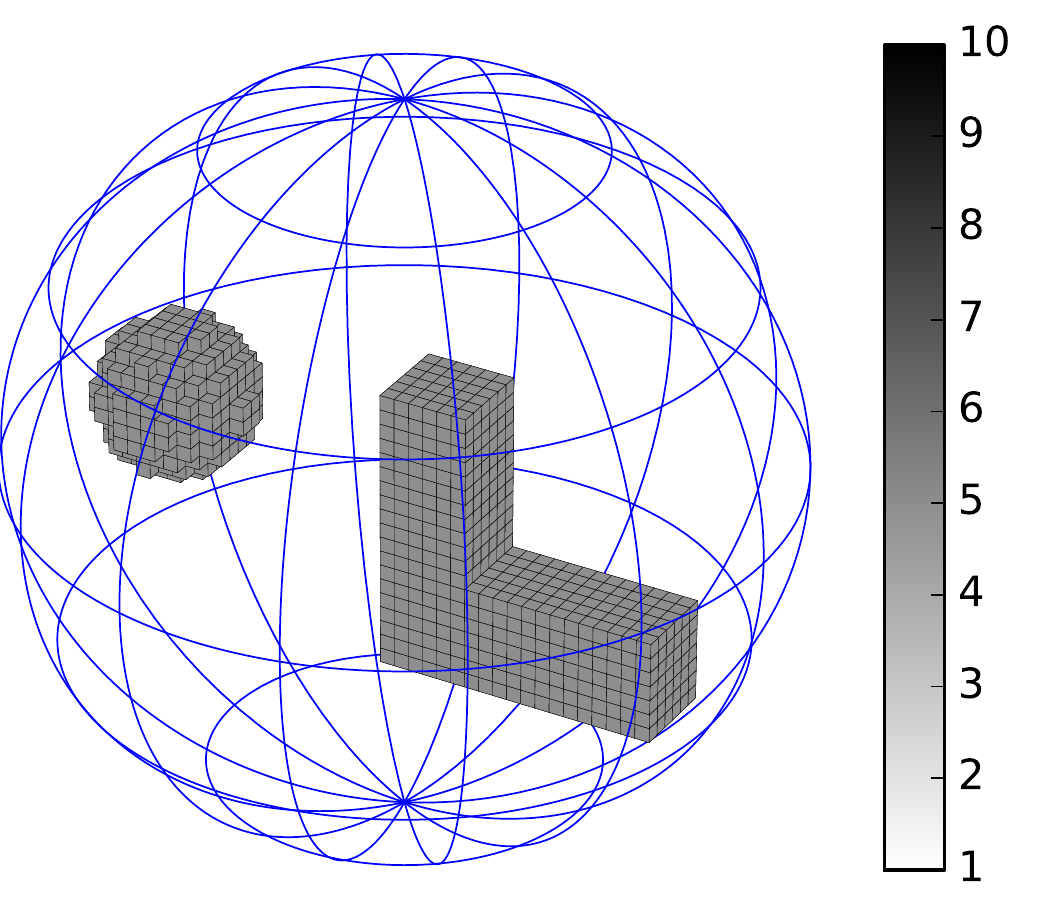} & 
\includegraphics[width=0.4\textwidth, trim = 4cm 3.5cm 1cm 3cm, clip=true]{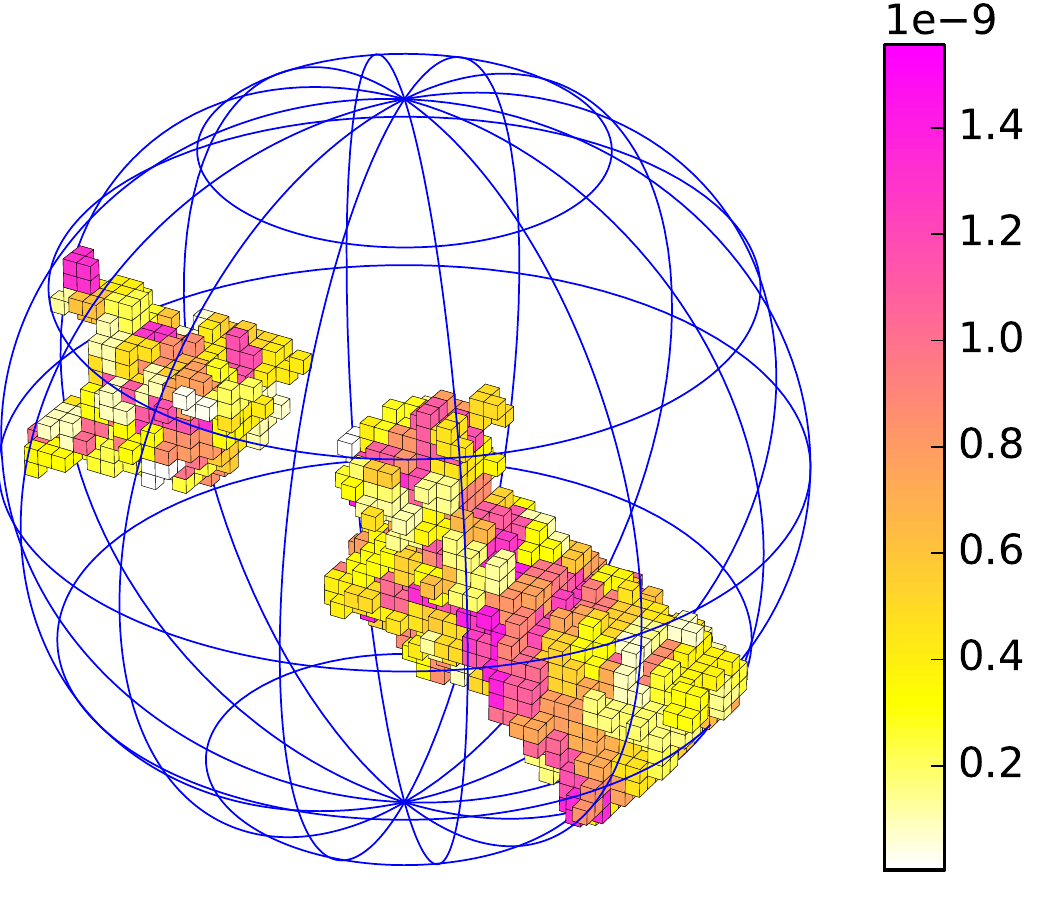}
\end{tabular}
\caption{\label{fig:three-dim} Three-dimensional reconstructions from simulated noise-free measurements with $32$ electrodes. The test sets $B$ are cubes. Reconstructions in the table column ``{\bf Noiseless 1}'' are calculated using Algorithm~1. For more information on the FE mesh, see the beginning of section~\ref{sec:numerical}.}
\vspace{0.4cm}
\hspace{-1.2cm}
\def\arraystretch{1.3}
\begin{tabular}{l}
\\[.07cm]
\hspace{0.54cm}
$
{\rm Top}
\left\{
\begin{array}{c}
\\ 
\\
\\
\end{array}
\right.
$
\\
\hspace{0.14cm}
$
{\rm Middle}
\left\{
\begin{array}{c}
\\ 
\\
\\
\end{array}
\right.
$
\\
\hspace{0.08cm}
$
{\rm Bottom}
\left\{
\begin{array}{c}
\\ 
\\
\\
\end{array}
\right.
$
\end{tabular}
\hspace{-0.4cm}
\begin{tabular}{cc}
\hline
 Parameter & {\bf Noiseless 1} \\\hline
 ${\rm diam}(B)$ & $0.069$ \\
 $\beta$ & $0.8$ \\
 $\mu$ & $1.023$ \\\hdashline[2pt/2pt]
  ${\rm diam}(B)$ & $0.069$ \\
 $\beta$ & $0.8$ \\
 $\mu$ & $1.002$ \\\hdashline[2pt/2pt]
  ${\rm diam}(B)$ & $0.069$ \\
 $\beta$ & $0.8$ \\
 $\mu$ & $0.9966$ \\\hline

\end{tabular}
\end{center}

{\bf Tbl. 5}\, \  Parameter values used in the computations; ${\rm diam}(B)$ is the diameter of the voxels in the reconstruction mesh \eqref{eq:chiB}, $\beta$ is the probing scalar(s) in the semidefiniteness test, and $\mu$ is the regularization parameter \eqref{eq:reg-param-num}.
\end{figure}

\end{example}

\section{Conclusions}\label{sec:conclusions}

We have extended previous works on the regularization analysis as well as the implementation of the monotonicity method. The leading idea of this reconstruction technique is to perform semidefiniteness tests on certain linear combinations of (noisy and discrete) current-to-voltage operators. We have proven that, as a suitably chosen sequence of regularization parameters tends to zero, the approximative test criterion converges uniformly to the idealistic one. Moreover, we rigorously justified the use of the CEM as an approximate model. 

To complement the theoretical study, two reconstruction algorithms were formulated and implemented. Numerical examples were carried out using both simulated CEM data and real-life measurement data. The tests indicate that the monotonicity method can very efficiently provide relatively good images on conductivity inhomogeneities if the homogeneous background conductivity and the electrode contact resistances are sufficiently accurately known. 

\begin{acknowledgements}
Henrik Garde is supported by advanced grant no. 291405 \emph{HD-Tomo} from the European Research Council. Stratos Staboulis is supported by grant no. 4002-00123 \emph{Improved Impedance Tomography with Hybrid Data} from The Danish Council for Independent Research | Natural Sciences. 

The authors are grateful to Professor Jari Kaipio's research group at the University of Eastern Finland (Kuopio) for granting us access to their EIT devices. We thank Marcel Ullrich at University of Stuttgart for his valuable insight on the implementation details of the presented method.
\end{acknowledgements}

\bibliographystyle{spmpsci}      
\bibliography{../mybib}   

\appendix
\section{Appendix: a lemma on the convergence of infima/suprema} \label{append:A}

\begin{lemma}\label{lem:A-1}
Let $J$ be an arbitrary index set and $\{a_j\}_{j\in J}, \{a_j(h)\}_{j\in J} \subset \R$, $h > 0$, be sequences such that $\inf_{j\in J}a_j>-\infty$ and 
\begin{equation*}
\lim_{h\to 0}\sup_{j\in J}|a_j - a_j(h)| = 0. 
\end{equation*}
Denoting $ a := \inf_{j\in J} a_j$ and $a(h) := \inf_{j\in J} a_j(h)$ we have
\begin{equation}\label{eq:A-lim}
\lim_{h\to 0} a(h) = a.
\end{equation}
\end{lemma}
\begin{proof}
Let us first show that the limit in \eqref{eq:A-lim} exists. Given an arbitrary $\veps > 0$, there exists an $h_{\veps} > 0$ such that $\sup_{j\in J}|a_j - a_j(h)| \leq \veps/2$ for all $h \in (0,h_{\veps})$. Let $h, h' \in (0,h_{\veps})$ then $\sup_{j\in J}|a_j(h)-a_j(h')|\leq \veps$, and fix a sequence $\{j(k)\}_{k=1}^\infty \subseteq J$ such that $ a_{j(k)}(h)$ converges to $a(h)$. Hence
\[
a(h') \leq \liminf_{k\to\infty} a_{j(k)}(h') \leq \liminf_{k\to\infty} a_{j(k)}(h) + \veps = a(h) + \veps.
\]
By symmetry with respect to $h$ and $h'$, it follows that $\{a(h)\}_{h>0}$ is a Cauchy sequence. 

It still remains to show that the limit coincides with $a$. For any $\veps > 0$, there exists $j_{\veps}\in J$ and $h_{\veps} > 0$ such that $|a_{j_{\veps}} - a|\leq \veps/2$ and $\sup_{j\in J}|a_j - a_j(h)| \leq \veps/2$ for $h\in(0,h_{\veps})$, respectively. Thus for $h\in(0,h_{\veps})$
\[
a(h) \leq a_{j_{\veps}}(h) \leq a_{j_{\veps}} + \veps/2 \leq a + \veps.
\]
For $h\in(0,h_{\veps})$ pick $j_{\veps}'$ such that $|a_{j_{\veps}'}(h)-a(h)|\leq \veps/2$ then
\[
	a\leq a_{j_{\veps}'} \leq a_{j_{\veps}'}(h) + \veps/2 \leq a(h)+\veps.
\]
Altogether we have shown for any $\veps > 0$ that $|a(h) - a| \leq \veps$ for $h\in (0,h_{\veps})$. \qed
\end{proof}

\section{Appendix: linearization of the CEM and the CM}\label{append:B}
\begin{proposition}\label{prop:fre}
The operators $\Lambda(\gamma) \in \Ls(L_\diamond^2(\partial\Omega))$ and $R(\gamma)\in \Ls(\R_\diamond^k)$ are analytic in $\gamma\in L_+^\infty(\Omega)$. In particular, they are infinitely many times Fr\'echet differentiable. Furthermore, if $\eta$ is compactly supported in $\Omega$, then the boundary value problems
\begin{align}
&
\left\{
\begin{array}{ll}
\label{eq:fre-cm}
\displaystyle{\nabla\cdot(\gamma\nabla u') = -\nabla\cdot(\eta \nabla u ) \quad}  &{\rm in}\;\; \Omega, \\[5pt] 
{\displaystyle{\nu\cdot\gamma\nabla u'} = 0 }\quad &{\rm on}\;\;\partial\Omega,
\end{array} 
\right.
\\[0.1cm]
&
\left\{
\begin{array}{ll}
\label{eq:fre-cem}
\displaystyle{\nabla\cdot(\gamma\nabla v') = -\nabla\cdot(\eta\nabla v) \quad}  &{\rm in}\;\; \Omega, \\[5pt] 
{\displaystyle{\nu\cdot\gamma\nabla v'} = 0 }\quad &{\rm on}\;\;{\partial\Omega}\setminus \bigcup_{j=1}^k \overbar{E_j},\\[5pt] 
{\displaystyle v' + z{\nu\cdot\gamma\nabla v'} = V'_j } \quad &{\rm on}\;\; E_j, \\ 
{\displaystyle \int_{E_j}\nu\cdot\gamma\nabla v'\, dS = 0}, \quad & j=1,2,\ldots k, \\[2pt]
\end{array}
\right.
\end{align} 
uniquely determine the Fr\'echet derivatives via
\[
\Lambda'(\gamma)\eta = u'|_{\partial\Omega}, \quad R'(\gamma)\eta = V',
\]
respectively. Above $u$ and $(v,V)$ are the unique weak solutions of \eqref{eq:cm} and \eqref{eq:cem}, respectively.
\end{proposition}
\begin{proof}
For clarity, we only consider the CEM case as the CM can be handled analogously \cite{Calder'on1980}. Given $(v,V)\in H^1(\Omega)\oplus \R_\diamond^k$ and $\eta\in L_+^\infty(\Omega)$, the variational problem 
\begin{equation}\label{eq:fre-cem-weak}
\int_\Omega \gamma\nabla v' \cdot \nabla w \, dx + \sum_{j=1}^k \int_{E_j}\frac{1}{z}(v'-V_j')(w-W_j)\,dS = -\int_\Omega \eta \nabla v \cdot \nabla w \,dx
\end{equation}
for all $(w,W)\in H^1(\Omega) \oplus \R_\diamond^k$,
is uniquely solvable. Moreover, if $(v,V)$ weakly solves \eqref{eq:cem}, then $V' = R'(\gamma)I$ \cite{Lechleiter2008}. Clearly, if $\eta$ is compactly supported, the right-hand side of \eqref{eq:fre-cem-weak} does not induce any boundary terms and hence $(v',V')$ satisfies \eqref{eq:fre-cem}. 

Define the mapping 
\[
D = D(\eta)\colon \R_\diamond^k \to \R_\diamond^k, \quad V \mapsto V'
\] 
as the solution operator to \eqref{eq:fre-cem-weak}. Consider the expansion
\[
V(\gamma + \eta) = V(\gamma) + \tilde{V}(\eta),
\]
where we denote $V(\gamma) = R(\gamma)I$ and $V(\gamma+\eta) = R(\gamma + \eta)I$. A direct calculation using the variational formulation with the associated internal potentials reveals
\[
\tilde{V}(\eta) = D(\eta)V(\gamma+\eta) = D(\eta)\tilde{V}(\eta) + D(\eta)V(\gamma).
\]
As $\|D(\eta)\| \leq C\|\eta\|_{L^\infty(\Omega)}$, the associated Neumann-series converges for small enough $\eta$. Consequently,
\begin{equation}\label{eq:taylor}
V(\gamma+\eta) = V(\gamma) + \tilde{V}(\eta) = V(\gamma) + (\id - D(\eta))^{-1} D(\eta)V(\gamma) = \sum_{m=0}^\infty D(\eta)^m V(\gamma).
\end{equation}
Different order Fr\'echet derivatives can be inductively derived using \eqref{eq:taylor} and the fact that $D(\eta)$ is linear. \qed


\end{proof}

\end{document}